\documentclass[reqno]{amsart}
\usepackage{amsmath,amssymb,amsthm,color}
\usepackage{mathtools}
\mathtoolsset{showonlyrefs=true}

\usepackage{lineno}

\usepackage{tikz}
\usetikzlibrary{positioning,intersections,patterns}

\numberwithin{equation}{section}

\newtheorem{theorem}{Theorem}[section]
\newtheorem{definition}[theorem]{Definition}

\newtheorem{lemma}[theorem]{Lemma}
\newtheorem{proposition}[theorem]{Proposition}
\newtheorem{remark}[theorem]{Remark}

\DeclareMathOperator{\supp}{supp}
\DeclareMathOperator*{\esssup}{ess\,sup}

\title[Semilinear space-dependent damped wave equation]{
Decay property of solutions to
the wave equation with
space-dependent damping,
absorbing nonlinearity,
and polynomially decaying data
}

\author[Y. Wakasugi]{Yuta Wakasugi}
\address{Laboratory of Mathematics,
Graduate School of Engineering,
Hiroshima University,
Higashi-Hiroshima, 739-8527, Japan}
\email{wakasugi@hiroshima-u.ac.jp}

\date{\today}

\begin{document}

\begin{abstract}
We study the large time behavior of solutions
to the semilinear wave equation with
space-dependent damping and
absorbing nonlinearity in
the whole space or exterior domains.
Our result shows how
the amplitude of the damping coefficient,
the power of the nonlinearity, and
the decay rate of the initial data at the spatial infinity
determine the decay rates of the
energy and the $L^2$-norm of the solution.
In Appendix, we also give a survey of basic results
on the local and global existence of solutions and
the properties of weight functions used in
the energy method.
\end{abstract}



\maketitle


\section{Introduction}
We study the initial-boundary value problem
of the wave equation with
space-dependent damping and 
absorbing nonlinearity
\begin{align}
\label{dwx}
    \left\{ \begin{array}{ll}
        \partial_t^2 u - \Delta u + a(x) \partial_t u + |u|^{p-1}u = 0, &t>0, x \in \Omega, \\
        u(t,x) = 0, &t>0, x \in \partial \Omega,\\
        u(0,x) = u_0(x), \ 
        \partial_t u(0,x) = u_1(x),
        &x \in \Omega.
        \end{array} \right.
\end{align}
Here,
$\Omega = \mathbb{R}^n$
with $n \ge 1$,
or
$\Omega \subset \mathbb{R}^n$
with $n\ge 2$
is an exterior domain,
that is,
$\mathbb{R}^n \setminus \Omega$
is compact.
We also assume that the boundary
$\partial \Omega$
of $\Omega$ is of class $C^2$.
When
$\Omega = \mathbb{R}^n$,
the boundary condition is omitted
and we consider the initial value problem.
The unknown function
$u=u(t,x)$
is assumed to be real-valued.
The function
$a(x)$
denotes the coefficient of the damping term.
Throughout this paper,
we assume that
$a \in C(\mathbb{R}^n)$
is nonnegative and bounded.
The semilinear term
$|u|^{p-1}u$,
where
$p>1$,
is the so-called absorbing nonlinearity,
which assists the decay of the solution.

The aim of this paper is to obtain
the decay estimates of the energy
\begin{align}\label{eq:def:E}
    E[u](t) &:= \frac{1}{2} \int_{\Omega} (|\partial_t u(t,x)|^2 + |\nabla u(t,x)|^2 ) \,dx
    +\frac{1}{p+1} \int_{\Omega} |u(t,x)|^{p+1}\,dx
\end{align}
and the weighted $L^2$-norm
\begin{align}
    \int_{\Omega} a(x) |u(t,x)|^2 \,dx
\end{align}
of the solution.

First, for the energy
$E[u](t)$,
we observe from the equation \eqref{dwx} that
\begin{align}
    \frac{d}{dt} E[u](t)
    &=
    - \int_{\Omega} a(x) |\partial_t u(t,x)|^2 \,dx,
\end{align}
which gives the energy identity
\begin{align}
    E[u](t) + \int_0^t \int_{\Omega} a(x) |\partial_t u(s,x)|^2 \,dxds
    &=
    E[u](0).
\end{align}
Since
$a(x)$
is nonnegative,
the energy
is monotone decreasing in time.
Therefore, a natural question arises as to
whether the energy tends to zero
as time goes to infinity
and, if that is true,
what the actual decay rate is.
Moreover, we can expect
that the amplitude of the damping coefficient
$a(x)$,
the power $p$ of the nonlinearity, and
the spatial decay of the initial data
$(u_0,u_1)$
will play crucial roles for this problem.
Our goal is to clarify how these three factors
determine the decay property of the solution.

Before going to the main result,
we shall review previous studies on
the asymptotic behavior of solutions to
linear and nonlinear damped wave equations.

The study of the asymptotic behavior of solutions
to the damped wave equation
goes back to the pioneering work by
Matsumura \cite{Ma76}.
He studied the initial value problem of the linear
wave equation with the classical damping
\begin{align}\label{dw}
    \left\{ \begin{array}{ll}
        \partial_t^2 u - \Delta u + \partial_t u = 0, &t>0, x \in \mathbb{R}^n, \\
        u(0,x) = u_0(x), \ 
        \partial_t u(0,x) = u_1(x),
        &x \in \mathbb{R}^n.
        \end{array} \right.
\end{align}
In this case the energy of the solution
$u$
is defined by
\begin{align}\label{def:en:lin}
    E_{L}(t) :=
    \frac{1}{2} \int_{\mathbb{R}^n} (|\partial_t u(t,x)|^2 + |\nabla u(t,x)|^2 ) \,dx.
\end{align}
By using the Fourier transform,
he proved the so-called Matsumura estimates
\begin{align}
    \| \partial_t^k \partial_x^{\gamma} u(t) \|_{L^{\infty}}
    &\le
    C (1+t)^{-\frac{n}{2m}-k-\frac{|\gamma|}{2}}
    \left( \| u_0 \|_{L^m} + \| u_1 \|_{L^m}
    + \| u_0 \|_{H^{[\frac{n}{2}]+k+|\gamma| + 1}}
    + \| u_1 \|_{H^{[\frac{n}{2}]+k+|\gamma|}}
    \right),\\
    \| \partial_t^k \partial_x^{\gamma} u(t) \|_{L^2}
    &\le
    C (1+t)^{-\frac{n}{2}\left(\frac{1}{m}-\frac{1}{2} \right)-k-\frac{|\gamma|}{2}}
    \left( \| u_0 \|_{L^m} + \| u_1 \|_{L^m}
    + \| u_0 \|_{H^{k+|\gamma|}}
    + \| u_1 \|_{H^{k+|\gamma|-1}}
    \right)
\label{Matsumura:est}
\end{align}
for $1 \le m \le 2$, $k \in \mathbb{Z}_{\ge 0}$,
and
$\gamma \in \mathbb{Z}_{\ge 0}^n$,
and applied them to semilinear problems.
In particular, the above estimate implies
\begin{align}
    &(1+t)E_L(t) + \| u(t) \|_{L^2}^2 \\
    &\le
    C (1+t)^{-n\left(\frac{1}{m}-\frac{1}{2}\right)}
    \left( \| u_0 \|_{L^m} + \| u_1 \|_{L^m}
    + \| u_0 \|_{H^1}
    + \| u_1 \|_{L^2}
    \right)^2.
\label{eq:est:en:Mat}
\end{align}
This indicates that the spatial decay of
the initial data improves the time decay of the solution.

Moreover, the decay rate in the
estimates \eqref{Matsumura:est}
suggests that the solution of \eqref{dw}
is approximated by a solution of the corresponding heat equation
\begin{align}
    \partial_t v - \Delta v = 0,
    \quad t>0,x\in \mathbb{R}^n.
\end{align}
This is the so-called
diffusion phenomenon
and firstly proved by
Hsiao and Liu \cite{HsLi92}
for the hyperbolic conservation law with
damping.

There are many improvements and generalizations
of the Matsumura estimates and
the diffusion phenomenon for \eqref{dw}.
We refer the reader to
\cite{ChHa03,HoOg04,IkeInWa17,IkeInOkWa19,Ik02,IkNi03,Ka00,KaNaOn95,MaNi03,Mi21,Na04,Ni03MathZ,SaWa17,So20MA,Ta15,YaMi00}
and the references therein.

Next, we consider
the initial boundary value problem of the
linear wave equation
with space-dependent damping
\begin{align}\label{dwx:lin}
    \left\{ \begin{array}{ll}
        \partial_t^2 u - \Delta u + a(x)\partial_t u = 0, &t>0, x \in \Omega, \\
        u(t,x) = 0,
        &t>0, x\in \partial \Omega,\\
        u(0,x) = u_0(x), \ 
        \partial_t u(0,x) = u_1(x),
        &x \in \Omega.
        \end{array} \right.
\end{align}
Mochizuki \cite{Mo76} firstly
studied the case
$\Omega = \mathbb{R}^n \ (n\neq 2)$
and showed 
that if
$a(x) \le C\langle x \rangle^{-\alpha}$
with $\alpha > 1$,
then the wave operator exists and is not
identically vanishing.
Namely, the energy
$E_L(t)$ defined by \eqref{def:en:lin}
of the solution does not decay to zero
in general, and the solution behaves like
a solution of the wave equation without damping.
This means that
if the damping is sufficiently small at the
spatial infinity,
then the energy of the solution
does not decay to zero in general.
His result actually includes
the time and space dependent damping,
and generalizations in the damping coefficients
and domains can be found in
Mochizuki and Nakazawa \cite{MoNa96},
Matsuyama \cite{Mat02},
and
Ueda \cite{Ue16}.

On the other hand,
for \eqref{dwx:lin} with
$\Omega = \mathbb{R}^n$,
from the result by
Matsumura \cite{Ma77},
we see that if
$u_0, u_1 \in C_0^{\infty}(\mathbb{R}^n)$
and
$a(x) \ge C \langle x \rangle^{-1}$,
then
$E_L(t)$
decays to
zero as $t \to \infty$
(see also Uesaka \cite{Ue79}).
These results indicate that
for the damping coefficient
$a(x) = \langle x \rangle^{-\alpha}$,
the value $\alpha = 1$
is critical for the energy decay or non-decay.

Regarding the precise decay rate of the solution to
\eqref{dwx:lin},
Todorova and Yordanov \cite{ToYo09}
proved that if
$\Omega = \mathbb{R}^n$,
$a(x)$ is positive, radial and satisfies
$a(x) = a_0 |x|^{-\alpha} + o(|x|^{-\alpha})
\ (|x|\to \infty)$
with some $\alpha \in [0,1)$,
and the initial data has compact support,
then the solution satisfies
\begin{align}
    (1+t) E_L(t) + \int_{\mathbb{R}^n} a(x) |u(t,x)|^2 \,dx
    &\le
    C (1+t)^{-\frac{n-\alpha}{2-\alpha}+\delta}
    \left(
    \| u_0 \|_{H^1} + \| u_1 \|_{L^2}
    \right)^2,
\end{align}
where
$\delta > 0$ is arbitrary constant and
$C$ depends on $\delta$ and the support of the data.
We note that if we formally take
$\alpha = 0$ and $\delta = 0$,
then the decay rate coincides with that of
\eqref{eq:est:en:Mat}.
The proof of \cite{ToYo09} is based on
the weighted energy method with the weight function
\begin{align}
    t^{-\frac{n-\alpha}{2-\alpha}+2\delta}
    \exp \left(
    - \left( \frac{n-\alpha}{2-\alpha} - \delta \right) \frac{A(x)}{t}
    \right),
\end{align}
where $A(x)$ is a solution of the Poisson equation
$\Delta A(x) = a(x)$.
Such weight functions were firstly introduced by
Ikehata and Tanizawa \cite{IkTa05} and Ikehata \cite{Ik05IJPAM} for damped wave equations.
Some generalizations of the principal part to
variable coefficients were made by
Radu, Todorova, and Yordanov \cite{RaToYo09, RaToYo10}.
The assumption of the radial symmetry of
$a(x)$ was relaxed by
Sobajima and the author \cite{SoWa17_AIMS}.
Moreover, in \cite{SoWa19_CCM, SoWa21_JMSJ},
the compactness assumption on the support of 
the initial data was removed and
polynomially decaying data were treated.
The point is the use of
a suitable supersolution of the corresponding
heat equation
\begin{align}
    a(x) \partial_t v - \Delta v = 0
\end{align}
having polynomial order in the far field.
This approach is also a main tool in this paper.
For the diffusion phenomenon,
we refer the reader to
\cite{JoRo18, Nis16, RaToYo11, RaToYo16, SoWa16_JDE, SoWa18_ADE,Wa14JHDE}.

When the damping coefficient is critical
for the energy decay,
the situation becomes more delicate.
Ikehata, Todorova, and Yordanov \cite{IkToYo13}
studied \eqref{dwx:lin} in the case where
$\Omega = \mathbb{R}^n \ (n\ge 3)$,
$a(x)$
satisfies
$a_0 \langle x \rangle^{-1}
\le a(x) \le a_1 \langle x \rangle^{-1}$
with some
$a_0, a_1 > 0$,
and the initial data has compact support.
They obtained the decay estimates
\begin{align}
    E_L(t) &= \begin{cases}
    O(t^{-a_0}) &(1<a_0<n),\\
    O(t^{-n+\delta}) &(a_0 \ge n)
    \end{cases}
\end{align}
as $t \to \infty$
with arbitrary small $\delta > 0$.
This indicates that the decay rate depends on
the constant $a_0$.
Similar results in the lower dimensional cases
and the optimality of the above estimates
under additional assumptions
were also obtained in \cite{IkToYo13}.

We also mention that
$a(x)$ is not necessarily positive everywhere.
It is known
that the so-called geometric control condition
(GCC)
introduced by
Rauch and Taylor \cite{RauTa74}
and Bardos, Lebeau, and Rauch \cite{BaLeRa92}
is sufficient for the energy decay
of solutions with initial data in the energy space.
For the problem \eqref{dwx:lin} with
$\Omega = \mathbb{R}^n$,
(GCC) is read as follows:
There exist constants
$T > 0$ and $c > 0$ such that
for any
$(x_0, \xi_0) \in \mathbb{R}^n \times \mathbb{S}^{n-1}$,
we have
\begin{align}
    \frac{1}{T} \int_0^T a(x_0+s\xi_0) \,ds \ge c.
\end{align}
For this and related topics,
we refer the reader to
\cite{AlIbKh15, BuJo16, Da16, Ik03JDE, KaNaSo04, Nakao01MathZ, Nis09, Nis16, Zu90}.
We note that for
$a(x) = \langle x \rangle^{-\alpha}$
with $\alpha > 0$,
(GCC) is not fulfilled.

We note that
for the linear wave equation with time-dependent damping
\begin{align}
    \partial_t^2 u - \Delta u + b(t) \partial_t u = 0,
\end{align}
the asymptotic behavior of the solution
can be classified depending on
the behavior of $b(t)$.
See
\cite{Wi04, Wid, Wi06, Wi07JDE, Wi07ADE, Ya06}.

Thirdly, we consider the semilinear problem
\begin{align}\label{dw:nonlin}
    \left\{ \begin{array}{ll}
        \partial_t^2 u - \Delta u + \partial_t u = f(u), &t>0, x \in \Omega, \\
        u(t,x) = 0,
        &t>0, x\in \partial \Omega,\\
        u(0,x) = u_0(x), \ 
        \partial_t u(0,x) = u_1(x),
        &x \in \Omega.
        \end{array} \right.
\end{align}
When
$f(u) = |u|^{p-1}u$ or $\pm |u|^p$
with $p>1$,
the nonlinearity works as a sourcing term
and it may cause the singularity of the solution
in a finite time.
In this case, it is known that
there exists the critical exponent
$p_F(n) = 1+ \frac{2}{n}$,
that is,
if
$p > p_F(n)$,
then \eqref{dw:nonlin}
admits the global solution for small initial data;
if
$p < p_F(n)$,
then the solution may blow up in finite time
even for the small initial data.
The number
$p_F(n)$
is the so-called Fujita critical exponent
named after the pioneering work by
Fujita \cite{Fu66} for the semilinear heat equation.

When $\Omega = \mathbb{R}^n$
and $f(u) = \pm |u|^p$,
Todorova and Yordanov \cite{ToYo01}
determined the critical exponent
for compactly supported initial data.
Later on,
Zhang \cite{Zh01} and 
Kirane and Qafsaoui \cite{KiQa02} proved that
the critical case
$p = p_F(n)$
belongs to the blow-up case.

There are many improvements and related studies
to the results above.
The compactness assumption of the support of the
initial data were removed by
\cite{HaKaNa04DIE, IkeInWa17, IkeInOkWa19, IkTa05, NaNi08}.
The diffusion phenomenon for the global solution
was proved by
\cite{GaRa98, HaKaNa04DIE, KaUe13, KaTa16}.
The case where
$\Omega$
is the half space or the exterior domain
was studied by
\cite{IkeSo19, IkeTaWa, Ik03JMAA, Ik05JMAA1, OgTa09, On03JMAA, So19DIE}
Also, estimates of lifespan for blowing-up solutions were obtained by
\cite{LaZh19, LiZh95, Ni03Ib, IkeWa15, IkeOg16, IkeSo19, IkeSo19JMAA}.

When
$f(u) = |u|^{p-1}u$,
the global existence part
can be proved completely the same way
as in the case $f(u) = \pm |u|^p$.
However, regarding the blow-up of solutions,
the same proof as before
works only for $n \le 3$,
since the fundamental solution of
the linear damped wave equation
is not positive for $n \ge 4$,
which follows from the explicit formula
of the linear wave equation
(see e.g., \cite[p.1011]{SaWa17}).
Ikehata and Ohta \cite{IkOh02}
obtained the blow-up of solutions
for the subcritical case
$p < p_F(n)$.
The critical case
$p = p_F(n)$
with $n \ge 4$
seems to remain open.

When $f(u) = -|u|^{p-1}u$
with $p > 1$,
the nonlinearity works as an absorbing term.
In this case with
$\Omega = \mathbb{R}^n$,
Kawashima, Nakao, and Ono \cite{KaNaOn95}
proved the large data global existence.
Moreover, decay estimates of solutions were obtained for
$p > 1+\frac{4}{n}$.
Later on,
Nishihara and Zhao \cite{NiZh06} and
Ikahata, Nishihara, and Zhao \cite{IkNiZh06}
studied the case
$1 < p \le 1 + \frac{4}{n}$.
From their results,
we have the energy estimate
\begin{align}
    (1+t) E[u](t)
    + \| u(t) \|_{L^2}^2
    &\le
    C(I_0) (1+t)^{-2\left(\frac{1}{p-1}-\frac{n}{4} \right)},
\label{eq:est:IkNiZh}
\end{align}
where
\begin{align}
    I_0 :=
    \int_{\mathbb{R}^n}
    \left( |u_1(x)|^2 + |\nabla u_0(x)|^2
    + |u_0(x)|^{p+1} + |u_0(x)|^2 \right)
    \langle x \rangle^{2m} \,dx,
    \quad
    m > 2\left( \frac{1}{p-1} - \frac{n}{4} \right)
\end{align}
and we recall that
$E[u](t)$
is defined by \eqref{eq:def:E}.
Also, the asymptotic behavior was discussed by
\cite{Ka00, Ham10, HaKaNa06, HaNa17JMAA, IkNiZh06, Ni06}.
There seems no result for
exterior domain cases.

Finally, we consider the semilinear problem
with space-dependent damping which is
slightly more general than \eqref{dwx}:
\begin{align}\label{dwx:nonlin}
    \left\{ \begin{array}{ll}
        \partial_t^2 u - \Delta u + a(x)\partial_t u = f(u), &t>0, x \in \Omega, \\
        u(t,x) = 0,
        &t>0, x\in \partial \Omega,\\
        u(0,x) = u_0(x), \ 
        \partial_t u(0,x) = u_1(x),
        &x \in \Omega.
        \end{array} \right.
\end{align}
When the nonlinearity works as a sourcing term,
we expect that there is the critical exponent
as in the case $a(x) \equiv 1$.
Indeed, in the case where
$\Omega = \mathbb{R}^n$,
$f(u) = \pm |u|^{p}$,
the initial data has compact support,
and
$a(x)$ is positive, radial, and satisfies
$a(x) = a_0 | x |^{-\alpha} + o(|x|^{-\alpha})
\ (|x| \to \infty)$
with
$\alpha \in [0,1)$,
Ikehata, Todorova, and Yordanov \cite{IkToYo09} determined the
critical exponent as
$p_F(n-\alpha) = 1+\frac{2}{n-\alpha}$.
The estimate of lifespan for blowing-up solutions was obtained in
\cite{IkeSo19, IkeWa15}.
The blow-up of solutions for the case
$f(u) = |u|^{p-1}u$
seems to be an open problem.

Recently, Sobajima \cite{So_pre}
studied the critical damping case
$a(x) = a_0 |x|^{-1}$
in an exterior domain
$\Omega$
with
$n \ge 3$,
and proved the small data global existence
of solutions under the conditions
$a_0 > n-2$
and
$p > 1+ \frac{4}{n-2+\min\{n,a_0\}}$.
The blow-up part was investigated by
\cite{IkeSo21FE, Li15, So_pre}.
In particular, when
$\Omega$
is the outside a ball with
$n\ge 3$,
$a_0 \ge n$,
and
$f(u) = \pm |u|^p$,
the critical exponent is determined as
$p = p_F(n-1)$.
Moreover, in Ikeda and Sobajima \cite{IkeSo21FE},
the blow-up of solutions was obtained for
$\Omega = \mathbb{R}^n \ (n\ge 3)$,
$0\le a_0 <\frac{(n-1)^2}{n+1}$,
$f(u) = \pm |u|^p$ with
$\frac{n}{n-1} < p \le p_S(n+a_0)$,
where
$p_S(n)$
is the positive root of the quadratic equation
\begin{align}
    2 + (n+1)p - (n-1) p^2 = 0
\end{align}
and is the so-called Strauss exponent.
We remark that
$p_S(n+a_0) > p_F(n-1)$ holds if
$a_0 <\frac{(n-1)^2}{n+1}$.
From this, we can expect that the critical exponent
changes depending on the value $a_0$.

For the absorbing nonlinear term
$f(u) = - |u|^{p-1}u$
in the whole space case $\Omega = \mathbb{R}^n$
was studied by
Todorova and Yordanov \cite{ToYo07}
and Nishihara \cite{Ni10}.
In \cite{Ni10},
for compactly supported initial data,
the following two results were proved:

(i) If $a(x) = a_0 \langle x \rangle^{-\alpha}$
with some $a_0 > 0$ and $\alpha \in [0,1)$,
then we have
\begin{align}
    (1+t)E[u](t) + \int_{\mathbb{R}^n} a(x) |u(t,x)|^2 \,dx
    &\le
    C (1+t)^{-\frac{n-\alpha}{2-\alpha}+\delta}
\end{align}
with arbitrary small $\delta > 0$;

(ii) If
$a_0 \langle x \rangle^{-\alpha} \le a(x)
\le a_1 \langle x \rangle^{-\alpha}$
with some $a_0, a_1 > 0$ and $\alpha \in [0,1)$,
then we have
\begin{align}
    (1+t)E[u](t) + \int_{\mathbb{R}^n} a(x) |u(t,x)|^2 \,dx
    &\le
    C
    \begin{dcases}
    (1+t)^{-\frac{4}{2-\alpha}\left(
    \frac{1}{p-1} - \frac{n-\alpha}{4} \right)}
    &(p > p_{subc}(n,\alpha)),\\
    (1+t)^{-\frac{2}{p-1}} \log(2+t)
    &(p = p_{subc}(n,\alpha)),\\
    (1+t)^{-\frac{2}{p-1}}
    &(p<p_{subc}(n,\alpha)),
    \end{dcases}
\end{align}
where
\begin{align}\label{eq:def:psubc}
    p_{subc}(n,\alpha) := 1+\frac{2\alpha}{n-\alpha}.
\end{align}
We note that the decay rate in
(i) is the same as that of the linear problem
\eqref{dwx:lin}
and it is better than that of (ii) if
$p > p_F(n-\alpha)$.
This means
$p_F(n-\alpha)$ is critical in the sense
of the effect of the nonlinearity to
the decay rate of the energy.
Moreover, (ii) shows that
the second critical exponent
$p_{subc}(n,\alpha)$
appears and it
divides the decay rate of the 
energy.
We also note that the estimate
for the case
$p > p_{subc}(n,\alpha)$
corresponds to the estimate
\eqref{eq:est:IkNiZh}.
Thus, we may interpret the situation in the
following way:
When the damping is weak in the sense of
$a(x) \sim \langle x \rangle^{-\alpha}$
with
$\alpha \in (0,1)$,
we cannot obtain the same type energy estimate as in \eqref{eq:est:IkNiZh}
for all
$p > 1$,
and the decay rate becomes worse under or on
the second critical exponent
$p_{subc}(n,\alpha)$.
Our main goal in this paper is to give a generalization of
the results (i) and (ii) above.

In recent years,
semilinear wave equations with
time-dependent damping
have been intensively studied.
For the progress of this problem,
we refer the reader to
Sections 1 and 2 in
Lai, Schiavone, and Takamura \cite{LaScTa20}.
We also refer to
\cite{NiSoWa18} and the references therein
for a recent study of
semilinear wave equations with
time and space dependent damping.


To state our results, we define the solution.
\begin{definition}[Mild and strong solutions]\label{def:sol}
Let
$\mathcal{A}$
be the operator
\begin{align}
    \mathcal{A} =
    \begin{pmatrix} 0 & 1 \\ \Delta & -a(x) \end{pmatrix}
\end{align}
defined on
$\mathcal{H} := H_0^1(\Omega) \times L^2(\Omega)$
with the domain
$D(\mathcal{A}) = (H^2(\Omega) \cap H^1_0(\Omega)) \times H^1_0(\Omega)$.
Let
$U(t)$
denote the $C_0$-semigroup generated by
$\mathcal{A}$.
Let
$(u_0, u_1) \in \mathcal{H}$
and
$T \in (0,\infty]$.
A function
\begin{align}\label{eq:mildsol}
    u \in C([0,T); H^1_0(\Omega)) \cap C^1([0,T); L^2(\Omega))
\end{align}
is called a mild solution of \eqref{dwx} on
$[0,T)$
if
$\mathcal{U} = {}^t (u, \partial_t u)$
satisfies the integral equation
\begin{align}
    \mathcal{U}(t) = U(t)
    \begin{pmatrix} u_0 \\ u_1 \end{pmatrix}
    + \int_0^t U(t-s)
    \begin{pmatrix}
    0\\ - |u|^{p-1}u
    \end{pmatrix}
    \,ds
\end{align}
in
$C([0,T); \mathcal{H})$.
Moreover, when
$(u_0,u_1) \in D(\mathcal{A})$,
a function
\begin{align}
    u \in C([0,T); H^2(\Omega)) \cap C^1([0,T); H^1_0(\Omega)) \cap C^2([0,T); L^2(\Omega))
\end{align}
is said to be a strong solution of \eqref{dwx} on $[0,T)$ if
$u$
satisfies the equation of \eqref{dwx} in
$C([0,T); L^2(\Omega))$.
If $T = \infty$, we call $u$ a
global (mild or strong) solution.
\end{definition}

First, we prepare the existence and regularity of
the global solution.
\begin{proposition}\label{prop:ex}
Let
$\Omega = \mathbb{R}^n$
with $n \ge 1$,
or
$\Omega \subset \mathbb{R}^n$
with $n\ge 2$
be an exterior domain with $C^2$-boundary.
Let
$a(x) \in C(\mathbb{R}^n)$
be nonnegative and bounded.
Let
\begin{align}\label{p}
    1 < p < \infty \ (n=1,2),\quad
    1 < p \le \frac{n}{n-2} \ (n \ge 3),
\end{align}
and let
$(u_0, u_1) \in H^1_0(\Omega) \times L^2(\Omega)$.
Then, there exists a unique global mild solution
$u$ to \eqref{dwx}.
If we further assume
$(u_0, u_1) \in (H^2(\Omega) \cap H^1_0(\Omega)) \times H^1_0(\Omega)$,
then
$u$
becomes a strong solution to \eqref{dwx}.
\end{proposition}
\begin{remark}
The assumption
$\partial \Omega \in C^2$
is used to ensure
$D(\mathcal{A}) = (H^2(\Omega) \cap H^1_0(\Omega)) \times H^1_0(\Omega)$
(see Cazenave and Haraux \cite[Remark 2.6.3]{CaHa}
and Brezis \cite[Theorem 9.25]{Br}).
The restriction of the range of $p$
in \eqref{p}
is due to the use of
Gagliardo--Nirenberg inequality
(see Section \ref{sec:A2}).
\end{remark}

The proof of Proposition \ref{prop:ex} is standard.
However, for reader's convenience,
we will give an outline of the proof
in the appendix.

To state our result,
we recall that
$E[u](t)$
and
$p_{subc}(n,\alpha)$
are defined by
\eqref{eq:def:E} and \eqref{eq:def:psubc},
respectively.
The main result of this paper reads as follows.
\begin{theorem}\label{thm1}
Let
$\Omega = \mathbb{R}^n$
with $n \ge 1$
or
$\Omega \subset \mathbb{R}^n$
with $n\ge 2$
be an exterior domain with $C^2$-boundary.
Let
$p$ satisfy \eqref{p}
and
$(u_0, u_1) \in H^1_0(\Omega) \times L^2(\Omega)$,
and let
$u$
be the corresponding global mild solution of \eqref{dwx}.
Then, the followings hold.
\begin{itemize}
\item[(i)]
Assume that
$a \in C(\mathbb{R}^n)$
is positive and satisfies
\begin{align}\label{assum:a}
    \lim_{|x|\to \infty} |x|^{\alpha} a(x) = a_0
\end{align}
with some constants
$\alpha \in [0,1)$
and
$a_0 > 0$.
Moreover, we assume that the initial data satisfy
\begin{align}
    &I_0[u_0,u_1] \\
    &:= 
    \int_{\Omega} \left[
    ( |u_1(x)|^2 + |\nabla u_0(x)|^2 + |u_0(x)|^{p+1} ) \langle x \rangle^{\alpha} + |u_0(x)|^2 \langle x \rangle^{-\alpha}
    \right] \langle x \rangle^{\lambda(2-\alpha)} \,dx \\
    &< \infty
\label{assum:ini}
\end{align}
with some
$\lambda \in [0,\frac{n-\alpha}{2-\alpha})$.
Then, we have
\begin{align}
    (1+t)E[u](t) + \int_{\Omega} a(x) |u(t,x)|^2 \,dx
    &\le
    C I_0[u_0,u_1] (1+t)^{-\lambda}
\end{align}
for $t\ge 0$
with some constant
$C = C(n,a,p,\lambda) > 0$.
\item[(ii)]
Assume that
$a \in C(\mathbb{R}^n)$
is positive and satisfies
\begin{align}
    a_0 \langle x \rangle^{-\alpha}
    \le a(x)
    \le a_1 \langle x \rangle^{-\alpha}
\end{align}
with some constants
$\alpha \in [0,1)$,
$a_0, a_1 > 0$.
Moreover, we assume that the initial data satisfy the condition
$I_0[u_0, u_1] < \infty$
with some
$\lambda \in [0,\infty)$,
where
$I_0[u_0,u_1]$
is defined by \eqref{assum:ini}.
Then, we have
\begin{align}
    &(1+t)E[u](t) + \int_{\Omega} a(x) |u(t,x)|^2 \,dx \\
    &\le C (I_0[u_0,u_1] + 1) \\
    &\quad \times
    \begin{cases}
    (1+t)^{-\lambda}
        &(\lambda < \min\{ \frac{4}{2-\alpha}(\frac{1}{p-1} - \frac{n-\alpha}{4} ), \frac{2}{p-1} \}),\\
    (1+t)^{-\lambda} \log (2+t)
        &(\lambda = \min\{ \frac{4}{2-\alpha}(\frac{1}{p-1} - \frac{n-\alpha}{4} ), \frac{2}{p-1} \}, \ p \neq p_{subc}(n,\alpha)),\\
    (1+t)^{-\lambda} (\log (2+t))^2
        &(\lambda = \frac{4}{2-\alpha}(\frac{1}{p-1} - \frac{n-\alpha}{4} ) = \frac{2}{p-1}, \ \mathrm{i.e.,} \ p = p_{subc}(n,\alpha)),\\
    (1+t)^{-\frac{4}{2-\alpha} ( \frac{1}{p-1} - \frac{n-\alpha}{4})}
        &(\lambda > \frac{4}{2-\alpha}(\frac{1}{p-1} - \frac{n-\alpha}{4} ),
        \ p > p_{subc}(n,\alpha)),\\
    (1+t)^{-\frac{2}{p-1}} \log (2+t)
        &(\lambda > \frac{2}{p-1},
        \ p = p_{subc}(n,\alpha)),\\
    (1+t)^{-\frac{2}{p-1}}
        &(\lambda > \frac{2}{p-1}, 
        \ p<p_{subc}(n,\alpha))
    \end{cases}
\end{align}
for $t \ge 0$
with some constant 
$C = C(n,a,p,\lambda) > 0$.
\end{itemize}
\end{theorem}
\begin{remark}
Under the assumptions of (i),
the both conclusions of (i) and (ii) are true.
In
Figure \ref{fig1},
the decay rates of
$\displaystyle \int_{\Omega} a(x) |u(t,x)|^2 \,dx$
is classified in the case
$(n,\alpha) = (3,0.5)$
(for ease of viewing, the figure is multiplied by $7$ and $0.75$ in the horizontal and vertical axis, respectively).
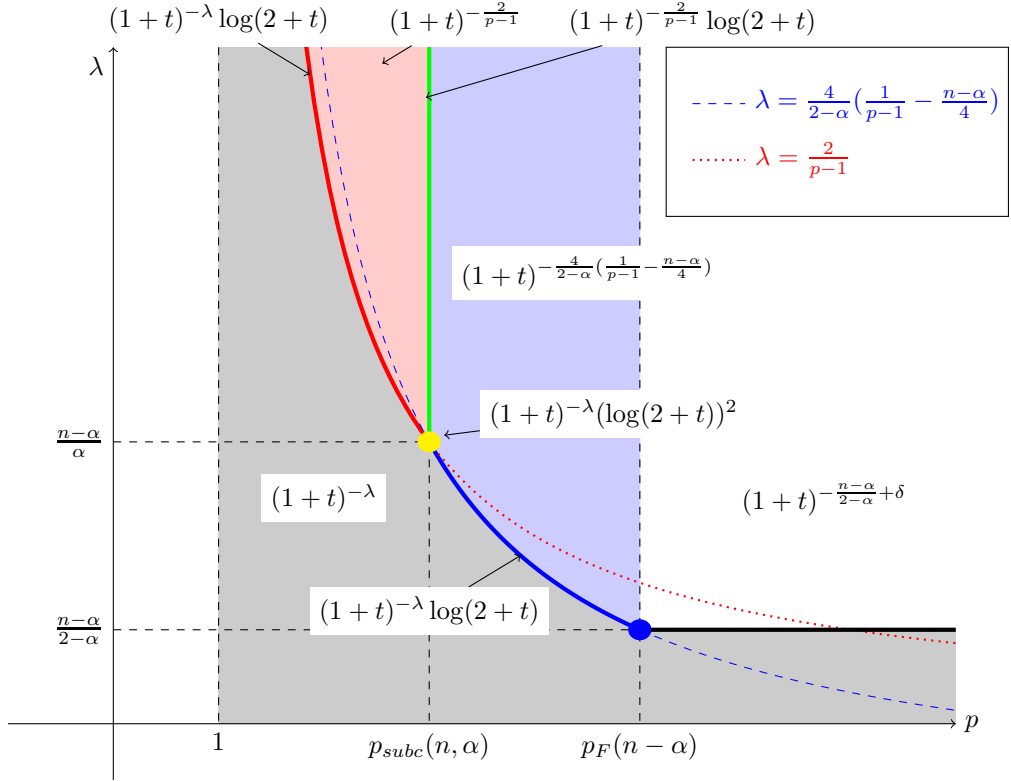
\begin{figure}[htbp]
\begin{center}
\begin{tikzpicture}[xscale=7,yscale=0.75]
\clip (0.6,-1) rectangle (2.5,13);

\fill[color=black!20] (1+2/2.5,2.5/1.5)--(2.4,2.5/1.5)--(2.4,0)--(1+2/2.5,0)--cycle;

\begin{scope} 
\clip (0.6,-2) rectangle (2.5,12);
\filldraw[color=black!20, domain=1.05:1+1/2.5]
plot (\x, {2/(\x-1)}) --(1+1/2.5,0)--(1,0)--(1,12)--cycle; 

\filldraw[color=red!20, domain=1.05:1+1/2.5]
plot (\x, {2/(\x-1)}) --(1+1/2.5,12)--cycle; 

\draw[red,ultra thick,draw,domain=1.05:1+1/2.5] plot (\x, {2/(\x-1)}); 

\filldraw[color=black!20, domain=1+1/2.5:1+2/2.5] plot (\x, {4/(1.5*(\x-1))-2.5/1.5}) --(1+2/2.5,0)--(1+1/2.5,0)--cycle; 
\filldraw[color=blue!20, domain=1+1/2.5:1+2/2.5] plot (\x, {4/(1.5*(\x-1))-2.5/1.5}) --(1+2/2.5,12)--(1+1/2.5,12)--cycle; 
\draw[samples=50,blue,ultra thick,draw,domain=1+1/2.5:1+2/2.5] plot (\x, {4/(1.5*(\x-1))-2.5/1.5}); 

\draw[samples=50,blue,dashed,draw,domain=1.1:1+1/2.5] plot (\x, {4/(1.5*(\x-1))-2.5/1.5}); 
\draw[samples=50,blue,dashed,draw,domain=1+2/2.5:2.4] plot (\x, {4/(1.5*(\x-1))-2.5/1.5}); 
\draw[samples=50,red,thick,dotted,draw,domain=1+1/2.5:2.4] plot (\x, {2/(\x-1)}); 

\draw[dashed] (0.8,2.5/1.5) node[left] {$\frac{n-\alpha}{2-\alpha}$} --(1+2/2.5,2.5/1.5); 
\draw[ultra thick] (1+2/2.5,2.5/1.5)--(2.4,2.5/1.5); 

\draw[dashed] (0.8,5) node[left] {$\frac{n-\alpha}{\alpha}$}--(1+1/2.5,5); 

\draw[dashed] (1,12)--(1,0) node[below] {$1$}; 
\draw[dashed] (1+1/2.5,12)--(1+1/2.5,0) node[below]  {$p_{subc}(n,\alpha)$}; 
\draw[dashed] (1+2/2.5,12)--(1+2/2.5,0) node[below] {$p_F(n-\alpha)$}; 

\filldraw[blue] (1+2/2.5,2.5/1.5) circle [x radius = 0.6pt, y radius = 5pt]; 

\draw[ultra thick,green] (1+1/2.5,5)--(1+1/2.5,12); 
\filldraw[yellow]  (1+1/2.5,5) circle [x radius = 0.6pt, y radius = 5pt]; 

\draw[->] (0.6,0)--(2.4,0) node[right] {$p$}; 
\draw[->] (0.8,-1)--(0.8,12) node[below left] {$\lambda$}; 
\end{scope}

\node[fill=white] at (1.2,4) {$(1+t)^{-\lambda}$}; 
\node[fill=white] at (1.7,8) {$(1+t)^{-\frac{4}{2-\alpha}(\frac{1}{p-1}-\frac{n-\alpha}{4})}$}; 
\node[fill=white] (redreg) at (1.45,12.5) {$(1+t)^{-\frac{2}{p-1}}$}; 
\draw[->] (redreg)--++(-100:0.8);
\node[fill=white] (redlin) at (1,12.5) {$(1+t)^{-\lambda} \log(2+t)$}; 
\draw[->] (redlin)--++(-80:1);
\node[fill=white] (blulin) at (1.4,2) {$(1+t)^{-\lambda} \log(2+t)$}; 
\draw[->] (blulin)--++(80:1);
\node[fill=white] (greenlin) at (1.9,12.5) {$(1+t)^{-\frac{2}{p-1}} \log(2+t)$}; 
\draw[->] (greenlin)--++(-110:1.48);
\node[fill=white] (yelpt) at (1.75,5.5) {$(1+t)^{-\lambda}(\log(2+t))^2$}; 
\draw[->] (yelpt)--++(-130:0.5);
\node at (2.15,4) {$(1+t)^{-\frac{n-\alpha}{2-\alpha}+\delta}$}; 

\draw[blue, dashed] (1.9,11)--++(0:0.1)
node[right] {$\lambda=\frac{4}{2-\alpha}(\frac{1}{p-1}-\frac{n-\alpha}{4})$};
\draw[red,thick,dotted] (1.9,10)--++(0:0.1) node[right] {$\lambda=\frac{2}{p-1}$};
\draw (1.85,9) rectangle (2.5,12);

\end{tikzpicture}
\caption{Classification of decay rates in $p$\,-\,$\lambda$ plane when
$(n,\alpha)=(3,\frac{1}{2})$}
\label{fig1}
\end{center}
\end{figure}
\end{remark}
\begin{remark}
From the proof of the above theorem, we also have
the following estimates for
the $L^2$-norm of $u$ without the weight $a(x)$:
Under the assumptions on (i) with
$\lambda \in
[\frac{\alpha}{2-\alpha},
\frac{n-\alpha}{2-\alpha})$,
we have
\begin{align}
    \int_{\Omega} |u(t,x)|^2 \,dx
    \le C (1+t)^{-\lambda + \frac{\alpha}{2-\alpha}}
\end{align}
for $t > 0$;
Under the assumptions on (ii) with
$\lambda \in [\frac{\alpha}{2-\alpha}, \infty)$,
we have
\begin{align}
    &\int_{\Omega} |u(t,x)|^2 \,dx \\
    &\le C
    \begin{cases}
    (1+t)^{-\lambda+ \frac{\alpha}{2-\alpha}}
        &(\lambda < \min\{ \frac{4}{2-\alpha}(\frac{1}{p-1} - \frac{n-\alpha}{4} ), \frac{2}{p-1} \}),\\
    (1+t)^{-\lambda+ \frac{\alpha}{2-\alpha}} \log (2+t)
        &(\lambda = \min\{ \frac{4}{2-\alpha}(\frac{1}{p-1} - \frac{n-\alpha}{4} ), \frac{2}{p-1} \}, \ p \neq p_{subc}(n,\alpha)),\\
    (1+t)^{-\lambda+ \frac{\alpha}{2-\alpha}} (\log (2+t))^2
        &(\lambda = \frac{4}{2-\alpha}(\frac{1}{p-1} - \frac{n-\alpha}{4} ) = \frac{2}{p-1}, \ \mathrm{i.e.,} \ p = p_{subc}(n,\alpha)),\\
    (1+t)^{-\frac{4}{2-\alpha} ( \frac{1}{p-1} - \frac{n-\alpha}{4})+ \frac{\alpha}{2-\alpha}}
        &(\lambda > \frac{4}{2-\alpha}(\frac{1}{p-1} - \frac{n-\alpha}{4} ),
        \ p > p_{subc}(n,\alpha)),\\
    (1+t)^{-\frac{2}{p-1}+ \frac{\alpha}{2-\alpha}} \log (2+t)
        &(\lambda > \frac{2}{p-1},
        \ p = p_{subc}(n,\alpha)),\\
    (1+t)^{-\frac{2}{p-1}+ \frac{\alpha}{2-\alpha}}
        &(\lambda > \frac{2}{p-1}, 
        \ p<p_{subc}(n,\alpha))
    \end{cases}
\end{align}
for $t>0$.
\end{remark}
\begin{remark}
(i)
Theorem \ref{thm1} generalizes
the result of Nishihara \cite{Ni10}
to
the exterior domain,
general damping coefficient
$a(x)$ satisfying \eqref{assum:a},
and
polynomially decaying initial data
satisfying \eqref{assum:ini}.

(ii) For the simplest case
$\Omega = \mathbb{R}^n$
and
$a(x) \equiv 1$,
the result of Theorem \ref{thm1} (ii)
extends that of 
Ikehata, Nishihara, and Zhao \cite{IkNiZh06},
in the sense that our estimate in the region
$\lambda > 2\left(\frac{1}{p-1}-\frac{n}{4}\right)$
coincides with their estimate
\eqref{eq:est:IkNiZh}.
Moreover, the result of
Theorem \ref{thm1} (i) in the case
$p > p_F(n)$
is better than
the estimate obtained in \cite{IkNiZh06}.
Hence, our result still has a novelty.
\end{remark}
\begin{remark}
The optimality of the decay rates
in Theorem \ref{thm1} is an open problem.
We expect that 
the estimate in the case (i)
is optimal if
$p > p_F(n-\alpha) = 1+ \frac{2}{n-\alpha}$,
since the decay rate is the same as that of
the linear problem \eqref{dwx:lin}
obtained by \cite{SoWa21_JMSJ}.
On the other hand, in the critical case
$p = p_F(n-\alpha)$,
the estimates in Theorem \ref{thm1}
will be improved in view of the known results
\cite{HaKaNa06, HaNa17JMAA}
for the classical damping \eqref{dw:nonlin} in the whole space.
Moreover, the optimality in the
subcritical case
$p < p_F(n-\alpha)$
is a difficult problem even when
$a(x) \equiv 1$ and $\Omega = \mathbb{R}^n$,
and we have no idea so far.
\end{remark}

The strategy of the proof of Theorem \ref{thm1}
is as follows.
For the both parts (i) and (ii),
we apply the weighted energy method.
The difficulty is how to estimate
the weighted
$L^2$-norm of the solution.
To overcome it, we take different approaches
for (i) and (ii).
First, for the part (i), we apply
the weighted energy method developed by
\cite{SoWa19_CCM, SoWa21_JMSJ}.
We shall use a suitable supersolution of
the corresponding heat equation
$a(x) \partial_t v -\Delta v = 0$
as the weight function.
Next, for the part (ii),
we shall use the same type weight function as in
Ikehata, Nishihara, and Zhao \cite{IkNiZh06}
with a modification to fit the space-dependent 
damping case.
In this case the absorbing semilinear term
helps to estimate
the weighted $L^2$-norm of the solution.

The rest of the paper is organized
in the following way.
In the next section,
we prepare the definitions and properties of the weight functions used in the proof.
Sections 3 and 4 are devoted to the proof of
Theorem \ref{thm1} (i) and (ii), respectively.
In Appendix A, we give a proof of
Proposition \ref{prop:ex}.
Finally, in Appendix B,
we prove the properties of weight functions
stated in Section 2.

We end up this section with
introducing notations used throughout this paper.
The letter $C$ indicates a generic positive constant, which may change from line to line.
In particular,
$C(\ast,\cdots,\ast)$
denotes a constant depending only on
the quantities in the parentheses.
For $x = (x_1,\ldots,x_n) \in \mathbb{R}^n$,
we define 
$\langle x \rangle = \sqrt{1+|x|^2}$.
We sometimes use
$B_{R}(x_0) = \{ x \in \mathbb{R}^n ;\,
|x-x_0| < R\}$
for
$R>0$
and $x_0 \in \mathbb{R}^n$.

Let
$L^p(\Omega)$
be the usual Lebesgue space
equipped with the norm
\begin{align}
    \| f \|_{L^p}
    = \begin{dcases}
    \left( \int_{\Omega} |f(x)|^p \,dx \right)^{1/p}
    &(1<p<\infty),\\
    \esssup_{x\in \Omega} |f(x)|
    &(p=\infty).
    \end{dcases}
\end{align}
In particular,
$L^2(\Omega)$
is a Hilbert space with the innerproduct
\begin{align}
    (f,g)_{L^2}
    &:=
    \int_{\Omega} f(x) g(x) \,dx.
\end{align}
Let
$H^k(\Omega)$
with a nonnegative integer $k$
be the Sobolev space
equipped with
the innerproduct and the norm
\begin{align}
    (f,g)_{H^k}
    &=
    \sum_{|\alpha|\le k}
    (\partial^{\alpha} f,
    \partial^{\alpha} g )_{L^2},
    \quad
    \|f \|_{H^k}
    =
    \sqrt{(f,f)_{H^k}},
\end{align}
respectively.
$C_0^{\infty}(\Omega)$
denotes the space of
smooth functions on $\Omega$
with compact support.
$H_0^k(\Omega)$
is the completion of
$C_0^{\infty}(\Omega)$
with respect to the norm
$\| \cdot \|_{H^k}$.
For an interval $I \subset \mathbb{R}$,
a Banach space $X$, and a nonnegative integer $k$,
$C^k(I;X)$
stands for the space of
$k$-times continuously differentiable functions from $I$ to $X$.

\section{Preliminaries}
In this section,
we prepare
weight functions for the weighted energy method
used in the proof of Theorem \ref{thm1}.

These lemmas were shown in
\cite{So19DIE,SoWa17_AIMS,SoWa19_CCM,SoWa21_JMSJ},
however, for the convenience,
we give a proof of them in the appendix.

Following \cite{SoWa17_AIMS},
we first take a suitable approximate solution of
the Poisson equation
$\Delta A(x) = a(x)$,
which will be used for the construction
of the weight function.

\begin{lemma}[\cite{SoWa17_AIMS, SoWa21_JMSJ}]\label{lem:A:ep}
Assume that
$a(x) \in C(\mathbb{R}^n)$ is positive
and satisfies the condition
$\lim_{|x|\to \infty} |x|^{\alpha}a(x) = a_0$
with some constants
$\alpha\in (-\infty,\min\{2,n\})$
and $a_0 > 0$.
Let
$\varepsilon \in (0,1)$.
Then, there exist a function
$A_{\varepsilon} \in C^2(\mathbb{R}^n)$
and positive constants
$c = c(n,a,\varepsilon)$
and
$C=C(n,a,\varepsilon)$
such that for
$x \in \mathbb{R}^n$, we have
\begin{align}
\label{A1}
	&(1-\varepsilon)a(x)\le \Delta A_{\varepsilon} (x) \le (1+\varepsilon) a(x),\\
\label{A2}
	&c \langle x\rangle^{2-\alpha} \leq A_{\varepsilon} (x) \le C \langle x \rangle^{2-\alpha},\\
\label{A3}
	&\frac{|\nabla A_{\varepsilon} (x)|^2}{a(x)A_{\varepsilon}(x)} \le \frac{2-\alpha}{n-\alpha}+\varepsilon.
\end{align}
\end{lemma}

For the construction of our weight function,
we also need the following
Kummer's confluent hypergeometric function.
\begin{definition}[Kummer's confluent hypergeometric functions]
\label{def:M}
For
$b,c \in \mathbb{R}$ with $-c \notin \mathbb{N} \cup \{0\}$,
Kummer's confluent hypergeometric function of first kind is defined by
\begin{align*}
	M(b, c; s) = \sum_{n=0}^{\infty} \frac{(b)_n}{(c)_n} \frac{s^n}{n!}, \quad s\in [0,\infty),
\end{align*}
where $(d)_n$ is the Pochhammer symbol defined by
$(d)_0 = 1$
and
$(d)_n = \prod_{k=1}^n (d+k-1)$
for $n\in\mathbb{N}$; 
note that when
$b=c$, $M(b,b;s)$
coincides with $e^s$.
\end{definition}

For $\varepsilon \in (0,1/2)$, we define
\begin{align}
\label{gammatilde}
	\widetilde{\gamma}_{\varepsilon}
	=\left( \frac{2-\alpha}{n-\alpha}+2 \varepsilon \right)^{-1}, \quad 
	\gamma_{\varepsilon} = (1-2\varepsilon)
	\widetilde{\gamma}_{\varepsilon}.
\end{align}

\begin{definition}\label{phi.beta}
For $\beta\in \mathbb{R}$, define 
\begin{align*}
	\varphi_{\beta,\varepsilon}(s)=e^{-s}M \left(\gamma_{\varepsilon}-\beta, \gamma_{\varepsilon}; s\right),
	\quad s \ge 0.
\end{align*}
\end{definition}
Since
$M(\gamma_{\varepsilon},\gamma_{\varepsilon},s)
=e^s$,
we remark that
$\varphi_{0,\varepsilon}(s) \equiv 1$.
Roughly speaking,
if we formally take
$\varepsilon = 0$,
then
$\{ \varphi_{\beta,0} \}_{\beta \in \mathbb{R}}$
gives a family of self-similar profiles
of the equation
$|x|^{-\alpha} \partial_t v = \Delta v$
with the parameter $\beta$.
See \cite{SoWa19_CCM}
for more detailed explanation.
The next lemma states basic properties of
$\varphi_{\beta,\varepsilon}$.

\begin{lemma}\label{lem:phi.beta}
The function $\varphi_{\beta,\varepsilon}$ defined in Definition \ref{phi.beta} satisfies the following properties.
\begin{itemize}
    \item[(i)]
    $\varphi_{\beta,\varepsilon}(s)$
    satisfies the equation
    \begin{align}\label{eq:varphi}
        s \varphi''(s) + (\gamma_{\varepsilon}+s)\varphi'(s) + \beta \varphi(s) = 0.
    \end{align}
    \item[(ii)]
    If $0 \le \beta < \gamma_{\varepsilon}$, then
    $\varphi_{\beta,\varepsilon}(s)$ satisfies the estimates
    \begin{align}
        k_{\beta,\varepsilon} (1+s)^{-\beta} \le \varphi_{\beta,\varepsilon} (s) \le K_{\beta,\varepsilon} (1+s)^{-\beta}
    \end{align}
    with some constants
    $k_{\beta,\varepsilon}, K_{\beta,\varepsilon} > 0$.
    \item[(iii)]
    For every $\beta \ge 0$,
    $\varphi_{\beta,\varepsilon}(s)$ satisfies
    \begin{align}
        |\varphi_{\beta,\varepsilon}(s)| \le K_{\beta,\varepsilon}(1+s)^{-\beta}
    \end{align}
    with some constant $K_{\beta,\varepsilon} > 0$.
    \item[(iv)]
    For every $\beta \in \mathbb{R}$,
    $\varphi_{\beta,\varepsilon}(s)$ and $\varphi_{\beta+1,\varepsilon}(s)$
    satisfy the recurrence relation
    \begin{align}
        \beta \varphi_{\beta,\varepsilon}(s) + s \varphi_{\beta,\varepsilon}'(s)
        = \beta \varphi_{\beta+1,\varepsilon}(s).
    \end{align}
    \item[(v)]
    For every $\beta \in \mathbb{R}$,
    we have
    \begin{align}
        \varphi_{\beta,\varepsilon}'(s) &= - \frac{\beta}{\gamma_{\varepsilon}} e^{-s} M(\gamma_{\varepsilon}-\beta, \gamma_{\varepsilon} + 1; s),\\
        \varphi_{\beta,\varepsilon}''(s) &= \frac{\beta(\beta+1)}{\gamma_{\varepsilon}(\gamma_{\varepsilon}+1)} e^{-s} M(\gamma_{\varepsilon}-\beta, \gamma_{\varepsilon} + 2; s).
    \end{align}
    In particular, if
    $0< \beta < \gamma_{\varepsilon}$,
    then
    $\varphi_{\beta,\varepsilon}'(s)$
    and
    $\varphi_{\beta,\varepsilon}''(s)$
    satisfy
    \begin{align}
        -K_{\beta,\varepsilon} (1+s)^{-\beta-1}
        \le \varphi_{\beta,\varepsilon}'(s)
        \le - k_{\beta,\varepsilon} (1+s)^{-\beta-1},\\
        k_{\beta,\varepsilon} (1+s)^{-\beta-2}
        \le \varphi_{\beta,\varepsilon}''(s)
        \le K_{\beta,\varepsilon} (1+s)^{-\beta-2}
    \end{align}
    with some constants
    $k_{\beta,\varepsilon}, K_{\beta,\varepsilon} > 0$.
\end{itemize}
\end{lemma}
Finally, we define the weight function
which will be used
for our energy method.
\begin{definition}\label{phi.beta.ep}
For
$\beta \in \mathbb{R}$
and
$(x,t) \in \mathbb{R}^n \times [0,\infty)$,
we define 
\begin{align}
	\Phi_{\beta,\varepsilon}(x,t; t_0)=(t_0+t)^{-\beta}\varphi_{\beta,\varepsilon}(z), 
	\quad 
	z=\frac{\widetilde{\gamma}_\varepsilon A_\varepsilon(x)}{t_0+t},
\end{align}
where
$\varepsilon \in (0,1/2)$,
$\widetilde \gamma_{\varepsilon}$ is the constant given in \eqref{gammatilde},
$t_0 \ge 1$,
$\varphi_{\beta,\varepsilon}$
is the function defined by Definition \ref{phi.beta},
and
$A_{\varepsilon}(x)$ is the function constructed in Lemma \ref{lem:A:ep}.
\end{definition}
Since
$\varphi_{0,\varepsilon}(s) \equiv 1$,
we again remark that 
$\Phi_{0,\varepsilon}(x,t;t_0) \equiv 1$.

For $t_0 \ge 1$, $t>0$, and $x\in \mathbb{R}^n$,
we also define
\begin{align}
\label{psi}
	\Psi(x,t; t_0) :=
		t_0 + t + A_{\varepsilon}(x).
\end{align}

\begin{proposition}\label{prop:super-sol}
The function
$\Phi_{\beta,\varepsilon}(x,t;t_0)$
satisfies the following properties:
\begin{itemize}
\item[(i)]
For every
$\beta \ge 0$,
we have
\begin{align*}
	\partial_t \Phi_{\beta,\varepsilon}(x,t;t_0) = -\beta \Phi_{\beta+1,\varepsilon}(x,t;t_0).
\end{align*}
\item[(ii)]
If $\beta \ge 0$, then
there exists a constant
$C = C(n,\alpha,\beta,\varepsilon) > 0$
such that
\begin{align*}
	|\Phi_{\beta,\varepsilon}(x,t;t_0) |
	\le
	C \Psi (x,t; t_0)^{-\beta}
\end{align*}
for any
$(x,t) \in \mathbb{R}^n \times [0,\infty)$.
\item[(iii)]
If
$0\le \beta < \gamma_{\varepsilon}$,
then there exists a constant
$c = c(n,\alpha,\beta,\varepsilon) > 0$
such that
\begin{align*}
	\Phi_{\beta,\varepsilon}(x,t;t_0) 
	\ge
	c \Psi (x,t; t_0)^{-\beta}
\end{align*}
for any $(x,t) \in \mathbb{R}^n \times [0,\infty)$.
\item[(iv)]
For
$\beta > 0$,
there exists a constant
$c = c(n,\alpha,\beta,\varepsilon)>0$
such that
\begin{align*}
	a(x)\partial_t\Phi_{\beta,\varepsilon}(x,t;t_0)
	-\Delta \Phi_{\beta,\varepsilon}(x,t;t_0)
	\ge
	c a(x) \Psi(x,t;t_0)^{-\beta-1}
\end{align*}
for any
$(x,t) \in \mathbb{R}^n \times [0,\infty)$.
\end{itemize}
\end{proposition}

Finally, we prepare a useful lemma
for our weighted energy method.
The proof can be found in
\cite[Lemma 3.6]{SoWa19_CCM}
or
\cite[Lemma 2.5]{So19DIE}.
However, for the convenience, we give its proof
in the appendix.
\begin{lemma}\label{lem:deltaphi}
Let
$\Omega = \mathbb{R}^n$
with $n \ge 1$
or
$\Omega \subset \mathbb{R}^n$
with $n \ge 2$
be an exterior domain with
$C^2$-boundary.
Let $\Phi \in C^2(\overline{\Omega})$
be a positive function and
let $\delta \in (0, 1/2)$.
Then, for any
$u\in H^2(\Omega) \cap H^1_0(\Omega)$
satisfying 
$\supp u \in B_R(0) = \{ x \in \mathbb{R}^n ;\, |x| < R \}$
with some
$R>0$,
we have
\begin{align*}
	\int_{\Omega} \left( u \Delta u \right) \Phi^{-1+2\delta}\,dx
		&\le - \frac{\delta}{1-\delta} \int_{\Omega} |\nabla u|^2 \Phi^{-1+2\delta}\,dx
			+ \frac{1-2\delta}{2} \int_{\Omega} u^2 (\Delta \Phi) \Phi^{-2+2\delta} \,dx.
\end{align*}
\end{lemma}

\section{Proof of Theorem \ref{thm1}: first part}\label{sec:case1}

In this section, we prove Theorem \ref{thm1} (i).
First, we note that
Proposition \ref{prop:ex}
implies the existence of the global
mild solution $u$.

Following the argument in
Sobajima \cite{So_pre},
we first prove Theorem \ref{thm1} (i)
in the case of compactly supported
initial data,
and after that, we will treat the
general case by an approximation
argument.

\subsection{Proof for the
compactly supported initial data}
We first consider the case
where the initial data are
compactly supported,
that is, we assume that
$\supp u_0 \cup \supp u_1 \subset
B_{R_0}(0) = \{ x \in \mathbb{R}^n ;\,
|x| < R_0\}$.
Then, by the finite propagation property
(see Section \ref{sec:fpp}),
the corresponding mild solution
$u$
satisfies
$\supp u(t,\cdot) \subset
B_{R_0+t}(0)$.

Let $T_0 > 0$ be arbitrary fixed
and let $T \in (0,T_0)$.
Then, we have
$\supp u(t,\cdot) \subset B_{R_0+T_0}(0)$
for all
$t \in [0,T]$.
Let
$D = \Omega \cap B_{R_0+T_0}(0)$.
Then,
for $t \in [0,T]$,
we can convert the problem \eqref{dwx}
to the problem in the bounded domain
\begin{align}\label{eq:3:dwx}
    \left\{ \begin{array}{ll}
        \partial_t^2 u - \Delta u + a(x) \partial_t u + |u|^{p-1}u = 0, &t\in (0,T], x \in D, \\
        u(t,x) = 0, &t \in (0,T], x \in \partial D,\\
        u(0,x) = u_0(x), \ 
        \partial_t u(0,x) = u_1(x),
        &x \in D
        \end{array} \right.
\end{align}
with
$(u_0,u_1) \in \mathcal{H}_D := H_0^1(D) \times L^2(D)$.

Let
$\mathcal{A}_D$
be the operator
\begin{align}
    \mathcal{A}_D =
    \begin{pmatrix} 0 & 1 \\ \Delta & -a(x) \end{pmatrix}
\end{align}
defined on
$\mathcal{H}_D$
with the domain
$D(\mathcal{A}_D) = (H^2(D) \cap H^1_0(D)) \times H^1_0(D)$.
Then, from the argument in
Section \ref{sec:A:1},
there exists
$\lambda_{\ast} > 0$
such that for any
$\lambda > \lambda_{\ast}$,
the resolvent
$J_{\lambda} = (I-\lambda^{-1}\mathcal{A}_D)^{-1}$
is defined as a bounded operator on
$\mathcal{H}_D$.
Take a sequence
$\{ \lambda_j \}_{j=1}^{\infty}$
such that
$\lambda_j > \lambda_{\ast}$
for $j \ge 1$
and
$\lim_{j \to \infty} \lambda_j = \infty$,
and define
\begin{align}
    \begin{pmatrix}
    u_0^{(j)}\\ u_1^{(j)}
    \end{pmatrix}
    := J_{\lambda_j} 
    \begin{pmatrix}
    u_0\\ u_1
    \end{pmatrix}.
\end{align}
Then, we have
\begin{align}\label{eq:3:approx}
    (u_0^{(j)}, u_1^{(j)})
    \in D(\mathcal{A}_D),
    \quad
    \lim_{j\to \infty}
    (u_0^{(j)},u_1^{(j)})
    = (u_0, u_1)
    \ \text{in} \ \mathcal{H}_D
\end{align}
(see e.g. the proof of
\cite[Theorem 2.18]{Ikawa}).
Therefore,
Proposition \ref{prop:ex} shows that
the mild solution
$u^{(j)}$
corresponding to the initial data 
$(u_0^{(j)}, u_1^{(j)})$
becomes a strong solution.
Moreover, the continuous dependence on
the initial data
(see Section \ref{sec:conti:dep})
implies
\begin{align}
    \lim_{j\to \infty}
    \sup_{t \in [0,T]}
    \| (u^{(j)}(t), \partial_t u^{(j)}(t))
    - (u(t),\partial_t u(t)) \|_{\mathcal{H}_D} = 0.
\end{align}
This means that, if we prove
the conclusion of Theorem \ref{thm1} (i)
for $u^{(j)}$, that is,
\begin{align}
    (1+t)E[u^{(j)}](t)
    + \int_{\Omega} a(x) |u^{(j)}(t,x)|^2 \,dx
    \le
    C I_0[u_0^{(j)}, u_1^{(j)}]
    (1+t)^{-\lambda}
\end{align}
for $t \in [0,T]$,
where the constant
$C$ is independent of
$j, T, T_0, R_0$,
then letting
$j \to \infty$
and also using the Sobolev embedding
$\| u \|_{L^{p+1}(D)} \le
C \| u \|_{H^1(D)}$,
we have the same estimate for
the original mild solution
$u$.
Note that \eqref{eq:3:approx}
implies
$\lim_{j\to \infty} I_0[u_0^{(j)},u_1^{(j)}] = I_0[u_0,u_1]$,
since the integral is taken over
the bounded region $D$.
Finally, since
$T$ and $T_0$ are arbitrary and
$C$ is independent of them,
we obtain the desired energy estimate
for any $t \ge 0$.

Therefore, in the following argument,
we may further assume
$(u_0, u_1) \in D(\mathcal{A}_D)$
and $u$ is the strong solution.
This enables us to justify
all the computations
in this section.

In what follows,
we shall use the weight functions
$\Phi_{\beta,\varepsilon}(x,t;t_0)$
and
$\Psi(x,t;t_0)$
defined by
Definition \ref{phi.beta.ep} and
\eqref{psi},
respectively.
We also recall that
the constant
$\gamma_{\varepsilon}$
is given by \eqref{gammatilde}.
Then, we define the following energies.

\begin{definition}\label{def:1:energy}
For a function
$u = u(t,x)$,
$\alpha \in [0,1)$,
$\delta \in (0,1/2)$,
$\varepsilon \in (0,1/2)$,
$\lambda \in [0,(1-2\delta)\gamma_{\varepsilon})$,
$\beta = \lambda/(1-2\delta)$,
$\nu > 0$,
and
$t_0 \ge 1$,
we define
\begin{align}
    E_1(t;t_0,\lambda) &=
    \int_{\Omega} \left[
    \frac{1}{2}\left( |\partial_t u(t,x)|^2 + |\nabla u(t,x)|^2 \right)
    + \frac{1}{p+1} |u(t,x)|^{p+1}
    \right]
    \Psi(t,x;t_0)^{\lambda+\frac{\alpha}{2-\alpha}}
    \,dx,\\
    E_0(t;t_0,\lambda)
    &= \int_{\Omega}
    \left(
    2u(t,x) \partial_t u(t,x) + a(x) |u(t,x)|^2 
    \right)
    \Phi_{\beta,\varepsilon}(t,x;t_0)^{-1+2\delta} \,dx,\\
    E_{\ast}(t;t_0,\lambda,\nu)
    &=
    E_1(t;t_0,\lambda)
    + \nu E_0(t;t_0,\lambda),\\
    \tilde{E}(t;t_0,\lambda)
    &=
    (t_0+t) \int_{\Omega} \left[
    \frac{1}{2}\left( |\partial_t u(t,x)|^2 + |\nabla u(t,x)|^2 \right)
    + \frac{1}{p+1} |u(t,x)|^{p+1}
    \right]
    \Psi(t,x;t_0)^{\lambda}
    \,dx
\end{align}
for $t \ge 0$.
\end{definition}
Since
\begin{align}\label{eq:3:en:pos1}
    2 u \partial_t u
    &\le
    \frac{a(x)}{2} |u|^2 + \frac{2}{a(x)} |\partial_t u|^2 
    \le 
    \frac{a(x)}{2} |u|^2 + C \Psi^{\frac{\alpha}{2-\alpha}} |\partial_t u|^2
\end{align}
and
$\Phi_{\beta,\varepsilon}^{-1+2\delta}
\le C \Psi^{\lambda}$
(see \eqref{A2} and Proposition \ref{prop:super-sol} (iii)),
we see that there exists a small constant
$\nu_0 = \nu_0(n,a,\delta,\varepsilon,\lambda) > 0$
such that for any
$\nu \in (0,\nu_0)$,
\begin{align}\label{eq:3:en:pos2}
    E_{\ast}(t;t_0,\lambda, \nu)
    &\ge
    \frac{1}{2} E_1(t;t_0,\lambda)
    + \frac{\nu}{2} \int_{\Omega}
    a(x) |u(t,x)|^2 
    \Psi(t,x;t_0)^{\lambda} \,dx
\end{align}
holds.

We first prepare the following
energy estimates for
$E_1(t;t_0,\lambda)$
and
$E_0(t;t_0,\lambda)$.

\begin{lemma}\label{lem:1:E1}
Under the assumptions on Theorem \ref{thm1} (i),
there exists
$t_1 = t_1(n,a,\lambda,\varepsilon) \ge 1$
such that for
$t_0 \ge t_1$
and $t > 0$, we have
\begin{align}
    \frac{d}{dt} E_1(t;t_0,\lambda) 
    &\le
    - \frac{1}{2} \int_{\Omega} a(x) |\partial_t u(t,x)|^2 \Psi(t,x;t_0)^{\lambda + \frac{\alpha}{2-\alpha}} \,dx \\
    &\quad +
    C \int_{\Omega}
    \left(
        |\nabla u(t,x)|^2
        + |u(t,x)|^{p+1}
    \right)
    \Psi(t,x;t_0)^{\lambda + \frac{\alpha}{2-\alpha}-1} \,dx
\end{align}
with some constant
$C = C(n,\alpha,p,\lambda) > 0$.
\end{lemma}
\begin{proof}
Differentiating
$E_1(t;t_0,\lambda)$,
one has
\begin{align}
    \frac{d}{dt} E_1(t;t_0,\lambda)
    &=
    \int_{\Omega}
    \left[
    \partial_t u \partial_t^2 u
    + \nabla u \cdot \nabla \partial_t u
    + |u|^{p-1}u \partial_t u
    \right]
    \Psi^{\lambda + \frac{\alpha}{2-\alpha}} \,dx \\
    &\quad
    + \left( \lambda + \frac{\alpha}{2-\alpha} \right)
    \int_{\Omega}
    \left[
    \frac{1}{2}\left( |\nabla u|^2 + |\partial_t u|^2 \right)
    + \frac{1}{p+1} |u|^{p+1}
    \right]
    \Psi^{\lambda + \frac{\alpha}{2-\alpha}-1} \,dx.
\end{align}
The integration by parts
and the equation \eqref{dwx} imply
\begin{align}
    \frac{d}{dt} E_1(t;t_0,\lambda)
    &=
    - \int_{\Omega} a(x) |\partial_t u|^2
    \Psi^{\lambda + \frac{\alpha}{2-\alpha}} \,dx \\
    &\quad
    - \left( \lambda + \frac{\alpha}{2-\alpha} \right)
    \int_{\Omega}
    \partial_t u (\nabla u \cdot \nabla \Psi) \Psi^{\lambda + \frac{\alpha}{2-\alpha}-1} \,dx \\
    &\quad
    + \left( \lambda + \frac{\alpha}{2-\alpha} \right)
    \int_{\Omega}
    \left[
    \frac{1}{2}\left( |\nabla u|^2 + |\partial_t u|^2 \right)
    + \frac{1}{p+1} |u|^{p+1}
    \right]
    \Psi^{\lambda + \frac{\alpha}{2-\alpha}-1} \,dx.
\label{eq:1:lem:E1:1}
\end{align}
Let us estimate the right-hand side.
First, the Schwarz inequality gives
\begin{align}
    \left| -\left( \lambda + \frac{\alpha}{2-\alpha} \right)
    \partial_t u (\nabla u \cdot \nabla \Psi)
    \right|
    &\le
    \frac{a(x)}{4} |\partial_t u|^2 \Psi
    + C |\nabla u|^2
    \frac{|\nabla \Psi|^2}{a(x)\Psi}.
\end{align}
Moreover, by \eqref{A3}, we have
\begin{align}\label{eq:1:Psi:quot}
    \frac{|\nabla \Psi|^2}{a(x)\Psi}
    \le \frac{|\nabla A_{\varepsilon}(x)|^2}{a(x)A_{\varepsilon}(x)}
    \le \frac{2-\alpha}{n-\alpha} + \varepsilon.
\end{align}
Also, from the definition of
$\Psi$,
\eqref{A2},
and $a(x) \sim \langle x \rangle^{-\alpha}$,
one obtains
\begin{align}\label{eq:1:Psi:-1}
    \Psi(t,x;t_0)^{-1}
    \le
    t_0^{-1+\frac{\alpha}{2-\alpha}}
    A_{\varepsilon}(x)^{-\frac{\alpha}{2-\alpha}}
    \le
    C t_0^{-\frac{2(1-\alpha)}{2-\alpha}}
    a(x).
\end{align}
Therefore, taking
$t_1 \ge 1$
sufficiently large,
we have, for $t_0 \ge t_1$,
\begin{align}
    \left( \lambda + \frac{\alpha}{2-\alpha} \right)
    \int_{\Omega} |\partial_t u|^2
    \Psi^{\lambda + \frac{\alpha}{2-\alpha} -1} \,dx
    \le
    \frac{1}{4}
    \int_{\Omega}a(x) |\partial_t u|^2
    \Psi^{\lambda + \frac{\alpha}{2-\alpha}}\,dx.
\end{align}
Using the above estimates to
\eqref{eq:1:lem:E1:1},
we deduce
\begin{align}
    \frac{d}{dt} E_1(t;t_0,\lambda) 
    &\le
    - \frac{1}{2} \int_{\Omega} a(x) |\partial_t u|^2 \Psi^{\lambda + \frac{\alpha}{2-\alpha}} \,dx \\
    &\quad
    + C \int_{\Omega}
    \left(
    |\nabla u|^2
    + |u|^{p+1}
    \right)
    \Psi^{\lambda + \frac{\alpha}{2-\alpha}-1} \,dx,
\end{align}
which completes the proof.
\end{proof}
\begin{lemma}\label{lem:1:E0}
Under the assumptions on Theorem \ref{thm1} (i),
for $t_0 \ge 1$ and $t > 0$,
we have
\begin{align}
    \frac{d}{dt} E_0(t;t_0,\lambda)
    &\le
    -
    \eta \int_{\Omega}
    \left(
        |\nabla u(t,x)|^2
        + |u(t,x)|^{p+1}
    \right)
    \Psi(t,x;t_0)^{\lambda} \,dx \\
    &\quad
    +
    C \int_{\Omega} |\partial_t u(t,x)|^2
    \Psi(t,x;t_0)^{\lambda} \,dx
\end{align}
with some positive constants
$\eta = \eta (n,\alpha,\delta,\varepsilon,\lambda)$
and
$C = C(n,\alpha,\delta,\varepsilon,\lambda)$.
\end{lemma}
\begin{proof}
Differentiating
$E_0(t;t_0,\lambda)$
and using the equation \eqref{dwx}
yield
\begin{align}
    \frac{d}{dt} E_0(t;t_0,\lambda)
    &=
    \int_{\Omega}
    \left( 
    2 |\partial_t u|^2
    + 2u \partial_t^2 u
    + 2 a(x) u \partial_t u
    \right)
    \Phi_{\beta,\varepsilon}^{-1+2\delta} \,dx \\
    &\quad
    - (1-2\delta)
    \int_{\Omega}
    \left(
        2u \partial_t u
        + a(x) |u|^2
    \right)
    \Phi_{\beta,\varepsilon}^{-2+2\delta}
    \partial_t \Phi_{\beta,\varepsilon}\,dx.
\end{align}
Using the equation \eqref{dwx},
we have
\begin{align}
    \frac{d}{dt} E_0(t;t_0,\lambda)
    &=
    2 \int_{\Omega} |\partial_t u|^2
    \Phi_{\beta,\varepsilon}^{-1+2\delta} \,dx
    +
    2 \int_{\Omega} u \Delta u
    \Phi_{\beta,\varepsilon}^{-1+2\delta} \,dx \\
    &\quad
    -2 \int_{\Omega} |u|^{p+1} \Phi_{\beta,\varepsilon}^{-1+2\delta} \,dx \\
    &\quad
    - (1-2\delta)
    \int_{\Omega}
    \left(
        2u \partial_t u
        + a(x) |u|^2
    \right)
    \Phi_{\beta,\varepsilon}^{-2+2\delta}
    \partial_t \Phi_{\beta,\varepsilon}\,dx.
\end{align}
Applying Lemma \ref{lem:deltaphi}
with
$\Phi = \Phi_{\beta,\varepsilon}$
to the second term of the
right-hand side,
one obtains
\begin{align}
    \frac{d}{dt} E_0(t;t_0,\lambda)
    &\le
    2 \int_{\Omega} |\partial_t u|^2
    \Phi_{\beta,\varepsilon}^{-1+2\delta}\,dx
    -
    \frac{2\delta}{1-\delta} \int_{\Omega}
    |\nabla u|^2 \Phi_{\beta,\varepsilon}^{-1+2\delta} \,dx \\
    &\quad
    -2 \int_{\Omega} |u|^{p+1} \Phi_{\beta,\varepsilon}^{-1+2\delta} \,dx \\
    &\quad
    - 
    2(1-2\delta)
    \int_{\Omega}
    u \partial_t u
    \Phi_{\beta,\varepsilon}^{-2+2\delta}
    \partial_t \Phi_{\beta,\varepsilon}\,dx \\
    &\quad
    - (1-2\delta)
    \int_{\Omega}
        |u|^2
    \Phi_{\beta,\varepsilon}^{-2+2\delta}
    \left(
    a(x) \partial_t \Phi_{\beta,\varepsilon}
    - \Delta \Phi_{\beta,\varepsilon}
    \right)
    \,dx.
\label{eq:1:E0est1}
\end{align}
Next, we estimate the terms in the right-hand side.
First, we remark that
if
$\lambda = 0$ (i.e., $\beta = 0$),
then the last two terms in \eqref{eq:1:E0est1}
vanish, since
$\Phi_{\beta,\varepsilon}\equiv 1$.
For the case
$\beta > 0$,
by Proposition \ref{prop:super-sol} (ii) and (iv),
we have
\begin{align}
    \int_{\Omega}
        |u|^2
    \Phi_{\beta,\varepsilon}^{-2+2\delta}
    \left(
    a(x) \partial_t \Phi_{\beta,\varepsilon}
    - \Delta \Phi_{\beta,\varepsilon}
    \right)
    \,dx
    &\ge
    \eta_1
    \int_{\Omega} a(x) |u|^2
    \Psi^{\lambda - 1} \,dx
\end{align}
with some constant
$\eta_1 = \eta_1(n,\alpha,\delta,\varepsilon,\lambda)>0$.
Moreover, Proposition \ref{prop:super-sol} (i), (ii), and (iii) imply
\begin{align}
    |u \partial_t u \Phi_{\beta,\varepsilon}^{-2+2\delta}
    \partial_t \Phi_{\beta,\varepsilon} |
    &\le
    C |u| |\partial_t u|
    |\Phi_{\beta,\varepsilon}^{-2+2\delta}|
    |\Phi_{\beta+1,\varepsilon}|
    \le
    C |u| |\partial_t u| \Psi^{\lambda -1}.
\end{align}
This and the Schwarz inequality lead to
\begin{align}
    &\left|2(1-2\delta)
    \int_{\Omega}
    u \partial_t u
    \Phi_{\beta,\varepsilon}^{-2+2\delta}
    \partial_t \Phi_{\beta,\varepsilon}\,dx\right| \\
    &\le
    C
    \int_{\Omega} |u| |\partial_t u| 
    \Psi^{\lambda-1} \,dx\\
    &\le
    C 
    \left( \int_{\Omega} a(x) |u|^2 \Psi^{\lambda -1} \,dx \right)^{1/2}
    \left( \int_{\Omega} a(x)^{-1} |\partial_t u|^2 
    \Psi^{\lambda-1} \,dx \right)^{1/2} \\
    &\le
    \frac{\eta_1}{2}
    \int_{\Omega} a(x) |u|^2 \Psi^{\lambda -1} \,dx
    + C
    \int_{\Omega} |\partial_t u|^2
    \Psi^{\lambda} \,dx
\end{align}
with some
$C = C(n,a,\delta,\varepsilon,\lambda) > 0$.
Summarizing the above computations,
we see that for both cases
$\lambda = 0$ and $\lambda > 0$,
the last two terms of \eqref{eq:1:E0est1}
can be estimated as
\begin{align}
    &- 
    2(1-2\delta)
    \int_{\Omega}
    u \partial_t u
    \Phi_{\beta,\varepsilon}^{-2+2\delta}
    \partial_t \Phi_{\beta,\varepsilon}\,dx \\
    &\quad - (1-2\delta)
    \int_{\Omega}
        |u|^2
    \Phi_{\beta,\varepsilon}^{-2+2\delta}
    \left(
    a(x) \partial_t \Phi_{\beta,\varepsilon}
    - \Delta \Phi_{\beta,\varepsilon}
    \right)
    \,dx \\
    &\le 
    C
    \int_{\Omega} |\partial_t w|^2
    \Psi^{\lambda} \,dx.
\end{align}
Finally, from Proposition \ref{prop:super-sol}
(ii) and (iii), one obtains
\begin{align}
    2 \int_{\Omega} |\partial_t u|^2
    \Phi_{\beta,\varepsilon}^{-1+2\delta}\,dx
    &\le
    C \int_{\Omega} |\partial_t u|^2 \Psi^{\lambda} \,dx
\end{align}
and
\begin{align}
    \frac{2\delta}{1-\delta} \int_{\Omega}
    |\nabla u|^2 \Phi_{\beta,\varepsilon}^{-1+2\delta} \,dx
    +
    2 \int_{\Omega} |u|^{p+1} \Phi_{\beta,\varepsilon}^{-1+2\delta} \,dx
    &\ge
    \eta \int_{\Omega}
    \left(
        |\nabla u|^2
        + |u|^{p+1}
    \right)
    \Psi^{\lambda} \,dx
\end{align}
with some positive constants
$C = C(n,\alpha,\delta,\varepsilon,\lambda)$
and
$\eta = \eta (n,\alpha,\delta,\varepsilon,\lambda)$.
Putting this all together,
we deduce from \eqref{eq:1:E0est1} that
\begin{align}
    \frac{d}{dt}E_0(t;t_0,\lambda)
    &\le
    - \eta \int_{\Omega}
    \left( 
    |\nabla u|^2 + |u|^{p+1}
    \right)
    \Psi^{\lambda} \,dx \\
    &\quad +
    C \int_{\Omega} |\partial_t u|^2 \Psi^{\lambda} \,dx,
\end{align}
and the proof is complete.
\end{proof}
Combining Lemmas \ref{lem:1:E1}
and \ref{lem:1:E0},
we have the following estimate
for $E_{\ast}(t;t_0,\lambda,\nu)$.
\begin{lemma}\label{lem:1:E0+E1}
Under the assumptions on Theorem \ref{thm1} (i),
there exist constants
$\nu_{\ast} = \nu_{\ast}(n,a,\delta,\varepsilon,\lambda) \in (0,\nu_0)$
and
$t_2 = t_2(n,a,p,\delta,\varepsilon,\lambda, \nu_{\ast}) \ge 1$
such that for
$t_0 \ge t_2$
and
$t > 0$,
we have
\begin{align}
    &E_{\ast}(t;t_0,\lambda,\nu_{\ast})
    + \int_0^t \int_{\Omega}
    a(x) |\partial_t u(s,x)|^2
    \Psi(s,x;t_0)^{\lambda + \frac{\alpha}{2-\alpha}} \,dx ds \\
    &\quad 
    +
    \int_0^t \int_{\Omega}
    (|\nabla u(s,x)|^2 + |u(s,x)|^{p+1})
    \Psi(s,x;t_0)^{\lambda} \,dxds \\
    &\le
    C E_{\ast}(0;t_0,\lambda,\nu_{\ast})
\end{align}
with some constant
$C = C(n,a,\delta,\varepsilon,\lambda,\nu_{\ast})>0$.
\end{lemma}
\begin{proof}
Let
$\nu \in (0,\nu_0)$,
where
$\nu_0$
is taken so that \eqref{eq:3:en:pos1} holds.
From the definition of
$E_{\ast}(t;t_0,\lambda,\nu)$
and
Lemmas \ref{lem:1:E1} and \ref{lem:1:E0},
one has
\begin{align}
    \frac{d}{dt} E_{\ast}(t;t_0,\lambda,\nu)
    &=
    \frac{d}{dt} E_1(t;t_0,\lambda)
    + \nu \frac{d}{dt} E_0(t;t_0,\lambda) \\
    &\le
    - \frac{1}{2} \int_{\Omega} a(x) |\partial_t u|^2 \Psi^{\lambda + \frac{\alpha}{2-\alpha}} \,dx \\
    &\quad
    + C \int_{\Omega}
    \left(
    |\nabla u|^2
    + |u|^{p+1}
    \right)
    \Psi^{\lambda + \frac{\alpha}{2-\alpha}-1} \,dx \\
    &\quad -
    \nu \eta \int_{\Omega}
    \left( 
    |\nabla u|^2 + |u|^{p+1}
    \right)
    \Psi^{\lambda} \,dx \\
    &\quad +
    C \nu \int_{\Omega} |\partial_t u|^2 \Psi^{\lambda} \,dx
\label{eq:1:estEast1}
\end{align}
for $t_0 \ge t_1$ and $t > 0$,
where $t_1 \ge 1$ is determined in
Lemma \ref{lem:1:E1}.
Noting that
\eqref{assum:a} and \eqref{A2} imply
\begin{align}
    |\partial_t u|^2 \Psi^{\lambda}
    &\le
    C  \langle x \rangle^{-\alpha}
    A_{\varepsilon}(x)^{\frac{\alpha}{2-\alpha}}
    |\partial_t u|^2 \Psi^{\lambda}
    \le
    C a(x) |\partial_t u|^2 \Psi^{\lambda + \frac{\alpha}{2-\alpha}}
\end{align}
with some constant
$C = C(n,a,\alpha,\varepsilon) > 0$,
and taking
$\nu = \nu_{\ast}$
with sufficiently small
$\nu_{\ast} \in (0,\nu_0)$,
we deduce
\begin{align}
    - \frac{1}{2} \int_{\Omega} a(x) |\partial_t u|^2 \Psi^{\lambda + \frac{\alpha}{2-\alpha}} \,dx 
    + 
    C \nu_{\ast}
    \int_{\Omega} |\partial_t u|^2 \Psi^{\lambda} \,dx
    &\le
    - \frac{1}{4} \int_{\Omega} a(x) |\partial_t u|^2 \Psi^{\lambda + \frac{\alpha}{2-\alpha}} \,dx.
\end{align}
Next, by
$\Psi^{\frac{\alpha}{2-\alpha}-1}
\le
(t_0+t)^{\frac{\alpha}{2-\alpha}-1}$
and taking $t_2 \ge t_1$ sufficiently large depending on
$\nu_{\ast}$,
one obtains
\begin{align}
    &C \int_{\Omega}
    \left(
    |\nabla u|^2
    + |u|^{p+1}
    \right)
    \Psi^{\lambda + \frac{\alpha}{2-\alpha}-1} \,dx
    - \nu_{\ast} \eta \int_{\Omega}
    \left( 
    |\nabla u|^2 + |u|^{p+1}
    \right)
    \Psi^{\lambda} \,dx \\
    &\le
    - \frac{\nu_{\ast} \eta}{2} \int_{\Omega}
    \left( 
    |\nabla u|^2 + |u|^{p+1}
    \right)
    \Psi^{\lambda} \,dx
\end{align}
for $t_0 \ge t_2$.
Finally, plugging the above estimates
into \eqref{eq:1:estEast1} with
$\nu = \nu_{\ast}$,
we conclude
\begin{align}
    \frac{d}{dt} E_{\ast}(t;t_0,\lambda,\nu_{\ast})
    &\le
    - \frac{1}{4} \int_{\Omega} a(x) |\partial_t u|^2 \Psi^{\lambda + \frac{\alpha}{2-\alpha}} \,dx \\
    &\quad -
    \frac{\nu_{\ast} \eta}{2} \int_{\Omega}
    \left( 
    |\nabla u|^2 + |u|^{p+1}
    \right)
    \Psi^{\lambda} \,dx
\end{align}
for $t_0 \ge t_2$ and $t > 0$.
Integrating it over
$[0,t]$,
we have the desired estimate.
\end{proof}
\begin{lemma}\label{lem:1:tildeE}
Under the assumptions on
Theorem \ref{thm1} (i),
there exists a constant
$t_2 = t_2(n,a,p,\delta,\varepsilon,\lambda) \ge 1$
such that
for $t_0 \ge t_2$
and $t > 0$, we have
\begin{align}
    \tilde{E}(t;t_0,\lambda)
    + \int_{\Omega} a(x) |u(t,x)|^2
    \Psi(t,x;t_0)^{\lambda} \,dx
    &\le
    C I_0[u_0,u_1]
\label{eq:1:lem:1:tildeE}
\end{align}
with some constant 
$C = C(n,a,p,\delta,\varepsilon,\lambda,\nu_{\ast},t_0) > 0$.
\end{lemma}
\begin{proof}
Take the same constants
$\nu_{\ast}$ and $t_2$
as in Lemma \ref{lem:1:E0+E1}.
The integration by parts and
the equation \eqref{dwx} imply
\begin{align}
    \frac{d}{dt} \tilde{E}(t;t_0,\lambda)
    &=
    \int_{\Omega} 
    \left[
        \frac{1}{2}\left(
        |\partial_t u|^2 + |\nabla u|^2
        \right) 
        + \frac{1}{p+1} |u|^{p+1}
    \right]
    \left(
    \Psi + \lambda (t_0+t)
    \right)
    \Psi^{\lambda-1} \,dx \\
    &\quad + 
    (t_0+t) \int_{\Omega}
    \left(
        \partial_t u \partial_t^2 u + \nabla u \cdot \nabla \partial_t u
        + |u|^{p-1}u \partial_t u
    \right) 
    \Psi^{\lambda} \,dx \\
    &=
    \int_{\Omega} 
    \left[
        \frac{1}{2}\left(
        |\partial_t u|^2 + |\nabla u|^2
        \right) 
        + \frac{1}{p+1} |u|^{p+1}
    \right]
    \left(
    \Psi + \lambda (t_0+t)
    \right)
    \Psi^{\lambda-1} \,dx \\
    &\quad -
    (t_0+t) \int_{\Omega}
    a(x) |\partial_t u|^2 \Psi^{\lambda}\,dx
    - \lambda (t_0+t)
    \int_{\Omega}
    \partial_t u (\nabla u \cdot \nabla \Psi) \Psi^{\lambda -1}\,dx.
\end{align}
The last term of the right-hand side is
estimated as
\begin{align}
    - \lambda (t_0+t)
    \int_{\Omega} \partial_t u
    (\nabla u \cdot \nabla \Psi)
    \Psi^{\lambda-1} \,dx
    &\le
    \eta (t_0+t) \int_{\Omega}
    a(x) |\partial_t u|^2 \frac{|\nabla \Psi|^2}{a(x)}
    \Psi^{\lambda-1} \,dx \\
    &\quad +
    C (t_0+t) \int_{\Omega} |\nabla u|^2 \Psi^{\lambda-1} \,dx
\end{align}
for any $\eta > 0$.
Using \eqref{eq:1:Psi:quot}
and
taking
$\eta = \eta(n,\alpha,\varepsilon)$
sufficiently small,
we have
\begin{align}
    \frac{d}{dt} \tilde{E}(t;t_0,\lambda)
    &\le
    C \int_{\Omega} \left(
    |\partial_t u|^2 + |\nabla u|^2
    +  |u|^{p+1}
    \right)
    (\Psi + (t_0+t) )
    \Psi^{\lambda-1}
    \,dx \\
    &\quad
    - \frac{1}{2}(t_0+t)
    \int_{\Omega} 
    a(x) |\partial_t u|^2
    \Psi^{\lambda} \,dx.
\end{align}
Noting
$t_0 + t \le \Psi$
and
$a(x)^{-1} \le C \Psi^{\frac{\alpha}{2-\alpha}}$,
we estimate
\begin{align}
    \int_{\Omega} |\partial_t u|^2 
    (\Psi + \lambda(t_0+t))
    \Psi^{\lambda-1} \,dx
    \le C \int_{\Omega} a(x) |\partial_t u|^2 \Psi^{\lambda+\frac{\alpha}{2-\alpha}} \,dx.
\end{align}
Therefore,
integrating over $[0,t]$
yield
\begin{align}
    &\tilde{E}(t;t_0,\lambda)
    + \frac{1}{2} \int_0^t (t_0+s)
    \int_{\Omega} 
    a(x) |\partial_t u|^2
    \Psi^{\lambda} \,dx ds \\
    &\le
    \tilde{E}(0;t_0,\lambda)
    +
    C \int_0^t \int_{\Omega} a(x) |\partial_t u|^2 
    \Psi^{\lambda + \frac{\alpha}{2-\alpha}}
    \,dxds
    + 
    C \int_0^t \int_{\Omega}
    \left(
        |\nabla u|^2 + |u|^{p+1}
    \right)
    \Psi^{\lambda}
    \,dxds.
\end{align}
Now, we multiply the both sides of above inequality
by a sufficiently small constant
$\mu > 0$,
and add it and the conclusion of
Lemma \ref{lem:1:E0+E1}.
Then, we obtain
\begin{align}
    &\mu \tilde{E}(t;t_0,\lambda)
    + E_{\ast}(t;t_0,\lambda,\nu_{\ast}) \\
    &\quad +
    \int_0^t
    \int_{\Omega} 
    a(x) |\partial_t u|^2
    \left[
    \frac{\mu}{2}(t_0+s) +
    (1-C\mu) \Psi^{\frac{\alpha}{2-\alpha}}
    \right]
    \Psi^{\lambda} \,dx ds \\
    &\quad +
    (1-C\mu)
    \int_0^t \int_{\Omega}
    \left(
        |\nabla u|^2 + |u|^{p+1}
    \right)
    \Psi^{\lambda}
    \,dxds \\
    &\le
    \mu \tilde{E}(0;t_0,\lambda)
    + C E_{\ast} (0;t_0,\lambda,\nu_{\ast})
\label{eq:1:tildeE+East}
\end{align}
for $t_0 \ge t_2$ and $t > 0$.
Let us take
$\mu$ sufficiently small
so that
$1 - C\mu > 0$.
Then, the last three terms in the left-hand side can be dropped.
Finally, from the definitions of
$E_{\ast}(t;t_0,\lambda)$
and
$\tilde{E}(t;t_0,\lambda)$,
we can easily verify
\begin{align}
    \mu \tilde{E}(0;t_0,\lambda)
    + E_{\ast} (0;t_0,\lambda,\nu_{\ast})
    \le
    C I_0[u_0,u_1]
\end{align}
with some constant
$C = C (a,p,\lambda,t_0)>0$.
Thus, we conclude
\begin{align}
    &\tilde{E}(t;t_0,\lambda)
    + E_{\ast} (t;t_0,\lambda,\nu_{\ast})
    \le
    C I_0[u_0,u_1]
\end{align}
for $t_0 \ge t_2$ and $t > 0$.
This and the lower bound \eqref{eq:3:en:pos2}
of
$E_{\ast}(t;t_0,\lambda,\nu_{\ast})$
give the desired estimate.
\end{proof}
\begin{proof}[Proof of Theorem \ref{thm1} (i) for compactly supported initial data]
Take $\lambda \in [0,\frac{n-\alpha}{2-\alpha})$
as in the assumption \eqref{assum:ini},
and then choose
$\delta, \varepsilon \in (0,1/2)$
so that
$\lambda \in [0,(1-2\delta)\gamma_{\varepsilon})$
holds.
Moreover, take the same constants
$\nu_{\ast}$ and $t_2$
as in Lemmas \ref{lem:1:E0+E1} and \ref{lem:1:tildeE}.
By
\eqref{eq:3:en:pos2},
Lemmas \ref{lem:1:E0+E1} and \ref{lem:1:tildeE},
Definition \ref{def:1:energy},
and
$(t_0+t)^{\lambda} \le \Psi^{\lambda}$,
we have
\begin{align}
    (t_0+t)^{\lambda+1} E[u](t)
    + (t_0+t)^{\lambda} \int_{\Omega} a(x) |u(t,x)|^2 \,dx
    \le
    C I_0[u_0,u_1]
\label{eq:1:conc}
\end{align}
for $t_0 \ge t_2$ and $t > 0$
with some constant
$C = C(n,a,p,\delta,\varepsilon,\lambda,\nu_{\ast},t_0)>0$.
This completes the proof.
\end{proof}
\begin{remark}
From \eqref{eq:1:tildeE+East},
we have a slightly more general estimate
\begin{align}
    &\int_{\Omega}
    \left(
    |\partial_t u|^2
    + |\nabla u|^2
    + |u|^{p+1}
    \right)
    \left[
    (t_0+t)
    + \Psi^{\frac{\alpha}{2-\alpha}}
    \right] 
    \Psi^{\lambda}
    + \int_{\Omega} a(x) |u|^2 \Psi^{\lambda} \,dx
    \\
    &\quad +
    \int_0^t
    \int_{\Omega} 
    a(x) |\partial_t u|^2
    \left[
    (t_0+s)
    + \Psi^{\frac{\alpha}{2-\alpha}}
    \right]
    \Psi^{\lambda} \,dx ds \\
    &\quad +
    \int_0^t \int_{\Omega}
    \left(
        |\nabla u|^2 + |u|^{p+1}
    \right)
    \Psi^{\lambda}
    \,dxds \\
    &\le
    C I_0[u_0,u_1]
\end{align}
for $t_0 \ge t_2$ and $t > 0$.
Moreover, from the proof of Lemma \ref{lem:1:E0},
we can add the term
$\int_{0}^t \int_{\Omega}
a(x) |u|^2 \Psi^{\lambda-1} \,dxds$
to the left-hand side
when
$\lambda > 0$.
\end{remark}

\subsection{Proof for the general case}
Here, we give a proof of Theorem \ref{thm1} (i)
for non-compactly supported initial data.

Let
$(u_0,u_1) \in H_0^1(\Omega)\times L^2(\Omega)$
satisfy
$I_0[u_0,u_1] < \infty$
and let
$u$ be the corresponding mild solution to \eqref{dwx}.
We take a cut-off function
$\chi \in C_0^{\infty}(\mathbb{R}^n)$
such that 
\begin{align}
    0 \le \chi (x) \le 1 \ (x \in \mathbb{R}^n),
    \quad
    \chi(x)
    = \begin{cases}
        1 &(|x| \le 1),\\
        0 &(|x| \ge 2).
    \end{cases}
\end{align}
For each $j \in \mathbb{N}$,
we define
$\chi_j(x) = \chi(x/j)$.
Then, we have
\begin{align}
    &0 \le \chi_j (x) \le 1 \ (x \in \mathbb{R}^n),
    \quad
    \chi_j(x)
    = \begin{cases}
        1 &(|x| \le j),\\
        0 &(|x| \ge 2j),
    \end{cases}\\
    &
    |\nabla \chi_j(x)| \le \frac{C}{j} \ (x\in \mathbb{R}^n),
    \quad
    \supp \nabla \chi_j
    \subset \overline{B_{2j}(0) \setminus B_j(0)},
\end{align}
where the constant
$C$
is independent of $j$.

Let
$(u_0^{(j)}, u_1^{(j)}) = (\chi_j u_0, \chi_j u_1)$
and let
$u^{(j)}$
be the corresponding mild solution to \eqref{dwx}.
First, by definition, it is easily seen that
\begin{align}
    \lim_{j\to \infty} (u_0^{(j)}, u_1^{(j)})
    = (u_0, u_1)
    \quad \text{in} \quad H_0^1(\Omega) \times L^2(\Omega).
\label{eq:1:limuj}
\end{align}
Therefore, the continuous dependence on the initial data
(see Section \ref{sec:conti:dep})
yields
\begin{align}
    \lim_{j\to \infty} (u^{(j)}(t), \partial_t u^{(j)}(t))
    = (u(t), \partial_t u(t))
    \quad \text{in} \quad
    C([0,T]; H_0^1(\Omega))\cap C^1([0,T]; L^2(\Omega))
\end{align}
for any fixed
$T > 0$.
From this and the Sobolev embedding,
we deduce
\begin{align}\label{eq:3:contidep}
    \lim_{j\to \infty} E[u^{(j)}](t)
    = E[u](t)
\end{align}
for any $t \ge 0$.

We next show 
\begin{align}
    \lim_{j\to \infty} I_0[u_0^{(j)},u_1^{(j)}]
    = I_0[u_0,u_1].
\label{eq:1:limI0}
\end{align}
To prove this, we use the notation
\begin{align}
    &I_0[u_0,u_1; D] \\
    &:= \int_{D} \left[
    ( |u_1(x)|^2 + |\nabla u_0(x)|^2 + |u_0(x)|^{p+1} ) \langle x \rangle^{\alpha} + |u_0(x)|^2 \langle x \rangle^{-\alpha}
    \right] \langle x \rangle^{\lambda(2-\alpha)} \,dx 
\end{align}
for a region
$D \subset \Omega$.
Using the properties of
$\chi_j$ described above and
\begin{align}
    |\nabla (\chi_j u_0)|^2
    = \chi_j^2 |\nabla u_0|^2
    + 2 (\nabla \chi_j \cdot \nabla u_0) \chi_j u_0
    + |\nabla \chi_j|^2 |u_0|^2,
\end{align}
we calculate
\begin{align}
    |I_0[u_0,u_1]
    - I_0[u_0^{(j)}, u_1^{(j)}]|
    &\le
    I_0[u_0,u_1; \Omega\setminus B_j(0)] \\
    &\quad +
    \left| \int_{B_{2j}(0)\setminus B_j(0)}
    2 (\nabla \chi_j \cdot \nabla u_0) \chi_j u_0
    \langle x \rangle^{\alpha + \lambda(2-\alpha)}
    \,dx \right| \\
    &\quad + 
    \int_{B_{2j}(0)\setminus B_j(0)}
    |\nabla \chi_j|^2 |u_0|^2
    \langle x \rangle^{\alpha + \lambda(2-\alpha)}
    \,dx.
\label{eq:1:diffI0}
\end{align}
The Schwarz inequality gives
\begin{align}
    &\left|
    \int_{B_{2j}(0)\setminus B_j(0)}
    2 (\nabla \chi_j \cdot \nabla u_0) \chi_j u_0
    \langle x \rangle^{\alpha + \lambda(2-\alpha)} \,dx
    \right| \\
    &\le
    I_0[u_0,u_1; \Omega\setminus B_j(0)]
    +
    \int_{B_{2j}(0)\setminus B_j(0)}
    |\nabla \chi_j|^2 |u_0|^2
    \langle x \rangle^{\alpha + \lambda(2-\alpha)} \,dx.
\end{align}
Furthermore, using the estimate of
$\nabla \chi_j$,
one sees that
\begin{align}
    &\int_{B_{2j}(0) \setminus B_j(0)}
    |\nabla \chi_j |^2 |u_0|^2
    \langle x \rangle^{\alpha + \lambda(2-\alpha)} \,dx \\
    &\le
    C j^{-2} (1+ |2j|^2 )^{\alpha}
    \int_{B_{2j}(0) \setminus B_j(0)}
    |u_0|^2
    \langle x \rangle ^{-\alpha + \lambda(2-\alpha)} \,dx \\
    &\le
    C I_0[u_0,u_1;\Omega \setminus B_j(0)],
\end{align}
where the constant
$C$
is independent of
$j$.
Putting this all together into \eqref{eq:1:diffI0},
we have
\begin{align}
    |I_0[u_0,u_1]
    - I_0[u_0^{(j)}, u_1^{(j)}]|
    &\le
    C I_0[u_0,u_1; \Omega\setminus B_j(0)].
\end{align}
Since
$I_0[u_0,u_1] < \infty$,
the right-hand side tends to zero
as $j \to \infty$.
This proves \eqref{eq:1:limI0}.

Now we are at the position to
proof Theorem \ref{thm1} (i).

\begin{proof}[Proof of Theorem \ref{thm1} (i) for the general case]
Take the same constant
$t_2$
as in Lemmas \ref{lem:1:E0+E1} and \ref{lem:1:tildeE}.
Let
$\{ (u_0^{(j)}, u_1^{(j)}) \}_{j=1}^{\infty}$
be the sequence defined above and let
$u^{(j)}$
be the corresponding mild solution to \eqref{dwx}
with the initial data
$(u_0^{(j)}, u_1^{(j)})$.
Since each
$(u_0^{(j)}, u_1^{(j)})$
has the compact support,
one can apply the result
\eqref{eq:1:conc} in the previous subsection
to obtain
\begin{align}
    (t_0+t)^{\lambda+1} E[u^{(j)}](t)
    + (t_0+t)^{\lambda} \int_{\Omega} a(x) |u^{(j)}(t,x)|^2 \,dx
    \le
    C I_0[u_0^{(j)},u_1^{(j)}]
\end{align}
for $t_0 \ge t_2$ and $t > 0$.
Finally, using \eqref{eq:3:contidep} and \eqref{eq:1:limI0},
we have
\begin{align}
    (t_0+t)^{\lambda+1} E[u](t)
    + (t_0+t)^{\lambda} \int_{\Omega} a(x) |u(t,x)|^2 \,dx
    \le
    C I_0[u_0,u_1]
\end{align}
for $t_0 \ge t_2$ and $t > 0$,
which completes the proof.
\end{proof}

\section{Proof of Theorem \ref{thm1}: second part}
In this section, we prove Theorem \ref{thm1} (ii).
By the same approximation argument
described in Section \ref{sec:case1},
we may assume
$(u_0,u_1) \in
D(\mathcal{A}_D)$
and consider the strong solution
$u$.

First, we note that,
since the larger $\lambda$ is,
the stronger the assumption on the
initial data is.
Thus, without loss of generality,
we may assume that
$\lambda$
always satisfies
\begin{align}\label{eq:2:rstr:lambda}
    \lambda <
    \min\left\{ \frac{2}{p-1},
    \frac{4}{2-\alpha} \left( \frac{1}{p-1} - \frac{n-\alpha}{4} \right) \right\}
    + \varepsilon,
\end{align}
where 
$\varepsilon > 0$
is a sufficiently small constant
specified later.
This will be used for the estimate
of the remainder term.

In contrast to the previous section,
in the following, we shall use
only 
\begin{align}
\label{eq:2:Psi}
    \Theta(x,t;t_0) := t_0 + t + \langle x \rangle^{2-\alpha}
\end{align}
as a weight function,
and we define the following energies.
\begin{definition}\label{def:2:energy}
For a function
$u = u(t,x)$,
$\alpha \in [0,1)$,
$\lambda \in [0,\infty)$,
$\nu > 0$,
and
$t_0 \ge 1$,
we define
\begin{align}
    E_1(t;t_0,\lambda) &=
    \int_{\Omega} \left[
    \frac{1}{2}\left( |\partial_t u(t,x)|^2 + |\nabla u(t,x)|^2 \right)
    + \frac{1}{p+1} |u(t,x)|^{p+1}
    \right]
    \Theta(t,x;t_0)^{\lambda+\frac{\alpha}{2-\alpha}}
    \,dx,\\
    E_0(t;t_0,\lambda)
    &= \int_{\Omega}
    \left(
    2u(t,x) \partial_t u(t,x) + a(x) |u(t,x)|^2 
    \right)
    \Theta(t,x;t_0)^{\lambda} \,dx,\\
    E_{\ast}(t;t_0,\lambda,\nu)
    &=
    E_1(t;t_0,\lambda)
    + \nu E_0(t;t_0,\lambda),\\
    \tilde{E}(t;t_0,\lambda)
    &=
    (t_0+t) \int_{\Omega} \left[
    \frac{1}{2}\left( |\partial_t u(t,x)|^2 + |\nabla u(t,x)|^2 \right)
    + \frac{1}{p+1} |u(t,x)|^{p+1}
    \right]
    \Theta(t,x;t_0)^{\lambda}
    \,dx
\end{align}
for $t \ge 0$.
\end{definition}

Similarly to 
\eqref{eq:3:en:pos1}
and \eqref{eq:3:en:pos2},
we can prove the lower bound
\begin{align}\label{eq:4:en:pos}
    E_{\ast} (t;t_0,\lambda,\nu)
    \ge
    \frac{1}{2} E_1(t;t_0,\lambda)
    + \frac{\nu}{2}
    \int_{\Omega} a(x) |u(t,x)|^2
    \Theta(t,x;t_0)^{\lambda} \,dx,
\end{align}
provided that
$\nu \in (0,\nu_0)$
with some constant
$\nu_0 > 0$.

We start with the following
simple estimates for
$E_1(t;t_0,\lambda)$
and
$E_0(t;t_0,\lambda)$.

\begin{lemma}\label{lem:2:E1}
Under the assumptions on Theorem \ref{thm1} (ii),
there exists
$t_1 = t_1(n,\alpha,a_0,\lambda,\varepsilon) \ge 1$
such that for
$t_0 \ge t_1$
and $t > 0$, we have
\begin{align}
    \frac{d}{dt} E_1(t;t_0,\lambda) 
    &\le
    - \frac{1}{2} \int_{\Omega} a(x) |\partial_t u(t,x)|^2 \Theta(t,x;t_0)^{\lambda + \frac{\alpha}{2-\alpha}} \,dx \\
    &\quad +
    C \int_{\Omega}
    \left(
        |\nabla u(t,x)|^2
        + |u(t,x)|^{p+1}
    \right)
    \Theta(t,x;t_0)^{\lambda + \frac{\alpha}{2-\alpha}-1} \,dx
\end{align}
with some constant
$C = C(n,\alpha,a_0,p,\lambda) > 0$.
\end{lemma}
\begin{proof}
The proof is almost the same as
that of
Lemma \ref{lem:1:E1}.
The only differences are the use of
\begin{align}\label{eq:4:Theta:quot}
    \frac{|\nabla \Theta|^2}{a(x) \Theta}
    &=
    (2-\alpha)^2 \frac{\langle x\rangle^{-2\alpha}|x|^2}{a(x)(t_0+t+\langle x \rangle^{2-\alpha})}
    \le
    \frac{(2-\alpha)^2}{a_0}
\end{align}
and
\begin{align}
    \Theta(t,x;t_0)^{-1}
    &\le
    t_0^{-1+\frac{\alpha}{2-\alpha}}
    \langle x \rangle^{-\alpha}
    \le
    \frac{1}{a_0} t_0^{-1+\frac{\alpha}{2-\alpha}}
    a(x)
\end{align}
instead of
\eqref{eq:1:Psi:quot}
and \eqref{eq:1:Psi:-1},
respectively.
Thus, we omit the detail.
\end{proof}
\begin{lemma}\label{lem:2:E0}
Under the assumptions on Theorem \ref{thm1} (ii),
for
$t_0 \ge 1$
and $t >0$,
we have
\begin{align}
    \frac{d}{dt} E_0(t;t_0,\lambda) 
    &\le
    - \int_{\Omega} |\nabla u(t,x)|^2
    \Theta(t,x;t_0)^{\lambda} \,dx
    -  2 \int_{\Omega} |u(t,x)|^{p+1}
    \Theta(t,x;t_0)^{\lambda} \,dx \\
    &\quad +
    C \int_{\Omega} a(x) |\partial_t u(t,x)|^2
    \Theta(t,x;t_0)^{\lambda + \frac{\alpha}{2-\alpha}} \,dx
    +
    C\int_{\Omega} a(x) |u(t,x)|^2
    \Theta(t,x;t_0)^{\lambda-1} \,dx
\end{align}
with some constant
$C = C(n,\alpha,a_0,\lambda) > 0$.
\end{lemma}
\begin{proof}
The equation \eqref{dwx} and
the integration by parts imply
\begin{align}
    \frac{d}{dt} E_0(t;t_0,\lambda)
    &=
    2 \int_{\Omega} |\partial_t u|^2 \Theta^{\lambda} \,dx
    +
    2 \int_{\Omega} \left( \partial_t^2 u + a(x) \partial_t u \right) \Theta^{\lambda} \,dx \\
    &\quad
    + \lambda \int_{\Omega} \left( 2 u \partial_t u + a(x) |u|^2 \right) \Theta^{\lambda-1} \,dx \\
    &=
    2 \int_{\Omega} |\partial_t u|^2 \Theta^{\lambda} \,dx
    + 2 \int_{\Omega} \left( \Delta u - |u|^{p-1} u \right) u \Theta^{\lambda} \,dx\\
    &\quad
    + \lambda \int_{\Omega} \left( 2 u \partial_t u + a(x) |u|^2 \right) \Theta^{\lambda-1} \,dx \\
    &= - 2 \int_{\Omega} |\nabla u|^2 \Theta^{\lambda} \,dx
    -  2 \int_{\Omega} |u|^{p+1} \Theta^{\lambda} \,dx \\
    &\quad + 2 \int_{\Omega} |\partial_t u|^2 \Theta^{\lambda} \,dx
    -2 \lambda \int_{\Omega} (\nabla u \cdot \nabla \Theta) u
    \Theta^{\lambda -1} \,dx \\
    &\quad + \lambda \int_{\Omega}
    \left(
        2u \partial_t u + a(x) |u|^2
    \right)
    \Theta^{\lambda-1} \,dx.
\label{eq:2:lemE0}
\end{align}
Let us estimates the right-hand side.
Applying the Schwarz inequality and
\eqref{eq:4:Theta:quot},
we obtain
\begin{align}
     -2 \lambda \int_{\Omega} (\nabla u \cdot \nabla \Psi) u
    \Theta^{\lambda -1} \,dx
    &\le
    \frac{1}{2} \int_{\Omega} |\nabla u|^2 \Theta^{\lambda} \,dx
    + C \int_{\Omega} |u|^2 |\nabla \Theta|^2 \Theta^{\lambda-2} \,dx \\
    &\le
    \frac{1}{2} \int_{\Omega} |\nabla u|^2 \Theta^{\lambda} \,dx
    + C \int_{\Omega} a(x) |u|^2 \Theta^{\lambda-1} \,dx.
\end{align}
Moreover, the Schwarz inequality and
$\Theta^{-1} \le \dfrac{1}{a_0} a(x)$
imply
\begin{align}
    \lambda \int_{\Omega}
        2u(t,x) \partial_t u(t,x) 
        \Theta^{\lambda-1} \,dx
    &\le
    \frac{1}{2} \int_{\Omega} |\nabla u|^2 \Theta^{\lambda} \,dx
    + C \int_{\Omega} |u|^2 \Theta^{\lambda-2}\,dx \\
    &\le
    \frac{1}{2} \int_{\Omega} |\nabla u|^2 \Theta^{\lambda} \,dx
    + C \int_{\Omega} a(x) |u|^2 \Theta^{\lambda-1} \,dx.
\end{align}
From
$1 \le \dfrac{1}{a_0} a(x) \Theta^{\frac{\alpha}{2-\alpha}}$,
we also obtain
\begin{align}
    2\int_{\Omega} |\partial_t u|^2 \Theta^{\lambda}\,dx
    \le
    C \int_{\Omega} a(x) |\partial_t u|^2 \Theta^{\lambda+\frac{\alpha}{2-\alpha}} \,dx.
\end{align}
Putting them all together into \eqref{eq:2:lemE0},
we conclude
\begin{align}
    \frac{d}{dt} E_0(t;t_0,\lambda)
    &\le
    - \int_{\Omega} |\nabla u|^2 \Theta^{\lambda} \,dx
    -  2 \int_{\Omega} |u|^{p+1} \Theta^{\lambda} \,dx \\
    &\quad + C \int_{\Omega} a(x) |\partial_t u|^2 \Theta^{\lambda + \frac{\alpha}{2-\alpha}} \,dx
    + C\int_{\Omega} a(x) |u|^2 \Theta^{\lambda-1} \,dx.
\end{align}
This completes the proof.
\end{proof}

Combining Lemmas \ref{lem:2:E1} and \ref{lem:2:E0},
we have the following.
\begin{lemma}\label{lem:E0+E1}
Under the assumptions on Theorem \ref{thm1} (ii),
there exist constants
$\nu_{\ast} = \nu_{\ast}(n,\alpha,a_0,\lambda) \in (0,\nu_0)$
and
$t_2 = t_2(n,\alpha,a_0,p,\lambda,\nu_{\ast}) \ge 1$
such that for
$t_0 \ge t_2$,
and $t >0$,
we have
\begin{align}
    &E_{\ast} (t;t_0,\lambda,\nu_{\ast})
    + 
    \int_0^t
    \int_{\Omega} a(x) |\partial_t u(s,x)|^2 \Theta(s,x;t_0)^{\lambda + \frac{\alpha}{2-\alpha}} \,dx \,ds \\
    &\quad
    + \int_0^t \int_{\Omega}
    \left(
    |\nabla u(s,x)|^2
    + |u(s,x)|^{p+1}
    \right)
    \Theta(s,x;t_0)^{\lambda} \,dxds\\
    &\le
    C E_{\ast}(0;t_0,\lambda,\nu)
    + C \int_0^t \int_{\Omega}
        a(x) |u(s,x)|^2
    \Theta(s,x;t_0)^{\lambda-1} \,dxds
\end{align}
with some constant
$C = C(n,\alpha,a_0,p,\lambda,\nu_{\ast})>0$.
\end{lemma}
\begin{proof}
Let $\nu \in (0,\nu_0)$,
where
$\nu_0$
is taken so that \eqref{eq:4:en:pos}
holds.
Let
$t_1$ be the constant determined by Lemma \ref{lem:2:E1}.
Then, by Lemmas \ref{lem:2:E1} and \ref{lem:2:E0},
we obtain for
$t_0 \ge t_1$
and
$t > 0$,
\begin{align}
    \frac{d}{dt} E_{\ast}(t;t_0,\lambda,\nu)
    &=
    \frac{d}{dt} E_1(t;t_0,\lambda)
    + \nu \frac{d}{dt} E_0(t;t_0,\lambda) \\
    &\le
    - \frac{1}{2} \int_{\Omega} a(x) |\partial_t u|^2 \Theta^{\lambda + \frac{\alpha}{2-\alpha}} \,dx \\
    &\quad + C \int_{\Omega} |\nabla u|^2
    \Theta^{\lambda + \frac{\alpha}{2-\alpha}-1} \,dx
    + C \int_{\Omega} |u|^{p+1}
    \Theta^{\lambda + \frac{\alpha}{2-\alpha}-1} \,dx \\
    &\quad
    - \nu \int_{\Omega} |\nabla u|^2 \Theta^{\lambda} \,dx
    -  2 \nu \int_{\Omega} |u|^{p+1} \Theta^{\lambda} \,dx \\
    &\quad + C \nu \int_{\Omega} a(x) |\partial_t u|^2 \Theta^{\lambda + \frac{\alpha}{2-\alpha}} \,dx
    + C \nu \int_{\Omega} a(x) |u|^2 \Theta^{\lambda-1} \,dx.
\end{align}
We take
$\nu =\nu_{\ast}$
with sufficiently small
$\nu_{\ast} \in (0,\nu_0)$
such that
the constants in front of the last two terms satisfy
$C\nu_{\ast} < \frac{1}{2}$.
Moreover,
taking
$t_2 > 0$
sufficiently large depending on
$\nu_{\ast}$
so that
$C \Theta^{\frac{\alpha}{2-\alpha}-1} < \nu_{\ast}$
for $t_0 \ge t_2$,
we conclude
\begin{align}
    \frac{d}{dt} E_{\ast}(t;t_0,\lambda,\nu)
    &\le
    - \eta  \int_{\Omega} a(x) |\partial_t u|^2 \Theta^{\lambda + \frac{\alpha}{2-\alpha}} \,dx
    -\eta
    \int_{\Omega} |\nabla u|^2
    \Theta^{\lambda} \,dx \\
    &\quad
    -\eta
    \int_{\Omega} |u|^{p+1}
    \Theta^{\lambda} \,dx
    + C \nu \int_{\Omega} a(x) |u|^2 \Theta^{\lambda-1} \,dx
\end{align}
with some constant
$\eta = \eta(n,\alpha,a_0,p,\lambda,\nu_{\ast}) > 0$.
Finally, integrating the above inequality over
$[0,t]$
gives the desired estimate.
\end{proof}

Besed on Lemma \ref{lem:E0+E1},
we show the following estimate
for
$\tilde{E}(t;t_0,\lambda)$.
\begin{lemma}\label{lem:Etilde}
Under the assumptions on Theorem \ref{thm1} (ii),
there exists a constant
$t_2 = t_2(n,\alpha,a_0,p,\lambda) \ge 1$
such that for
$t_0 \ge t_2$
and $t >0$,
we have
\begin{align}
    &\tilde{E}(t;t_0,\lambda)
    + \int_{\Omega} a(x) |u(t,x)|^2 \Theta(t,x;t_0)^{\lambda} \,dx \\
    &\quad + 
    \int_0^t
    \int_{\Omega} a(x) |\partial_t u(s,x)|^2
    \left[
    (t_0 + s) + \Theta(s,x;t_0)^{\frac{\alpha}{2-\alpha}}
    \right]
    \Theta(s,x;t_0)^{\lambda} \,dx \,ds \\
    &\quad
    + \int_0^t \int_{\Omega}
    \left(
    |\nabla u(s,x)|^2
    + |u(s,x)|^{p+1}
    \right)
    \Theta(s,x;t_0)^{\lambda} \,dxds\\
    &\le
    CI_0[u_0,u_1]
    + C \int_0^t \int_{\Omega}
    a(x)^{\frac{p+1}{p-1}}
    \Theta(s,x;t_0)^{\lambda - \frac{p+1}{p-1}} \,dxds
\end{align}
with some constant
$C = C(n,\alpha,a_0,a_1,p,\lambda,t_0) > 0$.
\end{lemma}
\begin{proof}
Take the same constants
$\nu_{\ast}$ and $t_2$
as in Lemma \ref{lem:E0+E1}.
By the same computation as in
Lemma \ref{lem:1:tildeE},
we can obtain
\begin{align}
    &\tilde{E}(t;t_0,\lambda)
    + \frac{1}{2} \int_0^t (t_0+s)
    \int_{\Omega} 
    a(x) |\partial_t u|^2
    \Theta^{\lambda} \,dx ds \\
    &\le
    \tilde{E}(0;t_0,\lambda)
    +
    C \int_0^t \int_{\Omega} a(x) |\partial_t u|^2 
    \Theta^{\lambda + \frac{\alpha}{2-\alpha}}
    \,dxds
    + 
    C \int_0^t \int_{\Omega}
    \left(
        |\nabla u|^2 + |u|^{p+1}
    \right)
    \Theta^{\lambda}
    \,dxds.
\end{align}
We multiply the both sides
by a sufficiently small constant
$\mu >0$,
and add it and the
conclusion of Lemma \ref{lem:E0+E1}.
Then, we obtain
\begin{align}
    &\mu \tilde{E}(t;t_0,\lambda)
    + E_{\ast}(t;t_0,\lambda,\nu_{\ast}) \\
    &\quad +
    \int_0^t
    \int_{\Omega} 
    a(x) |\partial_t u|^2
    \left[
    \frac{\mu}{2}(t_0+s) +
    (1-C\mu) \Theta^{\frac{\alpha}{2-\alpha}}
    \right]
    \Theta^{\lambda} \,dx ds \\
    &\quad +
    (1-C\mu)
    \int_0^t \int_{\Omega}
    \left(
        |\nabla u|^2 + |u|^{p+1}
    \right)
    \Theta^{\lambda}
    \,dxds \\
    &\le
    \mu \tilde{E}(0;t_0,\lambda)
    + C E_{\ast} (0;t_0,\lambda,\nu_{\ast})
\end{align}
for $t_0 \ge t_2$ and $t > 0$.
By taking
$\mu$
sufficiently small so that
$1-C\mu > 0$
holds,
the terms including
$|\partial_t u|^2$
and $|\nabla u|^2$
in
the left-hand side can be dropped.
Since
both
$\tilde{E}(0;t_0,\lambda)$
and
$E_{\ast}(0;t_0,\lambda,\nu_{\ast})$
are bounded by
$C I_0[u_0,u_1]$
with some constant
$C = C(a_1,p,\lambda,t_0)>0$,
one obtains
\begin{align}
    &\tilde{E}(t;t_0,\lambda)
    + \int_{\Omega} a(x) |u(t,x)|^2 \Theta(t,x;t_0)^{\lambda} \,dx
    + \int_0^t \int_{\Omega}
    |u|^{p+1}
    \Theta^{\lambda}
    \,dxds
    \\
    &\le
    CI_0[u_0,u_1]
    + C \int_0^t \int_{\Omega}
        a(x) |u|^2
    \Theta^{\lambda-1} \,dxds
\label{eq:2:est:Etilde}
\end{align}
with some
$C = C(n,\alpha,a_0,a_1,p,\lambda,t_0) > 0$.
Finally, applying the Young inequality to the last term of the right-hand side,
we deduce
\begin{align}
    C\int_0^t \int_{\Omega}
        a(x) |u|^2
    \Theta^{\lambda-1} \,dxds
    &=
    C \int_0^t \int_{\Omega}
        |u|^2 \Theta^{\frac{2}{p+1}\lambda}
        \cdot
        a(x) \Theta^{\lambda(1-\frac{2}{p+1})-1}
        \,dxds \\
    &\le
    \frac{1}{2} \int_0^t \int_{\Omega} |u|^{p+1} \Theta^{\lambda} \,dxds
    +
    C \int_0^t \int_{\Omega}
    a(x)^{\frac{p+1}{p-1}}
    \Theta^{\lambda - \frac{p+1}{p-1}} \,dxds.
\end{align}
This and \eqref{eq:2:est:Etilde} give
the conclusion.
\end{proof}
By virtue of Lemma \ref{lem:Etilde},
it suffices to estimate the term
\begin{align}
    C \int_0^t \int_{\Omega}
    a(x)^{\frac{p+1}{p-1}}
    \Theta(s,x;t_0)^{\lambda - \frac{p+1}{p-1}} \,dxds.
\end{align}
For this, we have the following lemma.
\begin{lemma}\label{lem:2:badterm}
Under the assumptions on Theorem \ref{thm1} (ii) and \eqref{eq:2:rstr:lambda},
we have for any
$t_0 > 0$
and
$t \ge 0$,
\begin{align}
    &\int_0^t \int_{\Omega}
    a(x)^{\frac{p+1}{p-1}}
    \Theta(s,x;t_0)^{\lambda - \frac{p+1}{p-1}} \,dxds \\
    &\le
    C \begin{cases}
    1
        &(\lambda < \min\{ \frac{4}{2-\alpha}(\frac{1}{p-1} - \frac{n-\alpha}{4} ), \frac{2}{p-1} \}),\\
    \log (t_0+t)
        &(\lambda = \min\{ \frac{4}{2-\alpha}(\frac{1}{p-1} - \frac{n-\alpha}{4} ), \frac{2}{p-1} \}, \ p \neq p_{subc}(n,\alpha)),\\
    (\log (t_0+t))^2
        &(\lambda = \frac{4}{2-\alpha}(\frac{1}{p-1} - \frac{n-\alpha}{4} ) = \frac{2}{p-1}, \ \mathrm{i.e.,} \ p = p_{subc}(n,\alpha)),\\
    (1+t)^{\lambda -\frac{4}{2-\alpha} ( \frac{1}{p-1} - \frac{n-\alpha}{4})}
        &(\lambda > \frac{4}{2-\alpha}(\frac{1}{p-1} - \frac{n-\alpha}{4} ),
        \ p > p_{subc}(n,\alpha)),\\
    (1+t)^{\lambda -\frac{2}{p-1}} \log (t_0+t)
        &(\lambda > \frac{2}{p-1},
        \ p = p_{subc}(n,\alpha)),\\
    (1+t)^{\lambda -\frac{2}{p-1}}
        &(\lambda > \frac{2}{p-1}, 
        \ p<p_{subc}(n,\alpha))
    \end{cases}
\end{align}
with some constant
$C = C(n,\alpha,a_1,p,\lambda) >0$.
\end{lemma}
\begin{proof}
Let 
$s \in (0,t)$.
First, we divide
$\Omega$
into
$\Omega = \Omega_1(s) \cup \Omega_2(s)$,
where
\begin{align}
    \Omega_1(s)
    &=
    \left\{ x \in \Omega ;\, 
    \langle x \rangle^{2-\alpha} \le t_0 + s
    \right\},\\
    \Omega_2(s)
    &=
    \Omega \setminus \Omega_1(s)
    =
    \left\{ x \in \Omega ;\, 
    \langle x \rangle^{2-\alpha} > t_0 + s
    \right\}.
\end{align}
The corresponding integral is also decomposed into
\begin{align}
    \int_{\Omega}
    a(x)^{\frac{p+1}{p-1}}
    \Theta(s,x;t_0)^{\lambda - \frac{p+1}{p-1}} \,dx
    &=
    \int_{\Omega_1(s)}
    a(x)^{\frac{p+1}{p-1}}
    \Theta(s,x;t_0)^{\lambda - \frac{p+1}{p-1}} \,dx \\
    &\quad +
    \int_{\Omega_2(s)}
    a(x)^{\frac{p+1}{p-1}}
    \Theta(s,x;t_0)^{\lambda - \frac{p+1}{p-1}} \,dx\\
    &=: I(s) + I\!I(s).
\end{align}
Note that,
in
$\Omega_1(s)$,
the function
$\Theta(s,x;t_0) = t_0+s+\langle x \rangle^{2-\alpha}$
is bounded from both above and below
by $t_0+s$.
Therefore, we estimate
\begin{align}
    I(s)
    &\le
    C (t_0 + s)^{\lambda - \frac{p+1}{p-1}}
    \int_{\Omega_1(s)}  a(x)^{\frac{p+1}{p-1}} \,dx \\
    &\le
    C (t_0 + s)^{\lambda - \frac{p+1}{p-1}}
    \int_{\Omega_1(s)} \langle x \rangle^{-\alpha \frac{p+1}{p-1}} \,dx \\
    &\le
    C (t_0 + s)^{\lambda - \frac{p+1}{p-1}}
    h(s),
\label{eq:2:I(s)}
\end{align}
where
\begin{align}
\label{eq:2:def:h}
    h(s)
    =
    \begin{dcases}
    1 &(p<p_{subc}(n,\alpha)),\\
    \log(t_0+s) &(p = p_{subc}(n,\alpha)),\\
    (t_0+s)^{\frac{1}{2-\alpha} \left( n - \alpha \frac{p+1}{p-1} \right)} &(p> p_{subc}(n,\alpha)).
    \end{dcases}
\end{align}
On the other hand, in 
$\Omega_2(s)$,
the function
$\Theta$
is bounded from both above and below by
$\langle x \rangle^{2-\alpha}$.
Thus, we have
\begin{align}
    I\!I(s)
    &\le 
    C \int_{\Omega_2(s)} \langle x \rangle^{-\alpha\frac{p+1}{p-1} + (2-\alpha)\left( \lambda - \frac{p+1}{p-1} \right)} \,dx.
\end{align}
Here, we remark that the condition
\eqref{eq:2:rstr:lambda}
ensures the finiteness of the above integral,
provided that
$\varepsilon$ is taken sufficiently small
depending on $n$ and $\alpha$.
A straightforward computation shows
\begin{align}
    I\!I(s)
    &\le 
    C (t_0+s)^{\lambda - \frac{p+1}{p-1} + \frac{1}{2-\alpha} \left(n - \alpha \frac{p+1}{p-1} \right)}.
\end{align}
Since the above estimate is
better than \eqref{eq:2:I(s)} if
$p \le p_{subc}(n,\alpha)$
and is the same if
$p > p_{subc}(n,\alpha)$,
we conclude
\begin{align}
    \int_{\Omega}
    a(x)^{\frac{p+1}{p-1}}
    \Theta(s,x;t_0)^{\lambda - \frac{p+1}{p-1}} \,dx
    &\le
    C (t_0 + s)^{\lambda - \frac{p+1}{p-1}}
    h(s).
\label{eq:2:est:bdtrm}
\end{align}
Next, we compute the integral of the function
$(t_0 + s)^{\lambda - \frac{p+1}{p-1}}h(s)$
over 
$[0,t]$.
From the definition \eqref{eq:2:def:h} of $h(s)$,
one has the following:
If
$p<p_{subc}(n,\alpha)$,
then
\begin{align}
    \int_0^t (t_0 + s)^{\lambda - \frac{p+1}{p-1}}h(s) \,ds
    &\le
    C \begin{dcases}
    1
        &\left( \lambda < \frac{2}{p-1} \right),\\
    \log(t_0+t)
        &\left( \lambda = \frac{2}{p-1} \right),\\
    (t_0+t)^{\lambda - \frac{2}{p-1}}
        &\left( \lambda > \frac{2}{p-1} \right);
    \end{dcases}
\end{align}
If
$p = p_{subc}(n,\alpha)$,
then
\begin{align}
    \int_0^t (t_0 + s)^{\lambda - \frac{p+1}{p-1}}h(s) \,ds
    &\le
    C \begin{dcases}
    1
        &\left( \lambda < \frac{2}{p-1} \right),\\
    (\log(t_0+t))^2
        &\left( \lambda = \frac{2}{p-1} \right),\\
    (t_0+t)^{\lambda - \frac{2}{p-1}} \log(t_0+t) 
        &\left( \lambda > \frac{2}{p-1} \right);
    \end{dcases}
\end{align}
If
$p > p_{subc}(n,\alpha)$,
then
\begin{align}
    \int_0^t (t_0 + s)^{\lambda - \frac{p+1}{p-1}}h(s) \,ds
    &\le
    C \begin{dcases}
    1
        &\left( \lambda < \frac{4}{2-\alpha}\left(\frac{1}{p-1}-\frac{n-\alpha}{4} \right) \right),\\
    \log(t_0+t)
        &\left( \lambda = \frac{4}{2-\alpha}\left(\frac{1}{p-1}-\frac{n-\alpha}{4} \right) \right),\\
    (t_0+t)^{\lambda - \frac{4}{2-\alpha}\left(\frac{1}{p-1}-\frac{n-\alpha}{4} \right)}
        &\left( \lambda > \frac{4}{2-\alpha}\left(\frac{1}{p-1}-\frac{n-\alpha}{4} \right) \right).
    \end{dcases}
\end{align}
This completes the proof.
\end{proof}
We are now at the position to prove
Theorem \ref{thm1} (ii):
\begin{proof}[Proof of Theorem \ref{thm1} (ii)]
By Lemmas \ref{lem:Etilde} and \ref{lem:2:badterm}
with the constant
$t_2 \ge 1$
determined in Lemma \ref{lem:Etilde},
we have
\begin{align}
    &\tilde{E}(t;t_0,\lambda)
    + \int_{\Omega} a(x)|u(t,x)|^2 \Theta(t,x;t_0)^{\lambda} \,dx \\
    &\le 
    CI_0[u_0,u_1]
    + C
    \begin{cases}
    1
        &(\lambda < \min\{ \frac{4}{2-\alpha}(\frac{1}{p-1} - \frac{n-\alpha}{4} ), \frac{2}{p-1} \}),\\
    \log (t_0+t)
        &(\lambda = \min\{ \frac{4}{2-\alpha}(\frac{1}{p-1} - \frac{n-\alpha}{4} ), \frac{2}{p-1} \}, \ p \neq p_{subc}(n,\alpha)),\\
    (\log (t_0+t))^2
        &(\lambda = \frac{4}{2-\alpha}(\frac{1}{p-1} - \frac{n-\alpha}{4} ) = \frac{2}{p-1}, \ \mathrm{i.e.,} \ p = p_{subc}(n,\alpha)),\\
    (1+t)^{\lambda -\frac{4}{2-\alpha} ( \frac{1}{p-1} - \frac{n-\alpha}{4})}
        &(\lambda > \frac{4}{2-\alpha}(\frac{1}{p-1} - \frac{n-\alpha}{4} ),
        \ p > p_{subc}(n,\alpha)),\\
    (1+t)^{\lambda -\frac{2}{p-1}} \log (t_0+t)
        &(\lambda > \frac{2}{p-1},
        \ p = p_{subc}(n,\alpha)),\\
    (1+t)^{\lambda -\frac{2}{p-1}}
        &(\lambda > \frac{2}{p-1}, 
        \ p<p_{subc}(n,\alpha))
    \end{cases}
\end{align}
for $t_0 \ge t_2$ 
and $t \ge 0$.
On the other hand,
the definition \eqref{eq:2:Psi} of
$\Theta$
immediately gives the lower bound
\begin{align}
    &\tilde{E}(t;t_0,\lambda)
    + \int_{\Omega} a(x)|u(t,x)|^2 \Theta(t,x;t_0)^{\lambda} \,dx \\
    &\ge
    (t_0+t)^{\lambda + 1} E[u](t)
    + (t_0+t)^{\lambda} \int_{\Omega} a(x)|u(t,x)|^2 \,dx,
\end{align}
where
$E(t)$
is defined by \eqref{eq:def:E}.
Combining them,
we have the desired estimate.
\end{proof}

\appendix
\section{Outline of the proof of Proposition \ref{prop:ex}}
In this section, we give a proof of
Proposition \ref{prop:ex}.
The solvability and basic properties
of the solution of the linear problem
\eqref{ldwx} below can be found in,
for example,
\cite{DaSh95, Ikawa, IkeSo21FE, Nis16}.
Here, we give an outline of the argument
along with \cite{Ikawa}.
The existence of the unique mild solution
of the semilinear problem \eqref{dwx}
is proved by the contraction mapping principle.
This argument can be found in, e.g.,
\cite{CaHa, IkeSo21FE, IkTa05, Strauss}.
Here, we will give a proof based on
\cite{CaHa}.

\subsection{Linear problem}\label{sec:A:1}
Let
$n \in \mathbb{N}$,
and let
$\Omega$
be an open set in
$\mathbb{R}^n$
with a compact $C^2$-boundary
$\partial \Omega$
or
$\Omega = \mathbb{R}^n$.
We discuss the linear problem
\begin{align}
\label{ldwx}
    \left\{ \begin{array}{ll}
        \partial_t^2 u - \Delta u + a(x) \partial_t u = 0, &t>0, x \in \Omega, \\
        u(x,t) = 0, &t>0, x \in \partial \Omega,\\
        u(0,x) = u_0(x), \ 
        \partial_t u(0,x) = u_1(x),
        &x \in \Omega.
        \end{array} \right.
\end{align}
The function
$a(x)$
is nonnegative, bounded, and continuous
in $\mathbb{R}^n$.
Let
$\mathcal{H} : = H^1_0(\Omega) \times L^2(\Omega)$
be the real Hilbert space equipped with the
inner product
\begin{align}
    \left(
    \begin{pmatrix}
    u\\ v
    \end{pmatrix},
    \begin{pmatrix}
    w\\ z
    \end{pmatrix}
    \right)_{\mathcal{H}}
    = (u,w)_{H^1} + (v,z)_{L^2}.
\end{align}
Let
$\mathcal{A}$ be the operator
\begin{align}
    \mathcal{A} =
    \begin{pmatrix} 0 & 1 \\ \Delta & -a(x) \end{pmatrix}
\end{align}
defined on $\mathcal{H}$
with the domain
$D(\mathcal{A}) = (H^2(\Omega) \cap H^1_0(\Omega)) \times H^1_0(\Omega)$,
which is dense in
$\mathcal{H}$.

We first show the estimate
\begin{align}\label{app:1:dissipativeop}
    \left(
    \mathcal{A} \begin{pmatrix}
    u\\ v
    \end{pmatrix},
    \begin{pmatrix}
    u\\ v
    \end{pmatrix}
    \right)_{\mathcal{H}}
    \le
    \left\| (u,v)
    \right\|_{\mathcal{H}}^2
\end{align}
for
$(u,v) \in D(\mathcal{A})$.
Indeed, we calculate
\begin{align}
    \left(
    \mathcal{A} \begin{pmatrix}
    u\\ v
    \end{pmatrix},
    \begin{pmatrix}
    u\\ v
    \end{pmatrix}
    \right)_{\mathcal{H}}
    &=
    \left(
    \begin{pmatrix}
    v\\ \Delta u - a(x) v
    \end{pmatrix},
    \begin{pmatrix}
    u\\ v
    \end{pmatrix}
    \right)_{\mathcal{H}} \\
    &=
    (v,u)_{H^1} + (\Delta u - a(x)v, v)_{L^2} \\
    &=
    (\nabla v, \nabla u)_{L^2}
    + (v,u)_{L^2}
    - (\nabla v, \nabla u)_{L^2}
    - (a(x)v,v)_{L^2} \\
    &\le
    (v,u)_{L^2}
    \le \| (u,v)\|_{\mathcal{H}}^2.
\end{align}
Next, we prove that
there exists
$\lambda_0 \in \mathbb{R}$
such that for any
$\lambda \ge \lambda_0$,
the operator
$\lambda - \mathcal{A}$
is invertible, that is, for any
$(f,g) \in \mathcal{H}$,
we can find a unique
$(u,v) \in D(\mathcal{A})$
satisfying
\begin{align}\label{app:1:maximalop}
    \left( \lambda - \mathcal{A} \right)
    \begin{pmatrix}
    u\\ v
    \end{pmatrix}
    = \begin{pmatrix}
    f\\ g
    \end{pmatrix}.
\end{align}
Indeed, the above equation is equivalent with
\begin{align}
    \begin{dcases}
    \lambda u - v = f, \\
    \lambda v - \Delta u + a(x)v = g.
    \end{dcases}
\end{align}
We remark that the first equation implies
$v = \lambda u - f$.
Substituting this into the second equation,
one has
\begin{align}\label{app:1:ellip}
    (\lambda^2 + \lambda a(x)) u - \Delta u
    = h,
\end{align}
where
$h = g + (\lambda + a(x)) f \in L^2(\Omega)$.
Take an arbitrary constant
$\lambda_0 > 0$
and let
$\lambda \ge \lambda_0$
be fixed.
Associated with the above equation, we define
the bilinear functional
\begin{align}
    \mathfrak{a}(z,w) &=
    ( (\lambda^2 + \lambda a(x))z, w)_{L^2}
    + (\nabla z, \nabla w)_{L^2}
\end{align}
for
$z,w \in H_0^1(\Omega)$.
Since
$\lambda > 0$
and 
$a(x)$
is nonnegative and bounded,
$\mathfrak{a}$
is bounded:
$\mathfrak{a}(z,w) \le C \| z \|_{H^1} \|w \|_{H^1}$,
and coercive:
$\mathfrak{a}(z,z) \ge C \| z \|_{H^1}^2$.
Therefore, by the Lax--Milgram theorem
(see, e.g., \cite[Theorem 1.1.4]{CaHa}),
there exists a unique
$u \in H^1_0(\Omega)$
satisfying
$\mathfrak{a}(u,\varphi) = (h, \varphi)_{H^1}$
for any $\varphi \in H^1_0(\Omega)$.
In particular,
$u$
satisfies the equation \eqref{app:1:ellip}
in the distribution sense.
This shows
$\Delta u \in L^2(\Omega)$,
and hence, a standard elliptic estimate implies
$u \in H^2(\Omega)$
(see, for example, Brezis \cite[Theorem 9.25]{Br}).
Defining $v$ by
$v = \lambda u - f \in H_0^1(\Omega)$,
we find the solution
$(u,v) \in D(\mathcal{A})$
to the equation \eqref{app:1:maximalop}.

The above properties enable us to apply
the Hille--Yosida theorem
(see, e.g., \cite[Theorem 2.18]{Ikawa}),
and there exists a
$C_0$-semigroup
$U(t)$
on $\mathcal{H}$
satisfying the estimate
\begin{align}\label{app:1:sgest}
    \left\| U(t) \begin{pmatrix}
    u_0\\ u_1
    \end{pmatrix}
    \right\|_{\mathcal{H}}
    &\le
    e^{Ct} \| (u_0, u_1) \|_{\mathcal{H}}
\end{align}
with some constant
$C>0$.
Moreover, if
$(u_0, u_1) \in D(\mathcal{A})$,
then
$\mathcal{U}(t) := U(t) \begin{pmatrix}
u_0\\ u_1
\end{pmatrix}$
satisfies
\begin{align}\label{app:1:absteq}
    \frac{d}{dt} \mathcal{U}(t)
    = \mathcal{A} \mathcal{U}(t),
    \quad t > 0.
\end{align}
Therefore, the first component
$u(t)$
of $\mathcal{U}(t)$
satisfies
\begin{align}
    u \in C([0,\infty); H^2(\Omega))
    \cap C^1([0,\infty); H^1_0(\Omega))
    \cap C^2([0,\infty); L^2(\Omega))
\end{align}
and the equation \eqref{ldwx}
in
$C([0,\infty); L^2(\Omega))$.

For
$(u_0, u_1) \in \mathcal{H}$,
let
$\mathcal{U}(t) =
\begin{pmatrix}
u(t)\\ v(t)
\end{pmatrix}
:=
U(t) \begin{pmatrix}
u_0\\ u_1
\end{pmatrix}$.
We next show that
$u$
satisfies
\begin{align}\label{app:1:mildsolclass}
    u \in C([0,\infty); H^1_0(\Omega))
    \cap C^1([0,\infty); L^2(\Omega)).
\end{align}
The property
$u \in C([0,\infty); H^1_0(\Omega))$
is obvious from
$\mathcal{U} \in C([0,\infty); \mathcal{H})$.
In order to prove
$u \in C^1([0,\infty); L^2(\Omega))$,
we employ an approximation argument.
Let
$\{ (u_0^{(j)}, u_1^{(j)}) \}_{j=1}^{\infty}$
be a sequence in
$D(\mathcal{A})$
such that
$\lim_{j\to \infty} (u_0^{(j)}, u_1^{(j)})
= (u_0, u_1)$
in $\mathcal{H}$,
and let
$\mathcal{U}^{(j)}(t) =
\begin{pmatrix}
u^{(j)}\\ v^{(j)}
\end{pmatrix}
:=
U(t) \begin{pmatrix}
u_0^{(j)}\\ u_1^{(j)}
\end{pmatrix}$.
From
$(u_0^{(j)}, u_1^{(j)}) \in D(\mathcal{A})$,
$\mathcal{U}^{(j)}$
satisfies the equation \eqref{app:1:absteq},
and hence,
one obtains
$v^{(j)} = \partial_t u^{(j)}$.
For any fixed $T>0$,
the estimate \eqref{app:1:sgest}
implies
\begin{align}
    \sup_{t\in [0,T]} \| u^{(j)}(t) - u(t) \|_{L^2}
    &\le 
    e^{CT} \| (u_0^{(j)}-u_0, u_1^{(j)}-u_1) \|_{\mathcal{H}}
    \to 0,\\ 
    \sup_{t\in [0,T]} \| \partial_t u^{(j)}(t) - v(t) \|_{L^2}
    &\le 
    e^{CT} \| (u_0^{(j)}-u_0, u_1^{(j)}-u_1) \|_{\mathcal{H}}
    \to 0
\end{align}
as $j\to \infty$.
This shows 
$u \in C^1([0,T]; L^2(\Omega))$
and
$\partial_t u = v$.
Since $T>0$ is arbitrary, 
we obtain \eqref{app:1:mildsolclass}.

\subsection{Semilinear problem}\label{sec:A2}
Let us turn to study the semilinear problem
\eqref{dwx}.

\subsubsection{Uniqueness of the mild solution}
We first show the uniqueness of the mild solution of the integral equation
\begin{align}\label{app:1:inteq}
    \mathcal{U}(t)
    &=
    \begin{pmatrix}
    u(t)\\ v(t)
    \end{pmatrix}
    = U(t) \begin{pmatrix}
    u_0\\ u_1
    \end{pmatrix}
    + \int_0^t U(t-s)
    \begin{pmatrix}
    0\\ -|u(s)|^{p-1}u(s)
    \end{pmatrix}
    \,ds
\end{align}
in
$C([0,T_0); \mathcal{H})$
for arbitrary fixed $T_0>0$.
Hereafter,
as long as there is no risk of confusion,
we call both
$\mathcal{U}$
and the first component
$u$ of $\mathcal{U}$
mild solutions.
Let
$T_0>0$
and
$C_0 = e^{CT_0}$,
where $C$ is the constant in \eqref{app:1:sgest}.
Let
$\mathcal{U}(t) = \begin{pmatrix}
u\\ v
\end{pmatrix}$
and
$\mathcal{W}(t) = \begin{pmatrix}
w\\ z
\end{pmatrix}$
be two solutions to \eqref{app:1:inteq}
in
$C([0,T_0); \mathcal{H})$.
Take $T \in (0,T_0)$ arbitrary and
put
$K := \sup_{t \in [0,T]}
(\| \mathcal{U}(t) \|_{\mathcal{H}}
+ \| \mathcal{W}(t) \|_{\mathcal{H}}$.
Then, the estimate
\eqref{app:1:sgest} implies
\begin{align}
    \| \mathcal{U}(t) - \mathcal{W}(t) \|_{\mathcal{H}}
    &\le C_0 \int_0^t
    \| |w(s)|^{p-1}w(s) - |u(s)|^{p-1}u(s) \|_{L^2} \,ds.
\end{align}
Since the nonlinearity satisfies
\begin{align}
    | |w|^{p-1}w - |u|^{p-1} u |
    &\le
    C (|w|+|u|)^{p-1} |u-w|
\end{align}
and
$p$ 
fulfills the condition \eqref{p},
we apply
the H\"{o}lder and
the Gagliardo--Nirenberg inequality
$\| u \|_{L^{2p}} \le C \|u \|_{H^1}$
to obtain
\begin{align}
    \| \mathcal{U}(t) - \mathcal{W}(t) \|_{\mathcal{H}}
    &\le
    C_0 \int_0^t
    \| |u(s)|^{p-1}u(s) - |w(s)|^{p-1}w(s) \|_{L^2} \,ds \\
    &\le
    C_0 C \int_0^t
    ( \| u(s)\|_{L^{2p}} + \| w(s) \|_{L^{2p}} )^{p-1} \| u(s) - w(s) \|_{L^{2p}} \,ds \\
    &\le
    C_0 C \int_0^t
    ( \| u(s)\|_{H^1} + \| w(s) \|_{H^1} )^{p-1} \| u(s) - w(s) \|_{H^1} \,ds \\
    &\le
    C_0 C K^{p-1}
    \int_0^t
    \| \mathcal{U}(s) - \mathcal{W}(s) \|_{\mathcal{H}} \,ds
\label{app:1:diffest}
\end{align}
for $t \in [0,T]$.
Therefore, by the Gronwall inequality,
we have
$ \| \mathcal{U}(t) - \mathcal{W}(t) \|_{\mathcal{H}} = 0$
for $t \in [0,T]$.
Since $T \in (0,T_0)$
is arbitrary,
we conclude
$\mathcal{U}(t) = \mathcal{W}(t)$
for all $t \in [0,T_0)$.

\subsubsection{Existence of the mild solution}\label{sec:app:ex}
Here, we show the existence of the mild solution.

Let
$T_0 > 0$
be arbitrarily fixed.
For
$T \in (0,T_0)$
and
$\mathcal{U} = \begin{pmatrix}
u\\ v
\end{pmatrix}
\in C([0,T]; \mathcal{H})$,
we define the mapping
\begin{align}
    \mathbf{\Phi} (\mathcal{U})(t) 
    &= U(t) \begin{pmatrix}
    u_0\\ u_1
    \end{pmatrix}
    + \int_0^t U(t-s)
    \begin{pmatrix}
    0\\ -|u(s)|^{p-1}u(s)
    \end{pmatrix}
    \,ds.
\end{align}
Let
$C_0 = e^{CT_0}$,
where $C$ is the constant in \eqref{app:1:sgest}.
Then, we have
\begin{align}
    \left\| U(t) \begin{pmatrix}
    u_0\\ u_1
    \end{pmatrix}
    \right\|_{\mathcal{H}}
    &\le
    C_0 \| (u_0, u_1) \|_{\mathcal{H}}
\end{align}
for $t \in (0,T_0)$.
Let
$K = 2C_0 \| (u_0, u_1) \|_{\mathcal{H}}$
and define
\begin{align}
    M_{T,K} :=
    \left\{
    \mathcal{U} = \begin{pmatrix}
    u\\ v
    \end{pmatrix}
    \in C([0,T]; \mathcal{H}) ;\,
    \sup_{t\in [0,T]}
    \| (u(t), v(t)) \|_{\mathcal{H}}
    \le K
    \right\}.
\end{align}
$M_{T,K}$
is a complete metric space with respect to the metric
\begin{align}
    d(\mathcal{U}, \mathcal{W}) = \sup_{t\in [0,T]}
    \| (u(t)-w(t), v(t)- z(t) )\|_{\mathcal{H}}
\end{align}
for
$\mathcal{U} = \begin{pmatrix}
u\\ v
\end{pmatrix}$
and
$\mathcal{W} = \begin{pmatrix}
w\\ z
\end{pmatrix}$.
We shall prove that
$\Phi$
is the contraction mapping on
$M_{T,R}$,
provided that
$T$
is sufficiently small.

First, we show that
$\mathbf{\Phi} (\mathcal{U}) \in M_{T,K}$
for $\mathcal{U} \in M_{T,K}$.
By the estimate \eqref{app:1:sgest}
and the Gagliardo--Nirenberg inequality,
we obtain for $t \in [0,T]$,
\begin{align}
    \| \mathbf{\Phi}(\mathcal{U})(t) \|_{\mathcal{H}}
    &\le
    \frac{K}{2} + C_0 \int_0^t \| |u(s)|^{p-1}u (s) \|_{L^2} \,ds \\
    &\le
    \frac{K}{2} + C_0 \int_0^t \| u(s) \|_{L^{2p}}^p \,ds \\
    &\le 
    \frac{K}{2} + C_0 C \int_0^t \| u(s) \|_{H^1}^p \,ds \\
    &\le
    \frac{K}{2} + C_0 C T K^p.
\label{app:1:contraction}
\end{align}
Therefore, taking
$T$ sufficiently small so that
\begin{align}
    \frac{K}{2} + C_0 C T K^p \le K
\end{align}
holds,
we see that
$\mathbf{\Phi}(\mathcal{U}) \in M_{T,K}$.
Moreover, for
$\mathcal{U} = \begin{pmatrix}
u\\ v
\end{pmatrix}$,
$\mathcal{W} = \begin{pmatrix}
w\\ z
\end{pmatrix}
\in M_{T,R}$,
the same computation as in \eqref{app:1:diffest}
yields for $t\in [0,T]$,
\begin{align}
    d (\mathbf{\Phi}(\mathcal{U}),
    \mathbf{\Phi}(\mathcal{W}) )
    &\le
    C_0 C T K^{p-1}
    d( \mathcal{U}, \mathcal{W} ).
\end{align}
Thus, retaking $T$ smaller if needed so that
\begin{align}
    C_0 C T K^{p-1} \le \frac{1}{2},
\end{align}
we have the contractivity of
$\mathbf{\Phi}$.
Thus, by the contraction mapping principle,
we see that there exists a fixed point
$\mathcal{U} = \begin{pmatrix}
u\\ v
\end{pmatrix} \in M_{T,K}$,
that is,
$\mathcal{U}$
satisfies the integral equation \eqref{app:1:inteq}.
We postpone to verify
$u \in C^1([0,T]; L^2(\Omega))$
and
$\partial_t u = v$
after proving the approximation property below.

\subsubsection{Blow-up alternative}\label{sec:bu}
Let
$T_{\mathrm{max}} = T_{\mathrm{max}}(u_0, u_1)$
be the maximal existence time of the mild solution defined by
\begin{align}
    T_{\mathrm{max}}
    &= \sup \left\{
    T \in (0,\infty] ;\,
    {}^{\exists} \mathcal{U} =
    \begin{pmatrix}
    u\\ v
    \end{pmatrix}
    \in
    C([0,T) ;\, \mathcal{H})
    \ \mbox{satisfies \eqref{app:1:inteq}}
    \right\}.
\end{align}
We show that
if $T_{\mathrm{max}} < \infty$,
the corresponding unique mild solution
$\mathcal{U} =
\begin{pmatrix}
u\\ v
\end{pmatrix}$
must satisfy
\begin{align}\label{app:bu}
    \lim_{t \to T_{\mathrm{max}}-0}
    \| \mathcal{U}(t) \|_{\mathcal{H}}
    = \infty.
\end{align}
Indeed, if
$m := \liminf_{t \to T_{\mathrm{max}}-0} \| \mathcal{U}(t) \|_{\mathcal{H}} < \infty$,
then there exists a
monotone increasing sequence
$\{ t_j \}_{j=1}^{\infty}$
in
$(0, T_{\mathrm{max}})$
such that 
$\lim_{j\to\infty}t_j = T_{\mathrm{max}}$
and
$\lim_{j\to \infty}\| \mathcal{U}(t_j) \|_{\mathcal{H}} = m$.
Let
$T_0 > T_{\mathrm{max}}$
be arbitrary fixed and let
$C_0 = e^{CT_0}$
as in Section \ref{sec:app:ex}.
Applying the same argument as in
Section \ref{sec:app:ex} with
replacement
$(u_0,u_1)$
by
$\mathcal{U}(t_j)$,
one can find there exists
$T$ depending only on
$p$, $m$, and $C_0$
such that there exists a mild solution
on the interval
$[t_j, t_j + T]$.
However, this contradicts the definition of
$T_{\mathrm{max}}$
when $j$ is large.
Thus, we have \eqref{app:bu}.

\subsubsection{Continuous dependence on the initial data}\label{sec:conti:dep}
Let
$(u_0,u_1) \in \mathcal{H}$
and
$T < T_0 < T_{\mathrm{max}}(u_0,u_1)$.
We take
$C_0 = e^{CT_0}$
as in Section \ref{sec:app:ex}.
Let
$\{ (u_0^{(j)}, u_1^{(j)})\}_{j=1}^{\infty}$
be a sequence in
$\mathcal{H}$
such that
$(u_0^{(j)}, u_1^{(j)}) \to (u_0,u_1)$
in
$\mathcal{H}$
as $j \to \infty$.
Then, we will prove that,
for sufficiently large
$j$,
$T_{\mathrm{max}}(u_0^{(j)}, u_1^{(j)}) > T$
and the corresponding solution
$\mathcal{U}^{(j)}$
with the initial data
$(u_0^{(j)}, u_1^{(j)})$
satisfies
\begin{align}
\label{app:1:contidep}
    \lim_{j\to \infty}
    \sup_{t\in [0,T]}
    \| \mathcal{U}^{(j)}(t) - \mathcal{U}(t) \|_{\mathcal{H}} = 0.
\end{align}
Let
$C_1 = 2 \sup_{t\in [0,T]} \| \mathcal{U}(t) \|_{\mathcal{H}}$
and let
\begin{align}
    \tau_j := \sup \left\{
    t \in [0,T_{\mathrm{max}}(u_0^{(j)},u_1^{(j)})) ;\,
    \sup_{t\in [0,T]} \| \mathcal{U}^{(j)}(t) \|_{\mathcal{H}}
    \le 2 C_1
    \right\}.
\end{align}
Since
$(u_0^{(j)},u_1^{(j)}) \to
(u_0,u_1)$
in $\mathcal{H}$
as $j\to \infty$,
we have
$\| (u_0^{(j)},u_1^{(j)}) \|_{\mathcal{H}}
\le C_1$
for large
$j$,
which ensures
$\tau_j > 0$
for such $j$.
Moreover, the same computation as in
\eqref{app:1:diffest}
and the Gronwall inequality
imply,
for $t \in [0,\min\{\tau_j, T\}]$,
\begin{align}
\label{app:1:contidepest}
    \| \mathcal{U}^{(j)}(t) - \mathcal{U}(t) \|_{\mathcal{H}}
    \le
    C_0
    \| \mathcal{U}^{(j)}(0) - \mathcal{U}(0) \|_{\mathcal{H}}
    \exp \left( C C_1^{p-1} T \right).
\end{align}
Note that the right-hand side tends to zero
as $j \to \infty$.
From this and the definition of
$C_1$,
we obtain
\begin{align}
    \| \mathcal{U}^{(j)}(t) \|_{\mathcal{H}}
    \le
    C_1 
    \quad (t\in [0,\min\{\tau_j, T\}])
\end{align}
for large $j$.
By the definition of $\tau_j$,
the above estimate implies
$\tau_j > T$,
and hence,
$T_{\mathrm{max}}(u_0^{(j)},u_1^{(j)}) > T$.
From this, the estimate \eqref{app:1:contidepest}
holds for $t \in [0,T]$.
Letting
$j \to \infty$
in \eqref{app:1:contidepest}
gives
\eqref{app:1:contidep}.

\subsubsection{Regularity of solution}\label{sec:regularity}
Next, we discuss the regularity of the solution.
Let
$(u_0, u_1) \in D(\mathcal{A})$
and
$T_{\mathrm{max}} = T_{\mathrm{max}}(u_0,u_1)$.
Then, we will show that the corresponding
mild solution
$\mathcal{U}$
satisfies
\begin{align}
    \mathcal{U} \in
    C([0,T_{\mathrm{max}}) ; D(\mathcal{A}))\cap 
    C^1([0,T_{\mathrm{max}}) ; \mathcal{H}).
\end{align}
Take
$T \in (0,T_{\mathrm{max}})$
arbitrary.
First, from Section \ref{sec:A:1},
the linear part of the mild solution
satisfies
$\mathcal{U}_{L}(t)
=U(t) \begin{pmatrix}
u_0\\ u_1
\end{pmatrix}\in
C([0,\infty); D(\mathcal{A}))
\cap C^1([0,\infty) ; \mathcal{H})$.
This implies,
for $h>0$
and $t \in [0,T-h]$,
\begin{align}\label{app:1:reg:lin}
    \| \mathcal{U}_{L}(t+h) - \mathcal{U}_{L}(t) \|_{\mathcal{H}}
    \le C h.
\end{align}
Thus, it suffices to show
\begin{align}
    \mathcal{U}_{NL}(t)
    &:= \int_0^t U(t-s)
    \begin{pmatrix}
    0\\ -|u(s)|^{p-1}u(s)
    \end{pmatrix}
    \,ds \\
    &\in C([0,T] ; D(\mathcal{A}))\cap 
    C^1([0,T] ; \mathcal{H}).
\label{app:1:reg:nl}
\end{align}
By the changing variable
$t+h-s \mapsto s$,
we calculate
\begin{align}
    \mathcal{U}_{NL}(t+h)
    - \mathcal{U}_{NL}(t) 
    &=
    \int_0^{t+h} U(t-s)
    \begin{pmatrix}
    0\\ -|u(s)|^{p-1}u(s)
    \end{pmatrix}
    \,ds \\
    &\quad
    -
    \int_0^t U(t-s)
    \begin{pmatrix}
    0\\ -|u(s)|^{p-1}u(s)
    \end{pmatrix}
    \,ds \\
    &=
    \int_0^t U(s)
    \begin{pmatrix}
    0\\ -|u|^{p-1}u(t+h-s)
    + |u|^{p-1}u(t-s)
    \end{pmatrix}
    \,ds\\
    &\quad
    + \int_t^{t+h}
    U(s)
    \begin{pmatrix}
    0\\ -|u|^{p-1}u(t+h-s)
    \end{pmatrix}
    \,ds.
\end{align}
Therefore, the same computation as
in \eqref{app:1:diffest} and \eqref{app:1:contraction}
implies
\begin{align}
    \| \mathcal{U}_{NL}(t+h)
    - \mathcal{U}_{NL}(t) \|_{\mathcal{H}}
    &\le
    C \int_0^t
    \| u(s+h) - u(s) \|_{H^1} \,ds
    + Ch.
\end{align}
Combining this with \eqref{app:1:reg:lin},
one obtains
\begin{align}
    \| \mathcal{U}(t+h)
    - \mathcal{U}(t) \|_{\mathcal{H}}
    &\le
    Ch
    + \int_0^t
    \| \mathcal{U}(s+h)
    - \mathcal{U}(s) \|_{\mathcal{H}}
    \,ds.
\end{align}
The Gronwall inequality implies
\begin{align}
    \| \mathcal{U}(t+h)
    - \mathcal{U}(t) \|_{\mathcal{H}}
    &\le Ch.
\end{align}
This further yields
\begin{align}
    \| -|u|^{p-1}u(t+h)
    + |u|^{p-1}u(t) \|_{H^1}
    &\le Ch,
\end{align}
that is, the nonlinearity is
Lipschitz continuous in $H^1_0(\Omega)$.
From this, we can see
$-|u|^{p-1}u \in W^{1,\infty}(0,T;H^1_0(\Omega))$
(see e.g. \cite[Corollary 1.4.41]{CaHa}).
Thus, we can differentiate
the expression
\begin{align}
    \int_0^t U(t-s)
    \begin{pmatrix}
    0\\ -|u|^{p-1}u(s)
    \end{pmatrix}
    \,ds
    =
    \int_0^t U(s)
    \begin{pmatrix}
    0\\ -|u|^{p-1}u(t-s)
    \end{pmatrix}
    \,ds
\end{align}
with respect to $t$
in $\mathcal{H}$,
and it implies
$\mathcal{U}_{NL} \in C^1([0,T]; \mathcal{H})$.
Finally, for
$h > 0$ and $t \in [0,T-h]$,
we have
\begin{align}
    \frac{1}{h}
    \left( U(t) - I \right)
    \mathcal{U}_{NL}(t)
    &=
    \frac{1}{h}
    \int_0^t U(t+h-s)
    \begin{pmatrix}
    0\\ -|u|^{p-1}u(s)
    \end{pmatrix}
    \,ds
    -
    \frac{1}{h}
    \int_0^t U(t-s)
    \begin{pmatrix}
    0\\ -|u|^{p-1}u(s)
    \end{pmatrix}
    \,ds \\
    &=
    \frac{1}{h}
    \left( \mathcal{U}_{NL}(t+h) - \mathcal{U}_{NL}(t) \right) 
    -
    \frac{1}{h}
    \int_t^{t+h} U(t+h-s)
    \begin{pmatrix}
    0\\ -|u|^{p-1}u(s)
    \end{pmatrix}
    \,ds.
\end{align}
This implies
$\mathcal{U}(t) \in D(\mathcal{A})$
and
\begin{align}
    \frac{d}{dt} \mathcal{U}_{NL}(t)
    &= \mathcal{A} \mathcal{U}_{NL}(t)
    +
    \begin{pmatrix}
    0\\ -|u|^{p-1}u(t)
    \end{pmatrix}.
\end{align}
Moreover, the above equation and
$\mathcal{U} \in C^1([0,T]; \mathcal{H})$
lead to 
$\mathcal{U} \in C([0,T]; D(\mathcal{A}))$.
This proves the property 
\eqref{app:1:reg:nl}.
We also remark that
the first component
$u$ of $\mathcal{U}$
is a strong solution to \eqref{dwx}.

\subsubsection{Approximation of the mild solution by strong solutions}\label{sec:approx}
Let
$(u_0, u_1) \in \mathcal{H}$
and
$T_{\mathrm{max}} = T_{\mathrm{max}}(u_0,u_1)$.
Let
$\{ (u_0^{(j)}, u_1^{(j)}) \}_{j=1}^{\infty}$
be a sequence in
$D(\mathcal{A})$
satisfying
$\lim_{j\to \infty}(u_0^{(j)}, u_1^{(j)}) 
= (u_0,u_1)$
in
$\mathcal{H}$.
Take
$T \in (0,T_{\mathrm{max}})$
arbitrary.
Then, the results of Sections
\ref{sec:conti:dep} and \ref{sec:regularity}
imply that
$T_{\mathrm{max}}(u_0^{(j)}, u_1^{(j)}) > T$
for large $j$,
and the corresponding mild solution
$\mathcal{U}^{(j)}
= \begin{pmatrix}
u^{(j)}\\ v^{(j)}
\end{pmatrix}$
with the initial data
$(u_0^{(j)}, u_1^{(j)})$
satisfies
$\mathcal{U}^{(j)} \in C([0,T]; D(\mathcal{A}))
\cap C^1([0,T]; \mathcal{H})$.
Moreover,
$\partial_t u^{(j)} = v^{(j)}$
holds and
$u^{(j)}$
is a strong solution to \eqref{dwx}.
By the result of Section \ref{sec:conti:dep},
we see that 
\begin{align}
    \lim_{j\to \infty}
    \sup_{t\in [0,T]} \| u^{(j)} (t) - u(t) \|_{H^1} = 0,\\
    \lim_{j\to \infty}
    \sup_{t\in [0,T]} \| \partial_t u^{(j)} (t) - v(t) \|_{L^2} = 0,
\end{align}
which yields
$u \in C^1([0,T] ; L^2(\Omega))$
and
$\partial_t u = v$.
Namely, we have the property stated
at the end of Section \ref{sec:app:ex}.

\subsubsection{Finite propagation property}\label{sec:fpp}
Here, we show the finite propagation property
for the mild solution.
In what follows, we use the notations
$B_R(x_0) :=
\{ x \in \mathbb{R}^n ;\, |x-x_0| < R \}$
for $x_0 \in \mathbb{R}^n$
and $R>0$.
Let
$T \in (0,T_{\max}(u_0,u_1))$
and $R>0$.
Assume that
$(u_0, u_1) \in \mathcal{H}$
satisfies
$\mathrm{supp\,}u_0 \cup \mathrm{supp\,}u_1
\subset B_R(0)\cap \Omega$.
Let
$u \in C([0,T]; H^1_0(\Omega)) \cap C^1([0,T]; L^2(\Omega))$
be the mild solution of \eqref{dwx}.
Then, we have
\begin{align}\label{eq:app:fpp}
    \mathrm{supp\,}u(t,\cdot)
    \subset
    B_{t+R}(0) \cap \Omega
    \quad
    (t \in [0,T]).
\end{align}

To prove this,
we modify the argument of \cite{John}
in which the classical solution is treated.
Let
$(t_0, x_0) \in [0,T] \times \Omega$
be a point such that
$|x_0| > t_0 + R$
and define
\begin{align}
    \Lambda(t_0,x_0)
    &=
    \{ (t,x) \in (0,T) \times \Omega ;\, 0< t < t_0, |x-x_0| < t_0- t\} \\
    &=
    \bigcup_{t\in (0,t_0)} \left( \{t \} \times (B_{t_0-t}(x_0)\cap \Omega) \right)).
\end{align}
It suffices to show
$u=0$
in $\Lambda (t_0,x_0)$.
We also put
$S_{t_0-t} := \partial B_{t_0-t}(x_0)\cap\Omega$
and 
$S_{b,t_0-t} := B_{t_0-t}(x_0) \cap \partial \Omega$.
Note that
$\partial (B_{t_0-t}(x_0)\cap \Omega)
= \overline{S_{t_0-t} \cup S_{b,t_0-t}}$
holds.

First, we further assume
$(u_0, u_1) \in D(\mathcal{A})$.
Then, by the result of Section \ref{sec:regularity},
$u$
becomes the strong solution.
This ensures that the following computations
make sense.

Define
\begin{align}
    \mathcal{E}(t;t_0,x_0) &:=
    \frac{1}{2}
    \int_{B_{t_0-t}(x_0)\cap \Omega}
    (|\partial_t u(t,x)|^2
    + |\nabla u(t,x)|^2 + |u(t,x)|^2 )\,dx
\end{align}
for $t \in [0,t_0]$.
By differentiating in $t$
and applying the integration by parts,
we have
\begin{align}
    \frac{d}{dt} \mathcal{E}(t;t_0,x_0) 
    &=
    \int_{B_{t_0-t}(x_0)\cap \Omega}
    \left( \partial_t^2 u - \Delta u + u \right)\partial_t u \,dx\\
    &\quad
    -\frac{1}{2}
    \int_{S_{t_0-t}\cup S_{b,t_0-t}}
    (|\partial_t u|^2 + |\nabla u|^2 + |u|^2
    - 2 (\mathbf{n}\cdot \nabla u) \partial_t u
    ) \,dS,
\end{align}
where
$\mathbf{n}$
is the unit outward normal vector of
$S_{t_0-t}\cup S_{b,t_0-t}$
and
$dS$ denotes the surface measure.
The Schwarz inequality implies
the second term of the right-hand side is
nonpositive, and hence, we can omit it.
Using the equation \eqref{dwx} to the
first term and the Gagliardo--Nirenberg inequality
$\| u(t) \|_{L^{2p}(B_{t_0-t}(x_0)\cap \Omega)}
\le C \| u(t) \|_{H^1(B_{t_0-t}(x_0)\cap \Omega)}$,
we can see that
\begin{align}
    \frac{d}{dt} \mathcal{E}(t;t_0,x_0) 
    &\le
    C \left(
    \| u(t) \|_{H^1(B_{t_0-t}(x_0)\cap \Omega)}^{2p}
    + \| \partial_t u(t) \|_{L^2(B_{t_0-t}(x_0)\cap \Omega)}^2
    + \| u(t) \|_{L^2(B_{t_0-t}(x_0)\cap \Omega)}^2
    \right) \\
    &\le
    C \mathcal{E}(t;t_0,x_0),
\end{align}
where we have also used
$\| u(t) \|_{H^1(B_{t_0-t}(x_0)\cap \Omega)}$
is bounded for $t \in (0,t_0)$.
Noting that the support condition of the initial data
implies
$\mathcal{E}(0;t_0,x_0) = 0$,
we obtain from the above inequality that
$\mathcal{E}(t;t_0,x_0) = 0$
for $t \in [0,t_0]$.
This yields
$u = 0$
in
$\Lambda (t_0,x_0)$.

Finally, for the general case
$(u_0, u_1) \in \mathcal{H}$,
we take an arbitrary small
$\varepsilon > 0$
and
a sequence
$\{ (u_0^{(j)}, u_1^{(j)})\}_{j=1}^{\infty}$
in $D(\mathcal{A})$
such that
$\mathrm{supp\,}u_0^{(j)} \cup
\mathrm{supp\,}u_1^{(j)}
\subset B_{R+\varepsilon}(0) \cap \Omega$
and
$\lim_{j\to \infty} (u_0^{(j)}, u_1^{(j)})
= (u_0,u_1)$
in
$\mathcal{H}$.
Here, we remark that such a sequence
can be constructed by the form
$(u_0^{(j)}, u_1^{(j)}) =
( \phi_{\varepsilon} \tilde{u}_0^{(j)}, \phi_{\varepsilon} \tilde{u}_1^{(j)})$,
where
$\{ (\tilde{u}_0^{(j)}, \tilde{u}_1^{(j)}) \}$
is a sequence in
$D(\mathcal{A})$
which converges to
$(u_0,u_1)$
in $\mathcal{H}$ as $j\to \infty$,
and
$\phi_{\varepsilon} \in C_0^{\infty}(\mathbb{R}^n)$
is a cut-off function satisfy
$0 \le \phi_{\varepsilon} \le 1$,
$\phi_{\varepsilon} = 1$ on $B_{R}(0)$,
and
$\phi_{\varepsilon} = 0$ on
$\mathbb{R}^n \setminus B_{R+\varepsilon}(0)$.
Then, the result of Section \ref{sec:regularity}
shows that the corresponding strong solution
$u^{(j)}$ to
$(u_0^{(j)}, u_1^{(j)})$
satisfies
$\mathrm{supp\,}u^{(j)}(t,\cdot) \subset
B_{R+\varepsilon+t}(0)$.
Moreover, the result of Section \ref{sec:approx}
leads to
$\lim_{j\to \infty} u^{(j)} = u$
in
$C([0,T]; H^1_0(\Omega))$.
Hence, we conclude
$\mathrm{supp\,} u(t,\cdot) \subset
B_{R+\varepsilon+t}(0)$.
Since
$\varepsilon$
is arbitrary, we have \eqref{eq:app:fpp}.

\subsubsection{Existence of the global solution}
Finally, we show the existence of
the global solution to \eqref{dwx}.
Let
$(u_0, u_1) \in \mathcal{H}$
and suppose that
$T_{\max}(u_0,u_1)$
is finite.
Then, by the blow-up alternative (Section \ref{sec:bu}),
the corresponding mild solution
$u$
must satisfy
\begin{align}\label{eq:app:bu}
    \lim_{t \to T_{\max}-0}
    \| (u(t), \partial_t u(t)) \|_{\mathcal{H}} = \infty.
\end{align}
Let
$\{ (u_0^{(j)}, u_1^{(j)} ) \}_{j=1}^{\infty}$
be a sequence in
$D(\mathcal{A})$
such that
$\lim_{j\to \infty} (u_0^{(j)},u_1^{(j)}) = (u_0,u_1)$
in
$\mathcal{H}$,
and let
$u^{(j)}$
be the corresponding strong solution
with the initial data
$(u_0^{(j)}, u_1^{(j)})$.

Using the integration by parts and the equation \eqref{dwx},
we calculate
\begin{align}
    \frac{d}{dt}
    \left[
       \frac{1}{2}\left(
       \| \partial_t u^{(j)}(t) \|_{L^2}^2
        + \| \nabla u^{(j)}(t) \|_{L^2}^2
        \right)
        +
        \frac{1}{p+1}
        \| u^{(j)}(t) \|_{L^{p+1}}^{p+1}
    \right]
    &=
    - \| \partial_t u^{(j)}(t) \|_{L^2}^2.
\end{align}
This and the Gagliardo--Nirenberg inequality imply
\begin{align}
    \| \partial_t u^{(j)}(t) \|_{L^2}^2
    + \| \nabla u^{(j)}(t) \|_{L^2}^2
    &\le
    C \left( 
    \| u^{(j)}_1 \|_{L^2}^2
        + \| \nabla u^{(j)}_0\|_{L^2}^2
        + \| u^{(j)}_0 \|_{H^1}^{p+1}
    \right).
\end{align}
Moreover, by 
\begin{align}
    u(t) = u_0 + \int_0^t \partial_t u(s) \,ds,
\end{align}
one obtains the bound
\begin{align}\label{eq:app:ge}
    \| (u^{(j)}(t), \partial_t u^{(j)}(t)) \|_{\mathcal{H}}^2
    &\le
    C (1+T)^2 \left( 
    \| u^{(j)}_1 \|_{L^2}^2
        + \| \nabla u^{(j)}_0\|_{L^2}^2
        + \| u^{(j)}_0 \|_{H^1}^{p+1}
    \right)
    \quad
\end{align}
for $t \in [0,T]$.
This and the blow-up alternative
(Section \ref{sec:bu})
show
$T_{\max}(u_0^{(j)}, u_1^{(j)}) = \infty$
for all $j$.
The bound \eqref{eq:app:ge} with
$T = T_{\max}(u_0,u_1)$
also yields that
\begin{align}\label{eq:app:ge2}
    \sup_{j \in \mathbb{N}} \sup_{t \in [0,T_{\max}(u_0,u_1)]}
    \| (u^{(j)}(t), \partial_t u^{(j)}(t)) \|_{\mathcal{H}}^2
    < \infty.
\end{align}
On the other hand, from the result of
Section \ref{sec:approx},
we have
\begin{align}\label{eq:app:ge3}
    \lim_{j\to \infty}
    \sup_{t \in [0,T]}
    \| (u^{(j)}(t)- u(t),
    \partial_t u^{(j)}(t) - \partial_t u(t)) \|_{\mathcal{H}} = 0
\end{align}
for any
$T \in (0,T_{\max}(u_0,u_1))$.
However,
\eqref{eq:app:ge2} and \eqref{eq:app:ge3}
contradict \eqref{eq:app:bu}.
Thus, we conclude
$T_{\max}(u_0,u_1) = \infty$.


\section{Proof of Preliminary lemmas}

\subsection{Proof of Lemma \ref{lem:A:ep}}

\begin{proof}[Proof of Lemma \ref{lem:A:ep}]
We define
\begin{align}
    b_1(x) = \Delta \left( \frac{a_0}{(n-\alpha)(2-\alpha)} \langle x \rangle^{2-\alpha} \right)
    = a_0 \langle x \rangle^{-\alpha} + \frac{a_0 \alpha}{n-\alpha} \langle x \rangle^{-\alpha-2}
\end{align}
and
$b_2(x) = a(x) - b_1(x)$.
By
\begin{align}
    \frac{b_2(x)}{a(x)}
    = \frac{1}{\langle x \rangle^{\alpha}a(x)}
    \left(
        \langle x \rangle^{\alpha} a(x) - a_0 - \frac{a_0 \alpha}{n-\alpha} \langle x \rangle^{-2} 
    \right)
\end{align}
and the assumption \eqref{assum:a},
there exists a constant
$R_{\varepsilon} > 0$
such that
$|b_2(x)| \le \varepsilon a(x)$
holds for
$|x| > R_{\varepsilon}$.
Let
$\eta_{\varepsilon} \in C_0^{\infty}(\mathbb{R}^n)$
satisfy
$0 \le \eta_{\varepsilon}(x) \le 1$
for $x \in \mathbb{R}^n$
and
$\eta_{\varepsilon} (x) = 1$
for $|x| < R_{\varepsilon}$.
Let
$N(x)$
denote the Newton potential, that is,
\begin{align}
    N(x) =
    \begin{dcases}
    \frac{|x|}{2} &(n=1),\\
    \frac{1}{2\pi} \log \frac{1}{|x|} &(n=2),\\
    \frac{\Gamma(n/2+1)}{n(n-2)\pi^{n/2}} |x|^{2-n} &(n\ge 3).
    \end{dcases}
\end{align}
We define
\begin{align}
    A_{\varepsilon}(x) =
    A_0 
    + \frac{a_0}{(n-\alpha)(2-\alpha)} \langle x \rangle^{2-\alpha}
    - N \ast \left( \eta_{\varepsilon} b_2 
    \right),
\end{align}
where
$A_0 > 0$
is a sufficiently large constant determined later.
We show that the above
$A_{\varepsilon}(x)$
has the desired properties.
First, we compute
\begin{align}
    \Delta A_{\varepsilon}(x)
    &= b_1(x) + \eta_{\varepsilon}(x)b_2(x)
    = a(x) - (1-\eta_{\varepsilon})b_2(x),
\end{align}
which implies \eqref{A1}.
Next, since
$\eta_{\varepsilon}b_2$
has the compact support,
$N\ast (\eta_{\varepsilon}b_2)$
satisfies
\begin{align}
    |N\ast (\eta_{\varepsilon}b_2) (x)|
    &\le
    C \begin{cases}
        1+ \log \langle x \rangle &(n=2)\\
        \langle x \rangle^{2-n} &(n=1, n\ge 3)
    \end{cases},
    \quad
    |\nabla N\ast (\eta_{\varepsilon}b_2) (x)|
    \le C \langle x \rangle^{1-n}
\end{align}
with some constant
$C = C(n,R_{\varepsilon},\| a \|_{L^{\infty}}, \alpha, a_0, \varepsilon) > 0$,
and the former estimate leads to \eqref{A2},
provided that
$A_0$
is sufficiently large.
Moreover, the latter estimate shows
\begin{align}
    \lim_{|x|\to \infty}
    \frac{|\nabla A_{\varepsilon}(x)|^2}{a(x) A_{\varepsilon}(x)}
    &=
    \lim_{|x|\to \infty}
    \frac{1}{\langle x \rangle^{\alpha} a(x)}
    \cdot
    \frac{1}{\langle x \rangle^{\alpha-2} A_{\varepsilon}(x)}
    \left| \frac{a_0}{n-\alpha} \langle x \rangle^{-1} x - \langle x \rangle^{\alpha-1} \nabla N\ast(\eta_{\varepsilon}b_2) \right|^2 \\
    &= \frac{2-\alpha}{n-\alpha},
\end{align}
which implies the inequality \eqref{A3}
for sufficiently large
$x$.
Finally, taking
$A_0$
sufficiently large, we have
\eqref{A3} for any $x \in \mathbb{R}^n$.
\end{proof}
\subsection{Properties of Kummer's function}
To prove Lemma \ref{lem:phi.beta},
we prepare some properties of
Kummer's function.
\begin{lemma}\label{lem:kummer}
Kummer's confluent hypergeometric function
$M(b,c;s)$
satisfies the properties listed as follows.
\begin{itemize}
    \item[(i)] $M(b,c;s)$
    satisfies Kummer's equation
    \begin{align}
        su''(s) + (c-s) u'(s) -bu(s) = 0.
    \end{align}
    \item[(ii)]
    If $c \ge b > 0$, then $M(b,c;s) > 0$ for $s \ge 0$
    and
    \begin{align}\label{eq:M:infty}
        \lim_{s\to \infty} \frac{M(b,c;s)}{s^{b-c}e^s} = \frac{\Gamma(c)}{\Gamma(b)}.
    \end{align}
    In particular,
    $M(b,c;s)$
    satisfies
    \begin{align}\label{eq:M:bdd}
        C (1+s)^{b-c} e^s
        \le M(b,c;s)
        \le C' (1+s)^{b-c} e^s
    \end{align}
    with some positive constants
    $C=C(b,c)$
    and
    $C'=C(b,c)'$.
    \item[(iii)]
    More generally,
    if $-c \notin \mathbb{N}\cup \{0\}$
    and
    $c \ge b$,
    then, while the sign of $M(b,c;s)$
    is indefinite,
    it still has the asymptotic behavior
    \begin{align}\label{eq:M:infty:2}
        \lim_{s\to \infty} \frac{M(b,c;s)}{s^{b-c}e^s} = \frac{\Gamma(c)}{\Gamma(b)},
    \end{align}
    where we interpret that
    the right-hand side is zero if $-b \in \mathbb{N}\cup \{0\}$.
    In particular, $M(b,c;s)$
    has a bound
    \begin{align}\label{eq:M:bdd:gene:bc}
        |M(b,c;s)| \le C (1+s)^{b-c} e^s
    \end{align}
    with some positive constant $C=C(b,c)$.
    \item[(iv)]
    $M(b,c;s)$
    satisfies the relations
    \begin{align}
        sM(b,c;s)
        &= sM'(b,c;s) + (c-b)M(b,c;s) - (c-b)M(b-1,c;s),\\
        cM'(b,c;s)
        &= cM(b,c;s) - (c-b)M(b,c+1;s).
    \end{align}
\end{itemize}
\end{lemma}
\begin{proof}
The property (i) is directly obtained from the
definition of $M(b,c;s)$.
When $c=b>0$, (ii) is obvious from
$M(b,b;s)=e^s$.
When $c>b>0$,
we have the integral representation
(see \cite[(6.1.3)]{BeWo10})
\begin{align}
    M(b,c;s)
    = \frac{\Gamma(c)}{\Gamma(b)\Gamma(c-b)}
    \int_0^1 t^{b-1} (1-t)^{c-b-1} e^{ts} \,dt,
\end{align}
which implies
$M(b,c;s) > 0$.
Moreover, \cite[(6.1.8)]{BeWo10} shows
the asymptotic behavior \eqref{eq:M:infty}.
The estimate \eqref{eq:M:bdd} is obvious,
since the right-hand side of \eqref{eq:M:infty} is positive
and $M(b,c;s) > 0$ for $s\ge 0$.
Next, the property (iii)
clearly holds if
$c=b$ or $-b \in \mathbb{N}\cup \{0 \}$,
since
$M(b,c;s)$
is a polynomial of order $-b$ if $-b\in \mathbb{N}\cup \{ 0\}$.
For the cases
$c > b$ and $-b \notin \mathbb{N}\cup \{0\}$,
note that for any $m \in \mathbb{N}\cup \{0\}$
we have
\begin{align}
    \frac{d^m}{ds^m}M (b,c;s) = \frac{(b)_m}{(c)_m} M(b+m,c+m;s),
\end{align}
which implies
$|\frac{d^m}{ds^m}M(b,c;s)| \to \infty$ as $s\to \infty$.
By taking $m \in \mathbb{N} \cup \{0\}$ so that
$b+m > 0$
and applying l'H\^{o}pital theorem we deduce
\begin{align}
    \lim_{s\to \infty} \frac{M(b,c;s)}{s^{b-c} e^s}
    &= \lim_{s\to \infty} \frac{\frac{d^m}{ds^m} M(b,c;s)}{\frac{d^m}{ds^m} (s^{b-c}e^s)}
    = \frac{(b)_m}{(c)_m} \lim_{s\to \infty}
    \frac{M(b+m,c+m;s)}{s^{b-c}e^s+o(s^{b-c}e^s)} \\
    &= \frac{(b)_m \Gamma(c+m)}{(c)_m \Gamma(b+m)}
    = \frac{\Gamma(c)}{\Gamma(b)}.
\end{align}
The estimate \eqref{eq:M:bdd:gene:bc} is
easily follows from the asymptotic behavior
\eqref{eq:M:infty:2} and we have (iii).
Finally, the property (iv) can be found in \cite[p.200]{BeWo10}.
\end{proof}
\subsection{Proof of Lemma \ref{lem:phi.beta}}

\begin{proof}[Proof of Lemma \ref{lem:phi.beta}]
The property (i) is directly follows from
Lemma \ref{lem:kummer} (i).
For (ii), noting that
$0\le \beta < \gamma_{\varepsilon}$
and applying Lemma \ref{lem:kummer} (ii) with
$b=\gamma_{\varepsilon}-\beta$ and $c=\gamma_{\varepsilon}$,
we have
$\varphi_{\beta}(s) > 0$ for $s\ge 0$ and
\begin{align}
    \lim_{s\to \infty} s^{\beta} \varphi_{\beta,\varepsilon}(s)
= \frac{\Gamma(\gamma_{\varepsilon})}{\Gamma(\gamma_{\varepsilon}-\beta)}.
\end{align}
This proves the property (ii).
Next, by Lemma \ref{lem:kummer} (iii) with
$b=\gamma_{\varepsilon} - \beta$ and $c = \gamma_{\varepsilon}$,
one still obtains
$\lim_{s\to \infty} s^{\beta} \varphi_{\varepsilon}(s)
= \Gamma(\gamma_{\varepsilon})/\Gamma(\gamma_{\varepsilon}-\beta)$,
where the right-hand side is interpreted as zero
if $\beta - \gamma_{\varepsilon} \in \mathbb{N}\cup \{0\}$.
In particular, this (or the estimate \eqref{eq:M:bdd:gene:bc}) gives
\begin{align}
    | \varphi_{\beta,\varepsilon}(s)| \le K_{\beta, \varepsilon} (1+s)^{-\beta}
\end{align}
with some constant $K_{\beta,\varepsilon}>0$.
Thus, we have (iii).
Noting that
\begin{align}\label{eq:varphi'}
    \varphi_{\beta,\varepsilon}'(s)
    = e^{-s} \left[
    - M(\gamma_{\varepsilon}-\beta, \gamma_{\varepsilon};s)
    + M'(\gamma_{\varepsilon}-\beta, \gamma_{\varepsilon};s)
    \right]
\end{align}
and applying the first assertion of Lemma \ref{lem:kummer} (iv),
we have the property (iv).
Finally, from \eqref{eq:varphi'} and the second assertion of Lemma \ref{lem:kummer} (iv), we obtain
\begin{align}
    \gamma_{\varepsilon} \varphi_{\beta,\varepsilon}'(s)
    = - \beta e^{-s} M(\gamma_{\varepsilon}-\beta, \gamma_{\varepsilon} + 1; s).
\end{align}
Differentiating again the above identity gives
\begin{align}
    \gamma_{\varepsilon} \varphi_{\beta,\varepsilon}''(s)
    = - \beta e^{-s}
     \left[
    - M(\gamma_{\varepsilon}-\beta, \gamma_{\varepsilon}+1;s)
    + M'(\gamma_{\varepsilon}-\beta, \gamma_{\varepsilon}+1;s)
    \right].
\end{align}
Therefore, the second assertion of Lemma \ref{lem:kummer} (iv) implies
\begin{align}
   \gamma_{\varepsilon} (\gamma_{\varepsilon}+1)
   \varphi_{\beta,\varepsilon}''(s)
   = \beta (\beta+1) e^{-s} M(\gamma_{\varepsilon}-\beta, \gamma_{\varepsilon}+2;s).
\end{align}
In particular, if
$0< \beta < \gamma_{\varepsilon}$,
then Lemma \ref{lem:kummer} (ii)
shows that
$M(\gamma_{\varepsilon}-\beta, \gamma_{\varepsilon}+1;s)$
(resp.
$M(\gamma_{\varepsilon}-\beta, \gamma_{\varepsilon}+2;s)$
)
is bounded from above and below by
$(1+s)^{-\beta-1}e^s$
(resp. $(1+s)^{-\beta-2}e^s$),
and hence,
we have the assertions of (v).
\end{proof}
\subsection{Proof of Proposition \ref{prop:super-sol}}
We are now in a position to prove
Proposition \ref{prop:super-sol}.
\begin{proof}[Proof of Proposition \ref{prop:super-sol}]
Let
$z= \widetilde{\gamma}_{\varepsilon}A_{\varepsilon}(x)/(t_0+t)$.
From Definition \ref{phi.beta.ep} and Lemma \ref{lem:phi.beta} (iv),
one obtains
\begin{align}
    \partial_t \Phi_{\beta,\varepsilon}(t,x;t_0)
    &= - (t_0+t)^{-\beta-1}
    \left[
    \beta \varphi_{\beta,\varepsilon}(z)
    + z \varphi_{\beta,\varepsilon}'(z)
    \right] \\
    &= - (t_0+t)^{-\beta-1} \beta \varphi_{\beta+1,\varepsilon}(z) \\
    &= - \beta \Phi_{\beta+1,\varepsilon} (t,x;t_0),
\end{align}
which proves (i).
Applying Lemma \ref{lem:phi.beta} (iii), we have
\begin{align}
    | \Phi_{\beta,\varepsilon} (t,x;t_0) |
    &\le K_{\beta,\varepsilon} (t_0+t)^{-\beta}
    \left( 1+ \frac{\widetilde{\gamma}_{\varepsilon}A_{\varepsilon}(x)}{t_0+t} \right)^{-\beta} \\
    &\le C
    \left( t_0 + t + A_{\varepsilon}(x) \right)^{-\beta} \\
    &= C
    \Psi(t,x;t_0)^{-\beta}
\end{align}
with some constant
$C = C(n,\alpha,\beta,\varepsilon) > 0$.
This implies (ii).
Next, by Lemma \ref{lem:phi.beta} (ii),
$\Phi_{\beta,\varepsilon}(t,x;t_0)$
satisfies
\begin{align}
    \Phi_{\beta,\varepsilon}(t,x;t_0)
    &\ge k_{\beta,\varepsilon} (t_0+t)^{-\beta}
    \left( 1 + \frac{\widetilde{\gamma}_{\varepsilon}A_{\varepsilon}(x)}{t_0+t} \right)^{-\beta} \\
    &\ge c 
    \left( t_0 + t + A_{\varepsilon}(x) \right)^{-\beta} \\
    &= c \Psi(t,x;t_0)^{-\beta}
\end{align}
with some constant
$c = c(n,\alpha,\beta,\varepsilon) > 0$,
and (iii) is verified.
For (iv), we again put
$z = \tilde{\gamma}_{\varepsilon} A_{\varepsilon}(x)/(t_0+t)$
and compute
\begin{align}
    &a(x) \partial_t \Phi_{\beta,\varepsilon} (x,t;t_0) - \Delta \Phi_{\beta,\varepsilon} (x,t;t_0) \\
    &=
    -a(x)(t_0+t)^{-\beta-1} \\
    &\times
    \left( \beta \varphi_{\beta,\varepsilon}(z) + z \varphi_{\beta,\varepsilon}'(z)
    + \tilde{\gamma}_{\varepsilon} \frac{\Delta A_{\varepsilon}(x)}{a(x)} \varphi_{\beta,\varepsilon}'(z) + \tilde{\gamma}_{\varepsilon} \frac{|\nabla A_{\varepsilon}(x)|^2}{a(x)A_{\varepsilon}(x)} z \varphi_{\beta,\varepsilon}''(z) \right).
\end{align}
Using the equation \eqref{eq:varphi}
and the definition \eqref{gammatilde},
we rewrite the right-hand side as
\begin{align}
    &\tilde{\gamma}_{\varepsilon} a(x) (t_0+t)^{-\beta-1}
    \left( 1 - 2\varepsilon - \frac{\Delta A_{\varepsilon}(x)}{a(x)} \right) \varphi_{\beta,\varepsilon}'(z) \\
    &\quad + a(x)(t_0+t)^{-\beta-1}
    \left( 1 - \tilde{\gamma}_{\varepsilon} \frac{|\nabla A_{\varepsilon}(x)|^2}{a(x)A_{\varepsilon}(x)} \right) \varphi_{\beta,\varepsilon}''(z).
\end{align}
By \eqref{A1} and \eqref{A3} in Lemma \ref{lem:A:ep},
we have
\begin{align}
    &1-2\varepsilon - \frac{\Delta A_{\varepsilon}(x)}{a(x)}
    \le - \varepsilon,\\
    &1 - \tilde{\gamma}_{\varepsilon} \frac{|\nabla A_{\varepsilon}(x)|^2}{a(x)A_{\varepsilon}(x)}
    \ge
    \varepsilon \left( \frac{2-\alpha}{n-\alpha} + 2\varepsilon \right)^{-1} > 0.
\end{align}
From them and the property (v) of Lemma \ref{lem:phi.beta},
we conclude
\begin{align}
    a(x) \partial_t \Phi_{\beta,\varepsilon} (x,t;t_0) - \Delta \Phi_{\beta,\varepsilon} (x,t;t_0)
    &\ge 
    - \varepsilon \tilde{\gamma}_{\varepsilon} a(x) (t_0+t)^{-\beta-1}  \varphi_{\beta,\varepsilon}' \left( \frac{\tilde{\gamma}_{\varepsilon} A_{\varepsilon}(x)}{t_0+t} \right) \\
    &\ge 
    \varepsilon k_{\beta,\varepsilon} a(x) (t_0+t)^{-\beta-1} \left( 1 + \frac{\tilde{\gamma}_{\varepsilon} A_{\varepsilon}(x)}{t_0+t} \right)^{-\beta-1} \\
    &\ge 
    c a(x)  \left(t_0 + t + A_{\varepsilon}(x) \right)^{-\beta-1} \\
    &=
    c a(x) \Psi(x,t;t_0)^{-\beta-1}
\end{align}
with some constant
$c = c(n,\alpha,\beta,\varepsilon) > 0$,
which completes the proof.
\end{proof}
\subsection{Proof of Lemma \ref{lem:deltaphi}}
\begin{proof}[Proof of Lemma \ref{lem:deltaphi}]
Putting
$v = \Phi^{-1+\delta} u$,
noting
$\nabla u = (1-\delta) \Phi^{-\delta}( \nabla \Phi )v + \Phi^{1-\delta} \nabla v$,
and applying integration by parts imply
\begin{align}
    &\int_{\Omega} |\nabla u|^2 \Phi^{-1+2\delta} \,dx \\
    &=
    \int_{\Omega} |\nabla v|^2 \Phi \,dx
    + 2(1-\delta) \int_{\Omega} v (\nabla v \cdot \nabla \Phi)\,dx
    + (1-\delta)^2 \int_{\Omega} |v|^2 \frac{|\nabla \Phi|^2}{\Phi} \,dx \\
    &=
    \int_{\Omega} |\nabla v|^2 \Phi \,dx
    - (1-\delta) \int_{\Omega} |v|^2 \Delta \Phi \,dx
    + (1-\delta)^2 \int_{\Omega} |v|^2 \frac{|\nabla \Phi|^2}{\Phi} \,dx \\
    &\ge
    -(1-\delta) \int_{\Omega} |u|^2 (\Delta \Phi) \Phi^{-2+2\delta}\,dx
    +(1-\delta)^2 \int_{\Omega} |u|^2 |\nabla \Phi|^2 \Phi^{-3+2\delta} \,dx.
\end{align}
By
$u\Delta u = - |\nabla u|^2 + \Delta (\frac{u^2}{2})$,
integration by parts,
and applying the above estimate,
we have
\begin{align}
    &\int_{\Omega} u\Delta u \Phi^{-1+2\delta} \,dx \\
    &=
    - \int_{\Omega} |\nabla u|^2 \Phi^{-1+2\delta} \,dx
    + \frac{1}{2} \int_{\Omega} |u|^2 \Delta (\Phi^{-1+2\delta})\,dx \\
    &=
     - \int_{\Omega} |\nabla u|^2 \Phi^{-1+2\delta} \,dx
     - \frac{1-2\delta}{2} \int_{\Omega} |u|^2 (\Delta \Phi) \Phi^{-2+2\delta} \,dx \\
     &\quad
     +(1-\delta)(1-2\delta) \int_{\Omega} |u|^2 |\nabla \Phi|^2 \Phi^{-3 + 2\delta} \,dx \\
     &\le
     - \frac{\delta}{1-\delta} \int_{\Omega} |\nabla u|^2 \Phi^{-1+2\delta} \,dx
     + \frac{1-2\delta}{2} \int_{\Omega} |u|^2 (\Delta \Phi) \Phi^{-2 + 2\delta} \,dx.
\end{align}
This completes the proof.
\end{proof}

\section*{Acknowledgements}
This work was supported by JSPS KAKENHI Grant Numbers
JP18H01132
and
JP20K14346.
The author would like to thank
Professor Hideaki Sunagawa
for the helpful comments
to simplify the argument of 
Section A.2.7.
The author is also grateful to
Professors Naoyasu Kita and
Motohiro Sobajima
for the valuable comments and 
discussions about the optimality
of the main results.
Finally, the author thank the anonymous referees
for careful reeding of the manuscript and
their very helpful comments.


\begin{thebibliography}{99}
\bibitem{AlIbKh15}\textsc{L. Aloui, S. Ibrahim, M. Khenissi},
\textit{Energy decay for linear dissipative wave equations in exterior domains},
J.\ Differential Equations \textbf{259} (2015), 2061--2079.

\bibitem{BaLeRa92}\textsc{C. Bardos, G. Lebeau, J. Rauch},
\textit{Sharp sufficient conditions for the observation,
control and stabilization of waves from the boundary},
SIAM Journal on Control and Optimization \textbf{30} (1992), 1024--1065.

\bibitem{BeWo10}
\textsc{R. Beals, R. Wong},
Special functions,
A graduate text. Cambridge Studies in Advanced Mathematics 126,
Cambridge University Press, Cambridge, 2010.

\bibitem{Br}\textsc{H. Brezis},
Functional Analysis, Sobolev Spaces and Partial Differential Equations,
Springer, 2011.

\bibitem{BuJo16}\textsc{N. Burq, R. Joly},
\textit{Exponential decay for the damped wave equation in unbounded domains},
Commun.\ Contemp.\ Math.\ \textbf{18} (2016), no.06, 1650012.


\bibitem{CaHa}\textsc{Th. Cazenave, A. Haraux},
An Introduction to Semilinear Evolution Equations,
Oxford University Press, 1998.

\bibitem{ChHa03}\textsc{R. Chill, A. Haraux},
{\em An optimal estimate for the difference of solutions of two abstract evolution equations},
J. Differential Equations {\bf 193} (2003), 385--395.

\bibitem{DaSh95}
\textsc{W. Dan, Y. Shibata},
\textit{On a local energy decay of solutions of a
dissipative wave equation},
Funkcialaj Ekvacioj \textbf{38} (1995), 545--568.

\bibitem{Da16}\textsc{M. Daoulatli},
\textit{Energy decay rates for solutions of the wave equation with linear damping in exterior domain},
Evol.\ Equ.\ Control Theory \textbf{5} (2016), 37--59.

\bibitem{Fu66}\textsc{H. Fujita},
\textit{On the blowing up of solutions of the Cauchy problem for $u_t=\Delta u+u^{1+\alpha}$},
J. Fac.\ Sci.\ Univ.\ Tokyo Sec.\ I \textbf{13} (1966), 109--124.

\bibitem{GaRa98}\textsc{Th. Gallay, G. Raugel},
\textit{Scaling variables and asymptotic expansions in damped wave equations},
J. Differential Equations {\bf 150} (1998), pp. 42--97.

\bibitem{Ham10}\textsc{M. Hamza},
\textit{Asymptotically self-similar solutions of the damped wave equation},
Nonlinear Analysis \textbf{73} (2010), 2897--2916.

\bibitem{HaKaNa04DIE}\textsc{N. Hayashi, E. I. Kaikina, P. I. Naumkin},
\textit{Damped wave equation with super critical nonlinearities},
Differential Integral Equations \textbf{17} (2004), 637--652.

\bibitem{HaKaNa04JDE}\textsc{N. Hayashi, E. I. Kaikina, P. I. Naumkin},
\textit{Damped wave equation in the subcritical case},
J. Differential Equations \textbf{207} (2004), 161--194.

\bibitem{HaKaNa06}\textsc{N. Hayashi, E. I. Kaikina, P. I. Naumkin},
\textit{Damped wave equation with a critical nonlinearity},
Trans.\ Amer.\ Math.\ Soc.\ \textbf{358} (2006), 1165--1185.

\bibitem{HaNa17JMAA}\textsc{N. Hayashi, P. I. Naumkin},
\textit{Damped wave equation with a critical nonlinearity in higher space dimensions},
J. Math.\ Appl.\ Anal.\ \textbf{446} (2017), 801--822.


\bibitem{HoOg04}\textsc{T. Hosono, T. Ogawa},
\textit{Large time behavior and $L^p$-$L^q$ estimate of
solutions of 2-dimensional nonlinear damped wave equations},
J. Differential Equations {\bf 203} (2004), 82--118.
	
\bibitem{HsLi92}\textsc{L. Hsiao, T.-P. Liu},
\textit{Convergence to nonlinear diffusion waves for solutions of
a system of hyperbolic conservation laws with damping},
Comm.\ Math.\ Phys.\ {\bf 43} (1992), 599--605.

\bibitem{Ikawa}
\textsc{M. Ikawa},
Hyperbolic Partial Differential Equations and Wave Phenomena, Translations of Mathematical Monographs,
American Mathematical Society, 2000.

\bibitem{IkeInWa17}\textsc{M. Ikeda, T. Inui, Y. Wakasugi},
\textit{The Cauchy problem for the nonlinear damped wave equation with slowly decaying data},
NoDEA Nonlinear Differential Equations Appl.\ \textbf{24} (2017), no. 2, Art. 10, 53 pp.

\bibitem{IkeInOkWa19}\textsc{M. Ikeda, T. Inui, M. Okamoto, Y. Wakasugi},
\textit{$L^p$-$L^q$ estimates for the damped wave equation
and the critical exponent for the nonlinear problem with slowly decaying data},
Commun.\ Pure Appl.\ Anal.\ \textbf{18} (2019), 1967--2008.

\bibitem{IkeOg16}
\textsc{M. Ikeda, T. Ogawa},
\textit{Lifespan of solutions to the damped wave equation with a critical nonlinearity},
J. Differential Equations \textbf{261} (2016), 1880--1903.

\bibitem{IkeSo19JMAA}\textsc{M.Ikeda, M. Sobajima},
\textit{Remark on upper bound for lifespan of solutions to semilinear evolution equations in a two-dimensional exterior domain},
J.\ Math.\ Anal.\ Appl.\ \textbf{470} (2019), 318--326.

\bibitem{IkeSo19}\textsc{M. Ikeda, M. Sobajima},
\textit{Sharp upper bound for lifespan of solutions to
some critical semilinear parabolic, dispersive and hyperbolic equations via a test function method},
Nonlinear Anal. \textbf{182} (2019), 57--74.

\bibitem{IkeSo21FE}\textsc{M. Ikeda, M. Sobajima},
\textit{Life-span of blowup solutions to semilinear wave equation
with space-dependent critical damping},
Funkcialaj Ekvacioj \textbf{64} (2021), 137--162.

\bibitem{IkeTaWa}\textsc{M. Ikeda, K. Taniguchi, Y. Wakasugi},
\textit{Global existence and asymptotic behavior for nonlinear damped wave equations on measure spaces}, 
arXiv:2106.10322v2.

\bibitem{IkeWa15}
\textsc{M. Ikeda, Y. Wakasugi},
\textit{A note on the lifespan of solutions to the semilinear damped wave equation},
Proc.\ Amer.\ Math.\ Soc. \textbf{143} (2015),  163--171.

\bibitem{Ik02}\textsc{R. Ikehata},
\textit{Diffusion phenomenon for linear dissipative wave equations in an exterior domain},
J. Differential Equations {\bf 186} (2002), 633--651.

\bibitem{Ik03JDE}\textsc{R. Ikehata},
\textit{Fast decay of solutions for linear wave equations
with dissipation localized near infinity in an exterior domain},
J. Differential Equations \textbf{188} (2003), 390--405.

\bibitem{Ik03JMAA}\textsc{R. Ikehata},
\textit{Critical exponent for semilinear damped wave equations in the $N$-dimensional half space},
J. Math. Anal. Appl. \textbf{288} (2003), 803--818. 

\bibitem{Ik05JMAA1}\textsc{R. Ikehata},
\textit{Two dimensional exterior mixed problem for semilinear damped wave equations},
J. Math. Anal. Appl. \textbf{301} (2005) 366--377.

\bibitem{Ik05IJPAM}\textsc{R. Ikehata},
\textit{Some remarks on the wave equation with potential type damping coefficients},
Int.\ J. Pure Appl.\ Math.\ {\bf 21} (2005), 19--24.

\bibitem{IkNi03}\textsc{R. Ikehata, K. Nishihara},
\textit{Diffusion phenomenon for second order linear evolution equations},
Studia Math.\ {\bf 158} (2003), 153--161.

\bibitem{IkNiZh06}\textsc{R. Ikehata, K. Nishihara, H. Zhao},
\textit{Global asymptotics of solutions to
the Cauchy problem for the damped wave equation with absorption},
J. Differential Equations \textbf{226} (2006), 1--29.

\bibitem{IkOh02}\textsc{R. Ikehata, M. Ohta},
\textit{Critical exponents for semilinear dissipative wave equations in $\mathbf{R}^N$},
J. Math.\ Anal.\ Appl.\ \textbf{269} (2002), 87--97. 

\bibitem{IkTa05}\textsc{R. Ikehata, K. Tanizawa},
{\em Global existence of solutions for semilinear damped wave equations in $\mathbf{R}^N$
with noncompactly supported initial data},
Nonlinear Anal. {\bf 61} (2005), 1189--1208.

\bibitem{IkToYo09}\textsc{R. Ikehata, G. Todorova, B. Yordanov},
\textit{Critical exponent for semilinear wave equations with space-dependent potential},
Funkcialaj Ekvacioj \textbf{52} (2009), 411--435.

\bibitem{IkToYo13}\textsc{R. Ikehata, G. Todorova, B. Yordanov},
\textit{Optimal decay rate of the energy for wave equations with critical potential},
J. Math.\ Soc.\ Japan {\bf 65} (2013), 183--236.
	
\bibitem{John}
\textsc{F. John},
Nonlinear Wave Equations, Formation of Singularities,
Pitcher Lectures in the Mathematical Sciences at Lehigh University, Univ. Lecture Ser., \textbf{2}, Amer. Math. Soc., Providence, RI, 1990.
	
\bibitem{JoRo18}
\textsc{R. Joly, J. Royer},
	\textit{Energy decay and diffusion phenomenon for the asymptotically periodic damped wave equation}
	J.\ Math.\ Soc.\ Japan {\bf 70} (2018), 1375--1418. 
	
\bibitem{Ka00}\textsc{G. Karch},
\textit{Selfsimilar profiles in large time asymptotics of solutions to damped wave equations},
Studia Math.\ {\bf 143} (2000), 175--197.

\bibitem{KaUe13}\textsc{T. Kawakami, Y. Ueda},
\textit{Asymptotic profiles to the solutions for a nonlinear damped wave equation},
Differential Integral Equations \textbf{26} (2013), 781--814.

\bibitem{KaTa16}\textsc{T. Kawakami, H. Takeda},
\textit{Higher order asymptotic expansions to the solutions for a nonlinear damped wave equation},
NoDEA Nonlinear Differential Equations Appl.\ \textbf{23} (2016), no. 5, Art. 54, 30 pp.

\bibitem{KaNaOn95}\textsc{S. Kawashima, M. Nakao, K. Ono},
\textit{On the decay property of solutions to the Cauchy problem of
the semilinear wave equation with a dissipative term},
J. Math.\ Soc.\ Japan \textbf{47} (1995), 617--653.

\bibitem{KaNaSo04}\textsc{M. Kawashita, H. Nakazawa, H. Soga},
\textit{Non decay of the total energy for the wave equation with dissipative
term of spatial anisotropy},
Nagoya Math. J. \textbf{174} (2004), 115--126. 

\bibitem{KiQa02}\textsc{M. Kirane, M. Qafsaoui},
\textit{Fujita's exponent for a semilinear wave equation with linear damping},
Adv.\ Nonlinear Stud.\ \textbf{2} (2002), 41--49. 

\bibitem{LaScTa20}\textsc{N. A. Lai, N. M. Schiavone, H. Takamura},
\textit{Heat-like and wave-like lifespan estimates
for solutions of semilinear damped wave equations via a Kato’s type lemma},
J.\ Differential Equations \textbf{269} (2020), 11575--11620.

\bibitem{LaZh19}\textsc{N. A. Lai, Y. Zhou},
\textit{The sharp lifespan estimate for semilinear damped wave equation
with Fujita critical power in higher dimensions},
J.\ Math.\ Pure Appl.\ \textbf{123} (2019), 229--243.

\bibitem{LiZh95}\textsc{T.-T. Li, Y. Zhou},
\textit{Breakdown of solutions to $\square u+u_t=|u|^{1+\alpha}$},
Discrete Contin. Dynam. Syst. \textbf{1} (1995), 503--520.

\bibitem{Li15}\textsc{X. Li},
\textit{Critical exponent for semilinear wave equation with critical potential},
NoDEA Nonlinear Differential Equations Appl.\ \textbf{20} (2013), 1379--1391.



\bibitem{MaNi03}\textsc{P. Marcati, K. Nishihara},
\textit{The $L^p$-$L^q$ estimates of solutions to
one-dimensional damped wave equations and
their application to the compressible flow through porous media},
J. Differential Equations {\bf 191} (2003), 445--469.

\bibitem{Ma76}\textsc{A. Matsumura},
\textit{On the asymptotic behavior of solutions of semi-linear wave equations},
Publ.\ Res.\ Inst.\ Math.\ Sci.\ {\bf 12} (1976), 169--189.

\bibitem{Ma77}\textsc{A. Matsumura},
\textit{Energy decay of solutions of dissipative wave equations},
Proc.\ Japan Acad., Ser.\ A \textbf{53} (1977), 232--236.


\bibitem{Mat02}\textsc{T. Matsuyama},
\textit{Asymptotic behavior of solutions for the wave
equation with an effective dissipation around
the boundary},
J. Math.\ Anal.\ Appl.\ {\bf 271} (2002), 467--492.

\bibitem{Mi21}\textsc{H. Michihisa},
\textit{$L^2$-asymptotic profiles of solutions to linear damped wave equations},
J. Differential Equations \textbf{296} (2021), 573--592.

\bibitem{Mo76}\textsc{K. Mochizuki},
\textit{Scattering theory for wave equations with dissipative terms},
Publ. Res. Inst. Math. Sci. {\bf 12} (1976), 383--390.

\bibitem{MoNa96}\textsc{K. Mochizuki, H. Nakazawa},
\textit{Energy decay and asymptotic behavior of solutions to
the wave equations with linear dissipation},
Publ.\ RIMS, Kyoto Univ.\ {\bf 32} (1996), 401--414.

\bibitem{Nakao01MathZ}\textsc{M. Nakao},
\textit{Energy decay for the linear and semilinear wave
equations in exterior domains with some localized
dissipations},
Math.\ Z. \textbf{238} (2001), 781--797.

\bibitem{Na04}\textsc{T. Narazaki},
\textit{$L^p$-$L^q$ estimates for damped wave equations
and their applications to semi-linear problem},
J. Math.\ Soc.\ Japan {\bf 56} (2004), 585--626.

\bibitem{NaNi08}\textsc{T. Narazaki, K. Nishihara},
\textit{Asymptotic behavior of solutions for
the damped wave equation with slowly decaying data},
J. Math.\ Anal.\ Appl.\ \textbf{338} (2008), 803--819.

\bibitem{Ni03MathZ}\textsc{K. Nishihara},
\textit{$L^p$-$L^q$ estimates of solutions to the damped wave equation
in 3-dimensional space and their application},
Math.\ Z. {\bf 244} (2003), 631--649.

\bibitem{Ni03Ib}\textsc{K. Nishihara},
\textit{$L^p$-$L^q$ estimates for the 3-D damped wave equation
and their application to the semilinear problem},
Seminar Notes of Math.\ Sci.\ \textbf{6}, Ibaraki Univ., (2003), 69--83.

\bibitem{Ni06}\textsc{K. Nishihara},
\textit{Global asymptotics for the damped wave equation with absorption in higher dimensional space},
J. Math.\ Soc.\ Japan \textbf{58} (2006), 805--836.

\bibitem{Ni10}\textsc{K. Nishihara},
\textit{Decay properties for the damped wave equation with
space dependent potential and absorbed semilinear term},
Comm.\ Partial Differential Equations {\bf 35} (2010), 1402--1418.

\bibitem{NiZh06}\textsc{K. Nishihara, H. Zhao},
\textit{Decay properties of solutions to the Cauchy problem for the damped
wave equation with absorption},
J. Math.\ Anal.\ Appl. \textbf{313} (2006), 598--610.

\bibitem{NiSoWa18}\textsc{K. Nishihara, M. Sobajima, Y. Wakasugi},
\textit{Critical exponent for the semilinear wave equations with a damping increasing in the far field},
NoDEA Nonlinear Differential Equations Appl.\ \textbf{25} (2018), no. 6, Paper No. 55, 32 pp.

\bibitem{Nis09}\textsc{H. Nishiyama},
\textit{Non uniform decay of the total energy of the dissipative wave equation},
Osaka J. Math. \textbf{46} (2009), 461--477.

\bibitem{Nis16}\textsc{H. Nishiyama},
\textit{Remarks on the asymptotic behavior of the solution to damped wave equations},
J. Differential Equations {\bf 261} (2016) 3893--3940.

\bibitem{OgTa09}\textsc{T. Ogawa, H. Takeda},
\textit{Non-existence of weak solutions to nonlinear damped wave equations
in exterior domains},
Nonlinear Anal.\ \textbf{70} (2009) 3696--3701.

\bibitem{On03JMAA}\textsc{K. Ono},
\textit{Decay estimates for dissipative wave equations in exterior domains},
J.\ Math.\ Anal.\ Appl.\ \textbf{286} (2003), 540--562.

	
\bibitem{RaToYo09}
\textsc{P. Radu, G. Todorova, B. Yordanov},
\textit{Higher order energy decay rates for damped wave equations
with variable coefficients},
Discrete Contin.\ Dyn.\ Syst.\ Ser. S.\ {\bf 2} (2009), 609--629.

\bibitem{RaToYo10}\textsc{P. Radu, G. Todorova, B. Yordanov},
\textit{Decay estimates for wave equations with variable coefficients},
Trans. Amer. Math. Soc. {\bf 362} (2010), 2279--2299.

\bibitem{RaToYo11}\textsc{P. Radu, G. Todorova, B. Yordanov},
\textit{Diffusion phenomenon in Hilbert spaces and applications},
J. Differential Equations {\bf 250} (2011), 4200--4218. 

\bibitem{RaToYo16}\textsc{P. Radu, G. Todorova, B. Yordanov},
\textit{The generalized diffusion phenomenon and applications},
SIAM J.\ Math.\ Anal.\ {\bf 48} (2016), 174--203.

\bibitem{RauTa74}\textsc{J. Rauch, M. Taylor},
\textit{Exponential decay of solutions to hyperbolic equations in bounded domains},
Indiana Univ. Math. J. \textbf{24} (1974), 79--86. 

\bibitem{SaWa17}
\textsc{S. Sakata, Y. Wakasugi}
	{\em Movement of time-delayed hot spots in Euclidean space},
	Math.\ Z {\bf 285} (2017), 1007--1040.


\bibitem{So19DIE}\textsc{M. Sobajima},
\textit{Global existence of solutions to semilinear damped wave equation
with slowly decaying initial data in exterior domain},
Differential integral equations \textit{32} (2019), 615--638.
	
\bibitem{So20MA}\textsc{M. Sobajima},
\textit{Higher order asymptotic expansion of solutions to abstract linear hyperbolic equations},
Math.\ Ann.\ \textbf{380} (2021), 1--19.
     
\bibitem{So_pre}\textsc{M. Sobajima},
\textit{On global existence for semilinear wave equations with space-dependent critical damping},
to appear in
J.\ Math.\ Soc.\ Japan, 
arXiv:210606107v1.
	
\bibitem{SoWa16_JDE}\textsc{M. Sobajima, Y. Wakasugi}, 
\textit{Diffusion phenomena for the wave equation with space-dependent damping 
in an exterior domain},  
J.\ Differential Equations \textbf{261} (2016), 5690--5718.

\bibitem{SoWa17_AIMS}\textsc{M. Sobajima, Y. Wakasugi}, 
\textit{Remarks on an elliptic problem arising in weighted energy estimates 
for wave equations with space-dependent damping term in an exterior domain}, 
AIMS Mathematics, \textbf{2} (2017), 1--15. 

\bibitem{SoWa18_ADE}\textsc{M. Sobajima, Y. Wakasugi}, 
\textit{Diffusion phenomena for the wave equation with space-dependent damping 
term growing at infinity}, 
Adv.\ Differential Equations \textbf{23} (2018), 581--614.

\bibitem{SoWa19_CCM}\textsc{M. Sobajima, Y. Wakasugi}, 
\textit{Weighted energy estimates for wave equation with space-dependent damping term for slowly decaying initial data}, 
Commun.\ Contemp.\ Math.\ \textbf{21} (2019), no. 5, 1850035, 30 pp.

\bibitem{SoWa21_JMSJ}
\textsc{M. Sobajima, Y. Wakasugi}, 
\textit{Supersolutions for parabolic equations with unbounded or degenerate diffusion coefficients and their applications to some classes of parabolic and hyperbolic equations},
J.\ Math.\ Soc.\ Japan \textbf{73} (2021), 1091--1128.
	
\bibitem{Strauss}\textsc{W. A. Strauss},
Nonlinear wave equations,
(CBMS Reg.\ Conf.\ Ser.\ Math.\ ) Providence,
RI: Am.\ Math.\ Soc.\ (1989).
	
\bibitem{Ta15}\textsc{H. Takeda},
\textit{Higher-order expansion of solutions for a damped wave equation},
Asymptotic Analysis {\bf 94} (2015), pp. 1--31.	

\bibitem{ToYo01}\textsc{G. Todorova, B. Yordanov},
\textit{Critical exponent for a nonlinear wave equation with damping},
J. Differential Equations \textbf{174} (2001), 464--489.

\bibitem{ToYo07}\textsc{G. Todorova, B. Yordanov}
\textit{Nonlinear dissipative wave equations with potential},
Contemp.\ Math.\ \textbf{426} (2007), 317--337.

\bibitem{ToYo09}\textsc{G. Todorova, B. Yordanov},
\textit{Weighted $L^2$-estimates for dissipative wave equations with
variable coefficients},
J. Differential Equations {\bf 246} (2009), 4497--4518.

\bibitem{Ue16}\textsc{H. Ueda},
\textit{A new example of the dissipative wave equations with the total energy decay},
Hiroshima Math.\ J.\ \textbf{46} (2016), 187--193.

\bibitem{Ue79}\textsc{Y. Uesaka},
\textit{The total energy decay of solutions for the wave equation with a dissipative term},
J. Math. Kyoto Univ. \textbf{20} (1979), 57--65.

\bibitem{Wa14JHDE}
\textsc{Y. Wakasugi},
\textit{On diffusion phenomena for the linear wave equation with space-dependent damping},
J.\ Hyp.\ Diff.\ Eq.\ {\bf 11} (2014), 795--819.

\bibitem{Wi04}\textsc{J. Wirth},
\textit{Solution representations for a wave equation with weak dissipation},
Math. Meth. Appl. Sci. {\bf 27} (2004), 101--124.

\bibitem{Wid}\textsc{J. Wirth},
\textit{Asymptotic properties of solutions to wave equations with time-dependent dissipation},
PhD thesis, TU Bergakademie Freiberg, 2005.

\bibitem{Wi06}\textsc{J. Wirth},
\textit{Wave equations with time-dependent dissipation I. Non-effective dissipation},
J. Differential Equations {\bf 222} (2006), 487--514.

\bibitem{Wi07JDE}\textsc{J. Wirth},
\textit{Wave equations with time-dependent dissipation II. Effective dissipation},
J. Differential Equations {\bf 232} (2007), 74--103.

\bibitem{Wi07ADE}\textsc{J. Wirth},
\textit{Scattering and modified scattering for abstract wave equations
with time-dependent dissipation},
Adv.\ Differential Equations {\bf 12} (2007), 1115--1133.

\bibitem{Ya06}\textsc{T. Yamazaki},
\textit{Asymptotic behavior for abstract wave equations with decaying dissipation},
Adv.\ Differential Equations {\bf 11} (2006), 419--456. 
	
\bibitem{YaMi00}\textsc{H. Yang, A. Milani},
\textit{On the diffusion phenomenon of quasilinear hyperbolic waves},
Bull. Sci. Math. {\bf124} (2000), 415--433.

\bibitem{Zh01}\textsc{Qi S. Zhang},
\textit{A blow-up result for a nonlinear wave equation with damping: the critical case},
C. R. Acad. Sci. Paris S\'{e}r. I Math. \textbf{333} (2001), 109--114.

\bibitem{Zu90}\textsc{E. Zuazua},
\textit{Exponential decay for the semilinear wave equation with locally distributed damping},
Comm. Partial Differential Equations \textbf{15} (1990), 205--235. 


\end{thebibliography}
\end{document}